\documentclass[reqno]{amsart}
\usepackage{etex}
\usepackage{amsmath}
\usepackage{amsfonts}
\usepackage{rotating}
\usepackage{bbm}
\usepackage[all,cmtip]{xy}
\usepackage{amscd}
\usepackage{sseq}
\usepackage{tikz}
\usepackage{subcaption}

\usepackage{mathrsfs}
\usepackage{color}
\usepackage{xcolor}
\usepackage{mathtools}
\usepackage{amssymb}
\usepackage[latin1]{inputenc}
\usepackage{graphicx}
\usepackage{dsfont}
\usepackage{color}
\usepackage{hyperref}
\usepackage{a4wide}

\usepackage{marginnote}

\newtheorem{theorem}{Theorem}[section]

\newtheorem{lemma}[theorem]{Lemma}
\newtheorem{proposition}[theorem]{Proposition}
\newtheorem{corollary}[theorem]{Corollary}

\theoremstyle{definition}
\newtheorem{definition}[theorem]{Definition}

\newtheorem{example}[theorem]{Example}

\theoremstyle{remark}
\newtheorem{remark}[theorem]{Remark}

\numberwithin{equation}{section}

\newcommand*{\longhookrightarrow}{\ensuremath{\lhook\joinrel\relbar\joinrel\rightarrow}}

\newcommand{\ow}{\omega}
\newcommand{\p}{\partial}

\newcommand{\C}{{\mathbb{C}}}
\newcommand{\I}{{\mathcal{I}}}
\newcommand{\R}{{\mathbb{R}}}
\newcommand{\Q}{{\mathbb{Q}}}
\newcommand{\Z}{{\mathbb{Z}}}
\newcommand{\N}{{\mathbb{N}}}

\renewcommand{\epsilon}{\varepsilon}
\renewcommand{\theta}{\vartheta}

\newcommand{\lda}{\lambda}

\makeatletter
\def\blfootnote{\gdef\@thefnmark{}\@footnotetext}
\makeatother

\DeclareMathOperator{\Hess}{Hess}
\DeclareMathOperator{\crit}{crit}
\DeclareMathOperator{\Supp}{Supp}

\DeclareMathOperator{\id}{id}
\DeclareMathOperator{\im}{im}
\DeclareMathOperator{\ind}{ind}

\DeclareMathOperator{\Spec}{\text{Spec}}

\DeclareMathOperator{\Fix}{Fix}
\DeclareMathOperator{\RS}{RS}
\DeclareMathOperator{\CZ}{CZ}

\DeclareMathOperator{\Sp}{Sp}

\DeclareMathOperator{\sh}{shift}

\usepackage{graphicx}

\begin{document}
\title[Equivariant wrapped Floer homology and symmetric periodic Reeb orbits]{Equivariant wrapped Floer homology and symmetric periodic Reeb orbits}
\author{Joontae Kim, Seongchan Kim and Myeonggi Kwon}
 \address{School of Mathematics, Korea Institute for Advanced Study, 85 Hoegiro, Dongdaemun-gu, Seoul 02455, Republic of Korea}
 \email {joontae@kias.re.kr}
 \address{Institut de Math\'
ematiques, Universit\'e de Neuch\^atel, Rue  Emile-Argand 11, 2000 Neuch\^atel, Switzerland}
 \email {seongchan.kim@unine.ch}
  \address{Center for Geometry and Physics, Institute for Basic Science (IBS), Pohang 37673, Korea}
 \email {mkwon@ibs.re.kr}

\maketitle
 
\begin{abstract}
The aim of this article is to apply a Floer theory to study symmetric periodic Reeb orbits. We define positive equivariant wrapped Floer homology using a (anti-)symplectic involution on a Liouville domain and investigate its algebraic properties. By a careful analysis of index iterations, we obtain a non-trivial lower bound on the minimal number of geometrically distinct symmetric periodic Reeb orbits on a certain class of real contact manifolds. This includes  non-degenerate real dynamically convex starshaped hypersurfaces in $\R^{2n}$ which are invariant under complex conjugation. As a result, we give a  partial answer to the Seifert conjecture on brake orbits in the contact setting.   
\end{abstract}

\setcounter{tocdepth}{2}
\setcounter{secnumdepth}{4}

\tableofcontents

\section{Introduction}

 \addtocontents{toc}{\protect\setcounter{tocdepth}{0}}
\subsection*{A conjecture of Seifert.}  Consider a mechanical Hamiltonian system   in $\R^{2n}$ associated with a Hamiltonian  of the form
\begin{equation}\label{eq: Hamil}
E(q,p) = \sum_{i,j}g_{ij}(q)p_ip_j + U(q),
\end{equation}
where the matrix $g_{ij}(q)$ is   symmetric and positive-definite for each $q$   and $U$ is a smooth function of $q$. Note that  this Hamiltonian is invariant under  complex conjugation $\rho_0(q,p)=(q,-p)$ which is an anti-symplectic involution, i.e., $\rho_0$ satisfies $\rho_0^2=\id_{\R^{2n}}$  and $\rho_0^* \omega= -\omega$, where $\omega=\sum_{j=1}^ndq_j\wedge dp_j$ is the standard symplectic form on $\R^{2n}$.

Let $c \in \R$ be a regular value of $E$. A $2T$-periodic orbit $x=(q,p): [0,2T]\to E^{-1}(c)$ of the Hamiltonian vector field of $E$ is called a \textit{brake orbit} if $p(0)=p(T)=0$. See Section~\ref{sec: Ham_Reeb_chords} for the definition of the Hamiltonian vector field. In \cite{Seifert}, Seifert proved that if $U$ is real analytic and  the projection of $E^{-1}(c)$ to the position space $\R^n$ is  bounded and homeomorphic to a closed unit ball, then there exists a brake orbit on $E^{-1}(c)$. Under these assumptions, he then conjectured the following:
$$
\text{  \emph{There exist at least $n$ geometrically distinct brake orbits on $E^{-1}(c)$.} }
$$
We remark that  the lower bound $n$ in the Seifert conjecture is optimal. For example,   the Hamiltonian system
\begin{equation*}
H(z_1, \dots, z_n) = \sum_{j=1}^n \frac{ \pi |z_j|^2}{a_j},   \quad a_i / a_j \notin \Q \quad \text{for\, $i \neq j$}
\end{equation*} 
on $\C^n$ has precisely $n$ geometrically distinct brake orbits on each positive energy level.

We state   some related results on brake orbits. Below, a Hamiltonian $H$ is an arbitrary smooth function on $\R^{2n}$. It is  said to be \textit{even in $p$} if it is invariant under $\rho_0$. 
\begin{itemize}
\item (Rabinowitz, \cite{Rabisym}) If $H$ is even in $p$, $c \in \R$ is a regular value, and $z \cdot \nabla H(z) \neq 0 $ for all $ z \in H^{-1}(c)$, then there exists a brake orbit on $H^{-1}(c)$.
\item (Szulkin, \cite{Szul}) If $H$ satisfies the assumptions of Rabinowitz and $H^{-1}(c)$ is $\sqrt{2}$-pinched, then the Seifert conjecture holds.
\item (Long--Zhang--Zhu, \cite{no}) If $H$ is even in $p$, the regular energy level $H^{-1}(c)$ is strictly convex, and $H$ is also invariant under the antipodal map $(q,p) \mapsto  (-q,-p)$, then there exist at least two brake orbits on $H^{-1}(c)$.
\item (Liu--Zhang, \cite{evenconvex}) Under the assumptions of Long--Zhang--Zhu, the Seifert conjecture holds.
\item (Giamb\`o--Giannoni--Piccione, \cite{Seifertn2}) The Seifert conjecture holds for the case $n=2$.
\item (Frauenfelder--Kang, \cite{UrsJungsoo}) If $H:\R^4 \rightarrow \R$ is even in $p$, and  the regular energy level $H^{-1}(c)$ is starshaped and dynamically convex,  then there exist either two or infinitely many brake orbits, see Remark~\ref{UrsJungssremka}.
\end{itemize}
We also refer the reader to    \cite{brake1, brake2,  Bolotingeneral, brake5, brake4,  brake7, brake12, Liu3, brake11,   brake6,brake10} for works that study the Seifert conjecture.

 \addtocontents{toc}{\protect\setcounter{tocdepth}{0}}
\subsection*{The Seifert conjecture in contact geometry.}  In this article, we study the Seifert conjecture in the contact setting. A \textit{real contact manifold} is a triple $(\Sigma, \alpha, \rho)$, where $\Sigma$ is a (co-oriented) contact manifold with a contact form $\alpha$ and $\rho \in \text{Diff}(\Sigma)$ is an anti-contact involution, meaning~that 
$$
\rho^2 = \id_{\Sigma}, \quad \rho^* \alpha = - \alpha.
$$
The Reeb vector field $R=R_{\alpha}$ is then anti-invariant under the involution $\rho$, i.e.,
$$
\rho^* R = -R.
$$
Assume that the fixed point set $\mathcal{L}= \text{Fix}(\rho)$, which is a Legendrian submanifold of $\Sigma$, is nonempty. The flow $\phi_R^t$ of the Reeb vector field $R$ satisfies 
\begin{equation}\label{eq:invoReeb}
\phi_R^t = \rho \circ \phi_R^{-t} \circ \rho.
\end{equation}
A smooth integral curve $c : [0,T] \rightarrow \Sigma$ of the Reeb vector field $R$ satisfying the boundary condition $c(0), c(T) \in \mathcal{L}$ is called a \textit{Reeb chord}. In view of equation \eqref{eq:invoReeb}, associated to each Reeb chord $c:[0,T] \rightarrow \Sigma$ is a \textit{symmetric periodic Reeb orbit}
$$
c^2(t) = \begin{cases} c(t) & \text{ if $ t \in [0,T]$,} \\ \rho \circ c(2T-t) & \text{ if $ t \in [T,2T]$,} \end{cases}
$$
see Figure \ref{fig:symmintro}.
\begin{figure}[h]
\begin{center}
\begin{tikzpicture}[scale=0.2]
\draw [dashed] (-5,0)--(15,0);

\draw [thick] plot [smooth,tension=0.7] coordinates {(7,0)
(7,0.5)
(7,3)
(9,3)
(10,5)
(8,9)
(6,9)
(5,7)
(6,5)
(5,2)
(3,3)
(2.5,6)
(1,6)
(0,5)
(1,3)
(2,1.3)
(2,0) };
\draw [->,thick] (9.2,7.2)--(8.8,7.9) ;
\draw [->,thick] (8.8,-7.9)--(9.2,-7.2) ;
\draw [thick] plot [smooth,tension=0.7] coordinates {
(2,0)
(2,-1.3)
(1,-3)
(0,-5)
(1,-6)
(2.5,-6)
(3,-3)
(5,-2)
(6,-5)
(5,-7)
(6,-9)
(8,-9)
(10,-5)
(9,-3)
(7,-3)
(7,-0.5)
(7,0)
 } ;
\draw [<->] (13,-2)--(13,2) ;
\node at (14,3) {$\rho$};
\draw [fill] (7,0) circle [radius=0.3];
 \draw [fill] (2,0) circle [radius=0.3];
\node at (15,0) [right]{$\mathcal{L}=\Fix(\rho)$};
\node at (1,8)  {$c(t)$};
\node at (-1,-9)  {$\rho \circ c(2T-t)$};
\node at (0,1) {$c(T)$};
\node at (9,1) {$c(0)$};
\end{tikzpicture}
\end{center}
\caption{A symmetric periodic Reeb orbit}
\label{fig:symmintro}
\end{figure}
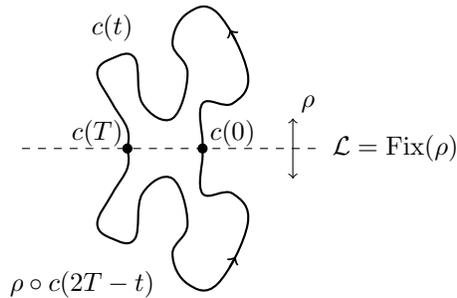
In particular, any Reeb chord comes in a pair and there is a one-to-one correspondence between the set of such pairs and the set of symmetric periodic Reeb orbits.

\begin{example}If the regular energy level $H^{-1}(c)$ is starshaped  (with respect to the origin) which is invariant under complex conjugation $\rho_0$, then the standard Liouville one-form   $\lambda = \tfrac{1}{2}\sum_{j=1}^n(q_jdp_j-p_jdq_j)$ on $\R^{2n}$ restricts to a contact form $\alpha$ on $H^{-1}(c)$. Moreover, the restriction $\rho := \rho_0 |_{H^{-1}(c)}$ defines an anti-contact involution on $H^{-1}(c)$. Consequently,  $(H^{-1}(c),\alpha, \rho)$ is a real contact manifold. The corresponding Legendrian submanifold $\mathcal{L}$ is given by $H^{-1}(c) \cap \left \{ p=0 \right \}$ which implies that  brake orbits are symmetric periodic Reeb orbits on $H^{-1}(c)$. In this case, the existence theorems of Seifert \cite{Seifert} and Rabinowitz \cite{Rabisym} follow immediately from the Arnold chord conjecture proved by Mohnke \cite{mohnke}. 
\end{example}

We translate the Seifert conjecture into the language of contact geometry:
\begin{quote}
\centering
\textit{ There exist at least $n$ geometrically distinct symmetric periodic Reeb orbits on any real contact manifold $(\Sigma^{2n-1}, \alpha, \rho)$.}
\end{quote}
Note that the aforementioned results on brake orbits  \cite{ UrsJungsoo, brake12, Liu3, brake11,  evenconvex,  no, Szul, brake10} study special cases of this conjecture. 

In this article, we apply a Floer theory to this conjecture, taking up  an approach given by Liu and Zhang \cite{Liu3, evenconvex}.  In their theory, they use strict convexity of a hypersurface twice as pointed out in \cite{GuttKang}: First, the Clarke dual action functional, which exists only in the strictly convex case, is used to obtain information on the interval on which the indices of brake orbits lie. Secondly, when the hypersurface is strictly convex, then the index of  brake orbits behaves well under iterations.   Borrowing an idea of Gutt and Kang \cite{GuttKang}, we show that under a weaker assumption  the approach of Liu and Zhang works well in the framework of a Floer  theory and that the index of symmetric periodic Reeb orbits behaves well under iterations, provided that the contact form is non-degenerate. 

Now let $W\subset \R^{2n}$ be a compact starshaped domain with the standard Liouville form $\lambda$. Assume that $W$ admits an exact anti-symplectic involution $\rho_W$ on $W$, i.e., $\rho_W^2=\id_W$ and $\rho^*_W\lambda=-\lambda$. Then the starshaped hypersurface $\Sigma$ given by the boundary of $W$ is a real contact manifold $(\Sigma=\p W,\alpha=\lambda|_{\Sigma},\rho=\rho_W|_\Sigma)$. The contact form $\alpha$ is called \textit{real dynamically convex} if   the Maslov indices of all Reeb chords satisfy suitable lower bounds, see Section~\ref{sec:realdyn}. One of our main results is the following assertion.
\begin{theorem}\label{thm:mianfirst}  Let the triple $(\Sigma,\alpha,\rho)$ be as above. Assume that the contact form $\alpha$ on $\Sigma$ is non-degenerate and real  dynamically convex. Then there exist at least $n  $ geometrically distinct  and simple symmetric periodic Reeb orbits on $\Sigma$.
\end{theorem} 
This theorem in particular  implies that   if a compact starshaped hypersurface in $\R^{2n}$ invariant under  complex conjugation is non-degenerate and real dynamically convex, then the contact Seifert conjecture holds.   
Without the non-degeneracy assumption, Liu and Zhang \cite{evenconvex} prove that if $H^{-1}(c)$ is strictly convex and invariant under   complex conjugation  and the antipodal map, then the contact Seifert conjecture holds. In Theorem~\ref{thm:mianfirst}, we obtain the same result under weaker assumptions,   but under the  additional assumption of non-degeneracy, in order to apply Floer theory.

In Theorem~\ref{theorem:appmain} we prove the same assertion   for a broader class of real contact manifolds   and Theorem~\ref{thm:mianfirst} is obtained as a corollary.  As an immediate corollary of the methods used in the proof of Theorem~\ref{thm:mianfirst} we obtain the following result.

\begin{corollary}   In addition to the assumptions of  Theorem~\ref{thm:mianfirst}, assume that there exist precisely $n$ geometrically distinct and simple symmetric periodic Reeb orbits on $\Sigma$. Then   their  indices  are  all different. 
\end{corollary}

\begin{remark}\label{UrsJungssremka} \rm In \cite[Conjecture 1.1]{Liu3} Liu and Zhang conjecture that if $\Sigma\subset \R^{2n}$ is a strictly convex hypersurface which is invariant under complex conjugation, then there exist either precisely $n$ or infinitely many geometrically distinct symmetric periodic Reeb  orbits. Note that  non-degeneracy is not assumed. The conjecture is proved for the case $n=2$ under a  weaker assumption by Frauenfelder-Kang \cite[Theorem~2.5]{UrsJungsoo} via holomorphic curve techniques.  More precisely,  they proved that if $(\Sigma, \alpha)$ is a starshaped dynamically convex hypersurface in $\C^2$ which is invariant under complex conjugation, then there are either two or infinitely many symmetric periodic Reeb orbits. Recall that convexity implies dynamical convexity, see \cite{HWZ}.  Their theorem says, in particular, that the existence of any nonsymmetric periodic Reeb orbit ensures infinitely many symmetric periodic Reeb orbits. This phenomenon can  also be found in the study of symmeric periodic points of reversible maps. Indeed, any area-preserving map defined on the open unit disk in the complex plane which is reversible with respect to complex conjugation must admit a symmetric periodic point. If there is more than one periodic point, which is  possibly nonsymmetric, then there have to be infinitely many symmetric periodic points, see \cite{Kang2}.
\end{remark}
We now further assume that $W$ admits an exact symplectic involution $\iota_W$, namely, $\iota^2_W=\id_W$ and $\iota^*_W\lambda=\lambda$, which commutes with $\rho_W$. Then its boundary $\Sigma $ carries a contact involution $\iota=\iota_W|_\Sigma$.  For a Reeb chord $c:[0,T] \rightarrow \Sigma$, the corresponding symmetric periodic Reeb  orbit $c^2$ is called  \emph{\text{doubly symmetric}} if $\iota( \text{im}(c^2) ) = \text{im}(c^2)$, see Figure \ref{Fig:symminvintro}.


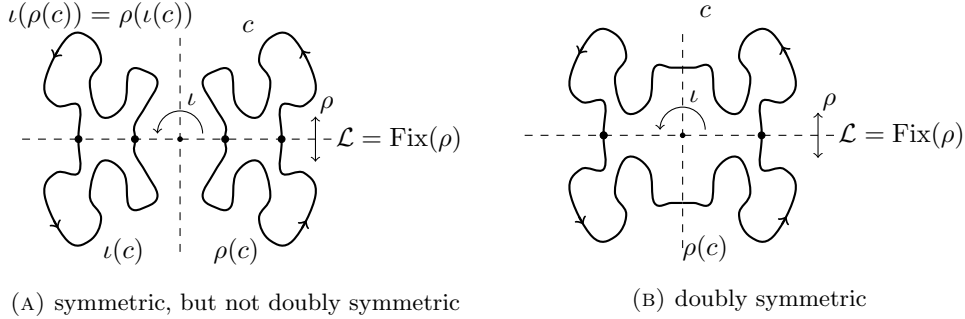
\begin{figure}[h]
\begin{subfigure}{0.45\textwidth}
   \centering
\begin{tikzpicture}[scale=0.15]

 \draw [fill] (4,0) circle [radius=0.2];
\begin{scope}[yscale=1,xscale=1,xshift=6cm]
 \draw [thick] plot [smooth,tension=0.7] coordinates {(7,0)
(7,0.5)
(7,3)
(9,3)
(10,5)
(8,9)
(6,9)
(5,7)
(6,5)
(5,2)
(3,3)
(2.5,6)
(1,6)
(0,5)
(1,3)
(2,1.3)
(2,0.3)
(2,0) };
\draw [->,thick] (9.2,7.2)--(8.8,7.9) ; 
\draw [fill] (7,0) circle [radius=0.3];
 \draw [fill] (2,0) circle [radius=0.3];
\node at (4,10)  {$c$};
\end{scope}

\begin{scope}[yscale=-1,xscale=1,xshift=6cm]
 \draw [thick] plot [smooth,tension=0.7] coordinates {(7,0)
(7,0.5)
(7,3)
(9,3)
(10,5)
(8,9)
(6,9)
(5,7)
(6,5)
(5,2)
(3,3)
(2.5,6)
(1,6)
(0,5)
(1,3)
(2,1.3)
(2,0.3)
(2,0) };
\draw [<-,thick] (9.2,7.2)--(8.8,7.9) ; 
\node at (3,10)  {$\rho(c)$};
\end{scope}

\begin{scope}[yscale=1,xscale=-1,xshift=-2cm]
 \draw [thick] plot [smooth,tension=0.7] coordinates {(7,0)
(7,0.5)
(7,3)
(9,3)
(10,5)
(8,9)
(6,9)
(5,7)
(6,5)
(5,2)
(3,3)
(2.5,6)
(1,6)
(0,5)
(1,3)
(2,1.3)
(2,0.3)
(2,0) };
\draw [<-,thick] (9.2,7.2)--(8.8,7.9) ; 
\draw [fill] (7,0) circle [radius=0.3];
\draw [fill] (2,0) circle [radius=0.3];
\node at (5,11)  {$\iota (\rho(c))=\rho(\iota(c))$};
\end{scope}

\begin{scope}[yscale=-1,xscale=-1,xshift=-2cm]
 \draw [thick] plot [smooth,tension=0.7] coordinates {(7,0)
(7,0.5)
(7,3)
(9,3)
(10,5)
(8,9)
(6,9)
(5,7)
(6,5)
(5,2)
(3,3)
(2.5,6)
(1,6)
(0,5)
(1,3)
(2,1.3)
(2,0.3)
(2,0) };
\draw [->,thick] (9.2,7.2)--(8.8,7.9) ; 
\node at (3,10)  {$\iota(c)$};
\end{scope}

\draw [<->] (16,-2)--(16,2) ;
\draw [dashed] (4,-10)--(4,9);
\draw [dashed] (-10,0)--(18,0);
\node at (17,3) {$\rho$};
\node at (17,0) [right]{$\mathcal{L}=\Fix(\rho)$};

\draw [->] (6,0.5) arc [x radius = 2cm, y radius = 2cm, start angle = 0, end angle = 180];
\node at (5,3.5) {$\iota$};
\end{tikzpicture}
  \caption{symmetric, but not doubly symmetric}
\end{subfigure}
\begin{subfigure}{0.45\textwidth}
  \centering
\begin{tikzpicture}[scale=0.15]

\begin{scope}[yscale=1,xscale=1]
 \draw [thick] plot [smooth,tension=0.7] coordinates {(7,0)
(7,0.5)
(7,3)
(9,3)
(10,5)
(8,9)
(6,9)
(5,7)
(6,5)
(5,2)
(3,3)
(2.5,6)
(1,6)
(0,6)
};
\draw [->,thick] (9.2,7.2)--(8.8,7.9) ; 
\draw [fill] (7,0) circle [radius=0.3];
\node at (2,11)  {$c$};
\end{scope}

\begin{scope}[yscale=1,xscale=-1]
 \draw [thick] plot [smooth,tension=0.7] coordinates {(7,0)
(7,0.5)
(7,3)
(9,3)
(10,5)
(8,9)
(6,9)
(5,7)
(6,5)
(5,2)
(3,3)
(2.5,6)
(1,6)
(0,6)
};
\draw [<-,thick] (9.2,7.2)--(8.8,7.9) ; 
\draw [fill] (7,0) circle [radius=0.3];
\end{scope}

\begin{scope}[yscale=-1,xscale=-1]
 \draw [thick] plot [smooth,tension=0.7] coordinates {(7,0)
(7,0.5)
(7,3)
(9,3)
(10,5)
(8,9)
(6,9)
(5,7)
(6,5)
(5,2)
(3,3)
(2.5,6)
(1,6)
(0,6)
};
\draw [->,thick] (9.2,7.2)--(8.8,7.9) ; 
\draw [fill] (7,0) circle [radius=0.3];
\end{scope}

\begin{scope}[yscale=-1,xscale=1]
 \draw [thick] plot [smooth,tension=0.7] coordinates {(7,0)
(7,0.5)
(7,3)
(9,3)
(10,5)
(8,9)
(6,9)
(5,7)
(6,5)
(5,2)
(3,3)
(2.5,6)
(1,6)
(0,6)
};
\draw [<-,thick] (9.2,7.2)--(8.8,7.9) ; 
\node at (2,10)  {$\rho(c)$};
\end{scope}

\draw [<->] (12,-2)--(12,2) ;
\draw [dashed] (0,-10)--(0,9);
\draw [dashed] (-14,0)--(14,0);
\node at (13,3) {$\rho$};
\node at (13,0) [right]{$\mathcal{L}=\Fix(\rho)$};
\draw [fill] (0,0) circle [radius=0.2];
\draw [->] (2,0.5) arc [x radius = 2cm, y radius = 2cm, start angle = 0, end angle = 180];
\node at (1,3.5) {$\iota$};

\end{tikzpicture}

  \caption{doubly symmetric}
\end{subfigure}

\caption{An illustration of doubly symmetric periodic Reeb orbits}
\label{Fig:symminvintro}
\end{figure}
Our next result is  a  slight generalization of \cite[Theorem~1.2]{Liu3}.
\begin{theorem}\label{thm:main2s} Let the quadruple $(\Sigma,\alpha,\rho,\iota)$ be as above. Assume that the contact form $\alpha$ on $\Sigma$ is non-degenerate and real  dynamically convex. Then  there exist at least $n + \mathcal{N}(\Sigma)$ geometrically distinct  and  simple symmetric periodic Reeb orbits on $\Sigma$, where $2\mathcal{N}(\Sigma)$ is the number of geometrically distinct and simple symmetric periodic  Reeb orbits   which are not doubly symmetric. 
\end{theorem}


We prove the same assertion for more general real contact manifolds in  Theorem~\ref{thm:appli}. Then Theorem~\ref{thm:main2s} is obtained as a corollary. Again, the proof has an immediate corollary.
 
\begin{corollary} In addition to the assumptions of Theorem~\ref{thm:main2s}, assume that there exist precisely $n$ geometrically distinct and simple symmetric periodic Reeb  orbits on $\Sigma$. Then they are all   doubly  symmetric periodic Reeb orbits and their   indices are all different. 
\end{corollary}

In the following we explain two main ingredients to prove our theorems: \textit{equivariant wrapped Floer homology} and \textit{analysis of index iterations}.
 \addtocontents{toc}{\protect\setcounter{tocdepth}{0}}
\subsection*{Equivariant wrapped Floer homology.} Our main tool is a variant of Lagrangian Floer theory, namely \emph{equivariant wrapped Floer homology}. Let $(W, \lda)$ be a Liouville domain with a Liouville form $\lda$ and   an admissible Lagrangian $L$, meaning that $L$ is exact and intersects the boundary $\p W$ in a Legendrian submanifold, see Section~\ref{sec: setHW} for definitions. This is a basic setup that one considers to define (non-equivariant) wrapped Floer homology,  see for example \cite{AboSei}, \cite{Rit}.

For an equivariant theory, we additionally put a $\Z_2$-symmetry on $W$ given by a symplectic or an anti-symplectic involution $\I: W \rightarrow W$ under which $L$ is invariant. We then have an induced $\Z_2$-action on the space of paths in $W$ with end points on the Lagrangian $L$, see equation \eqref{eq: involutionpathspace}. This allows us to define a $\Z_2$-equivariant version of wrapped Floer homology, which is an open string analogue of the $S^1$-equivariant symplectic homology in \cite{BO17}, \cite{Gutt}. 

We are mainly interested in the case when $L$ is given by the fixed point set of an anti-symplectic involution. In this case, we say that $(W, \lda, \I)$ is a \emph{real Liouville domain}. If the anti-symplectic involution $\I: W \rightarrow W$ is exact, then the contact boundary $(\p W,\alpha,\I)$ with the induced contact form $\alpha$ is a real contact manifold. Examples are starshaped domains in $\R^{2n}$ invariant under complex conjugation $(q, p) \mapsto (q, -p)$.

In Section~\ref{sec: defEWFH}, we define equivariant wrapped Floer homology groups for both symplectic and anti-symplectic involutions. In view of \cite{BO17}, one can develop equivariant Floer theory in two different flavors. This is basically due to the fact that the Borel construction, see Section~\ref{sec: Borel}, admits several geometric structures, for example, it can be seen as a quotient space by an action or as a total space of a fiber bundle, see \cite[Section~1]{BO17}. The construction of the equivariant theory given in this paper takes the second point of view, and this matches the construction in \cite{SS} where an equivariant Lagrangian Floer theory with symplectic involutions is outlined. See also \cite{Hut} for a general theory of Floer homology of families. As pointed out in \cite{BO17} and \cite{SS}, a technical benefit of this choice is that the analysis on Floer trajectories is easier to deal with. We also provide a description of equivariant Morse homology in Section~\ref{sec: eMH}. This hopefully makes many ideas of the constructions in equivariant wrapped Floer homology transparent.

For applications to multiplicity results of symmetric periodic Reeb orbits, we need to observe that the generators of positive equivariant wrapped Floer homology correspond to ($\Z_2$-pairs of) Reeb chords on the contact type boundary. Even though positive equivariant wrapped Floer homology is designed to have this property, it is not directly visible from its definition. To identify generators with Reeb chords, we construct a Morse--Bott spectral sequence in equivariant wrapped Floer homology, which might be of independent interest, in Section~\ref{sec: mbss}. If the contact form on the boundary is chord-non-degenerate, see Section~\ref{sec: genetorsof eHW}, the first page of the Morse--Bott spectral sequence is generated by ($\Z_2$-pairs of) Reeb chords and hence so is the resulting homology, see Corollary~\ref{cor: genwhfpsotive}. A similar technique was used in \cite{Gutt} to show that positive $S^1$-equivariant symplectic homology is generated by periodic Reeb orbits. 

We next compute equivariant wrapped Floer homology for the Liouville domains and the Lagrangians in question to detect sufficiently many generators. For this purpose, we establish algebraic properties of equivariant wrapped Floer homology, notably Leray--Serre type long exact sequences, see Corollary~\ref{cor: lsssfull}. A version of such a long exact sequence in symplectic homology is given in \cite[Theorem~1.2]{BO17}. The idea is that a special choice of Floer data, which we call \emph{periodic family of admissible Floer data} in Section~\ref{sec: periodicfamham}, simplifies the equivariant Floer chain complex, see Lemma~\ref{lem: idenper}, so that we have a $\Z_2$-complex structure on it, see Section~\ref{sec: z2cpx}. A $\Z_2$-complex structure on a chain complex then algebraically produces a Leray--Serre type spectral sequence for the resulting $\Z_2$-equivariant homology, see Lemma~\ref{lem: alglsss}.

 It turns out that if the non-equivariant wrapped Floer homology vanishes, then the positive equivariant wrapped Floer homology has enough generators for our multiplicity results, see Corollary~\ref{cor: sympinvcomp} and Corollary~\ref{cor: antisympcomp}. Note that in Liouville domains with vanishing symplectic homology such as subcritical Stein domains, the wrapped Floer homology of any admissible Lagrangian vanishes, see Remark~\ref{rmk: vanishingSHandHW}. In Section~\ref{sec: vanishing wfh sectioin}, we also prove a vanishing property of wrapped Floer homology under a displaceability condition on contact type boundaries and Lagrangians. A similar vanishing result for the symplectic homology of displaceable contact hypersurfaces can be found in \cite{Kang3}. 


\subsection*{Analysis of index iterations.} The main issue in applying Floer theory to multiplicity questions is to distinguish iterated orbits from simple ones. Even though the  equivariant wrapped Floer homology groups have infinitely many generators, it is not obvious that they are represented by (geometrically distinct) symmetric periodic Reeb orbits. Indeed, iterations of a single symmetric periodic Reeb orbit might represent distinct generators of the homology groups. In order to avoid this, we have to establish a suitable condition on the Maslov index of Reeb chords: First of all, we impose a condition on the contact form on a real contact manifold to satisfy a certain lower bound of the Maslov index. Such a contact form is called \emph{real dynamically convex}, see Definition~\ref{def:realdc}. Our terminology is motivated by the fact that a strictly convex hypersurface in~$\R^{2n}$ which is invariant under complex conjugation is real dynamically convex, see Theorem~\ref{thm:chorddynami}. Assuming real dynamical convexity, we prove by a careful analysis based on the properties of the Robbin-Salamon index and the formula for the H\"ormander index given in \cite[Formula 2.10]{Duister} that the Maslov index is increasing under the iterations, see Theorem~\ref{thm:indexincreasingproperty}.   This together with the common index jump theorem, which is proved by Liu and Zhang \cite[Theorem~1.5]{Liu3}, gives rise to a powerful tool for counting geometrically distinct symmetric periodic Reeb orbits in terms of the homology computations.

\subsection*{Acknowledgement.} The authors  would like to thank Urs Frauenfelder for suggesting an extension of the main theorems to the  displaceable case, Felix Schlenk for reading the draft of this paper and Chungen Liu for sending us his papers. 
The part of this article was written during a visit of the first author at  the Institut de Math\'ematiques at  Neuch\^atel and a visit of the second author at the Department of Mathematics at ETH Z\"urich.
The authors cordially thank Universit\'e de Neuch\^atel and ETH Z\"urich   for their warm hospitality. 
 JK is supported by the Swiss Government Excellence Scholarship and a KIAS Individual Grant MG068002 at Korea Institute for Advanced Study,   SK   by   the grant 200021-181980/1 of the Swiss National Foundation, and 
MK by the Institute for Basic Science (IBS-R003-D1) and the SFB/TRR 191 `Symplectic Structures in Geometry, Algebra and Dynamics' funded by the DFG (Projektnummer 281071066 -- TRR 191).
\addtocontents{toc}{\protect\setcounter{tocdepth}{2}}
\section{Wrapped Floer homology}
\subsection{Recollection of wrapped Floer homology}
We recall the definition of wrapped Floer homology. As we will apply the index iteration theory, the homology is equipped with a $\Z$-grading which comes from the Maslov index. Basically, our conventions are those of \cite{KKL}. We refer to \cite[Section~3]{AboSei} and \cite[Section~4]{Rit} for more detailed descriptions.

\subsubsection{Geometric setup}\label{sec: setHW}
Let $(W^{2n}, \lambda)$ be a Liouville domain. By definition, $W$ is a compact manifold with boundary such that $d\lambda$ is symplectic and the associated Liouville vector field $X_\lambda$, which is defined by the condition $d\lambda(X_\lambda,\cdot)=\lambda$, is pointing outwards along the boundary. The restricted 1-form $\alpha:=\lambda|_{\partial W}$ becomes a contact form on the boundary $\p W$. We denote by $R_\alpha$ the \emph{Reeb vector field} of $\alpha$ on $\p W$ characterized by the conditions
$$
d\alpha(R_\alpha,\cdot)=0\quad\text{and} \quad\alpha(R_\alpha)=1.
$$
Attaching an infinite cone, we obtain the \emph{completion} of $W$ as a noncompact manifold
$$
\widehat{W}:=W\cup ([1,\infty)\times \p W)
$$
equipped with the extended Liouville form
$$
\widehat{\lambda}:=\begin{cases}
\lambda & \text{on $W$},\\
r\alpha & \text{on $[1,\infty)\times \p W$},
\end{cases}
$$
where $r$ stands for the coordinate on $[1, \infty)$. We denote by $\widehat{\ow}=d\widehat{\lambda}$ the symplectic form on $\widehat{W}$. Note that the Liouville vector field $X_{\widehat{\lambda}}$ restricts to $r \p_r$ on $[1,\infty)\times \p W$.  
\begin{definition}
A Lagrangian $L$ in $W$ is called {\it admissible} if it has the following properties:
\begin{itemize}
	\item $L$ is transverse to the boundary $\partial W$ and $\p L=L \cap \partial W$ is a Legendrian, i.e., $\alpha|_{\p L}=0$;
	\item $L$ is exact, i.e., $\lambda|_L$ is an exact 1-form; and
	\item the Liouville vector field of $\lambda$ is tangent to $L$ near its  boundary. 
\end{itemize}	
\end{definition}
\noindent The third condition allows us to complete an admissible Lagrangian $L$ to a noncompact Lagrangian in $\widehat{W}$:
$$
\widehat{L}=L \cup_{\partial L} \left([1, \infty) \times \partial L \right).
$$
\subsubsection{Hamiltonian chords and Reeb chords}\label{sec: Ham_Reeb_chords}
Given a Hamiltonian $H:\widehat{W}\to \R$, the associated \emph{Hamiltonian vector field} $X_H$ is defined by the equation $d\widehat{\lambda}(X_H,\cdot)=dH$. We denote by $\phi^t_H:\widehat{W}\to \widehat{W}$ the \emph{Hamiltonian flow} of $X_H$. Its time-1 map $\phi_H^1$ is called the \emph{Hamiltonian diffeomorphism} of~$H$. An orbit $c: [0,1] \rightarrow \widehat W$ of the Hamiltonian vector field $X_H$ meeting the boundary condition $c(0), c(1) \in \widehat L$ is called a \textit{Hamiltonian chord} of $H$. The Hamiltonian chords naturally correspond to the intersection points of $L\cap \phi_H^1(L)$ via the assignment $c\mapsto c(1)$. We say that a Hamiltonian chord~$c$ is \textit{contractible} if the class $[c] \in \pi_1(\widehat W, \widehat L)$ is trivial. We denote by $\mathcal{P}(H)$ the set of all contractible Hamiltonian chords of $H$. A Hamiltonian chord $c$ is called {\it non-degenerate} if it~satisfies
	$$
	D\phi^1_{H}(T_{c(0)}\widehat{L})\cap T_{c(1)}\widehat{L}=\{0\}.
	$$
	A Hamiltonian $H:\widehat{W}\to \R$ is called {\it non-degenerate} if all Hamiltonian chords of $H$ are non-degenerate, equivalently, $\widehat{L}$ and $\phi_H^1(\widehat{L})$ intersect transversally. 
	
	Let $(\Sigma,\alpha)$ be a contact manifold and let $\mathcal{L}\subset \Sigma$ be a Legendrian. A {\it Reeb chord} of length $T>0$ (with respect to $\mathcal{L}$) is an orbit $c:[0,T]\to \Sigma$ of the Reeb vector field $R_\alpha$ meeting the boundary condition $c(0),c(T)\in \mathcal{L}$. A Reeb chord $c:[0,T]\to \Sigma$ is called \emph{non-degenerate} if it satisfies
	$$
	D\phi_R^T(T_{c(0)}\mathcal{L})\cap T_{c(T)}\mathcal{L}=\{0\},
	$$
	where $\phi^t_R:\Sigma\to \Sigma$ is the \emph{Reeb flow} of $R_\alpha$. We denote by $\Spec(\Sigma,\alpha,\mathcal{L})$ the set of all lengths $T$ of Reeb chords.
\begin{definition}
A time-independent Hamiltonian $H: \widehat W \rightarrow \R$ is called \textit{admissible} if $H<0$ on $W$ and $H$ is of the form $H(r,y) = ar+b$ on the symplectization $[1,\infty)\times \p W$ with $a \not \in \Spec(\partial W, \alpha, \partial L)$, $a>0$ and $b\in \R$. Here $a$ is called the \textit{slope} of $H$.
\end{definition}
\noindent We denote by $\mathcal{H}$ the space of admissible Hamiltonians on $\widehat{W}$. For any admissible Hamiltonian~$H$, its $C^2$-small generic perturbation on $W$ is an admissible Hamiltonian which is non-degenerate, see \cite[Lemma~8.1]{AboSei}. All Hamiltonian chords are contained in the compact domain $W$ as $a\notin\Spec(\p W,\alpha,\p L)$.

\subsubsection{Maslov index of Hamiltonian chords} In order to define the Maslov index of Hamiltonian chords, we additionally assume that the Maslov class $\mu_{\widehat{L}}:\pi_2(\widehat{W},\widehat{L})\to \Z$ of the Lagrangian $\widehat{L}$ vanishes. For example, this assumption is satisfied if $c_1(W)$ vanishes on $\pi_2(W)$ and $\pi_1(L)=0$.

Let $D^+=\{z\in \mathbb{C} \mid  |z|\le 1,\ \im z\ge 0\}$ be the half-disk with real part $D^+_\R:=D^+\cap\R$. A Hamiltonian chord $c$ is contractible if and only if it admits a \emph{capping half-disk}, namely, a map $v:(D^+,D^+_\R)\to (\widehat{W},\widehat{L})$ satisfying $v(e^{\pi it})=c(t)$ for $t\in [0,1]$. A capping half-disk $v:D^+\to \widehat{W}$ of~$c$ yields a  symplectic trivialization of $v^*T\widehat{W}$:
$$
\mathfrak{T}_v:(v^*T\widehat{W},\widehat{\ow})\longrightarrow D^+\times (\mathbb{C}^n,\ow_0)
$$
such that $\mathfrak{T}_v(T_{v(z)}\widehat{L})=\Lambda_0$ for $z\in D^+_\R$ where $\Lambda_0=\{z\in \C^n \mid \im z=0\}$ is the horizontal Lagrangian subspace in $(\C^n,\ow_0:=\frac{i}{2}\sum_jdz_j\wedge d\bar{z}_j)$. Such a trivialization $\mathfrak{T}_v$ is called an {\it adapted symplectic trivialization} of $v^*T\widehat{W}$. We write
$$
\Psi_c:[0,1]\longrightarrow \Sp(2n),\quad \Psi_c(t):=\mathfrak{T}_v(c(t))\circ D\phi_H^t(c(0))\circ \mathfrak{T}_v^{-1}(c(0))
$$
for the linearization of the Hamiltonian flow along $c$ with respect to the trivialization $\mathfrak{T}_v$.
\begin{definition}\label{defmaslov}
The {\it Maslov index} of a contractible Hamiltonian chord $c:[0,1]\to \widehat{W}$ is 
	$$
	\mu(c):=\mu_{\RS}(\Psi_c\Lambda_0,\Lambda_0) \in \tfrac{1}{2}\Z
	$$
	where $\mu_{\RS}$ denotes the Robbin-Salamon index, see Section~\ref{sec: RSindex}.
\end{definition}
\noindent Since the Maslov class $\mu_{\widehat{L}}$ vanishes, $\mu(c)$ is independent of the choices involved, see \cite[Lemma~2.1]{KKL}.

\begin{remark} One can define the Maslov index using any Lagrangian subspace $\Lambda$ in $\C^n$ other than~$\Lambda_0$. The Maslov index is independent of the choice of $\Lambda$.
\end{remark}

\begin{definition}
We define the {\it index} of $c\in\mathcal{P}(H)$ by $|c|:=-\mu(c)-\frac{n}{2}$.	
\end{definition}

\begin{remark}
If $c$ is non-degenerate, then  we have $|c|\in \Z$ by \cite[Theorem~2.4]{RSindex}.	
\end{remark}

\subsubsection{Action functional}
We let
\begin{equation}\label{eq: pathspace}
\mathscr{P}=\{c\in W^{1,2}([0,1],\widehat{W})\mid c(0),c(1)\in \widehat{L}; \;[c]=0\in \pi_1(\widehat{W},\widehat{L})\}	
\end{equation}
be the space of contractible paths of class $W^{1,2}$ with ends in $\widehat{L}$. Recall that $\widehat{L}$ is exact, so $\widehat{\lambda}|_{\widehat{L}}=d\ell$ for some $\ell\in C^\infty(\widehat{L})$. We fix any such $\ell$. The {\it action functional} $\mathcal{A}_H:\mathscr{P} \to \R$ associated to  a Hamiltonian $H \in C^{\infty}(\widehat{W})$ is defined by
\begin{eqnarray*}
\mathcal{A}_H(c) &=& -\int_{D^+}v^*d\widehat{\lambda}-\int_0^1 H(c(t))dt \\
&=& \ell(c(1))-\ell(c(0))-\int_0^1c^*\widehat{\lambda}-\int_0^1H(c(t))dt,
\end{eqnarray*}
where $v:D^+\to \widehat{W}$ is a capping half-disk of $c \in \mathscr{P}$. We write $\crit(\mathcal{A}_H)$ for the set of critical points of $\mathcal{A}_H$. Then $\crit(\mathcal{A}_H)=\mathcal{P}(H)$. An $\widehat{\ow}$-compatible almost complex structure $J$ on $\widehat{W}$ is said to be of \emph{contact type}  if it satisfies the condition $\widehat{\lambda}\circ J=dr$ on $[1,\infty)\times \p W$. We abbreviate
\begin{equation}\label{eq: space_almcpxstr}
\mathcal{J}:=\{\{J_t\}_{t\in[0,1]}\mid \text{there exists $r_0\ge 1$ such that each $J_t$ is of contact type for $r\ge r_0 $} \}.
\end{equation}
An element $J\in \mathcal{J}$ is called \emph{admissible}. In particular, for an admissible $J$ we  can apply the maximum principle for solutions to Floer-type equations, see \cite[Lemma~4.5]{Rit}.

Now we fix a non-degenerate admissible Hamiltonian $H \in C^{\infty}(\widehat{W})$ and $J=\{J_t\}_{t\in[0,1]}\in \mathcal{J}$. For $c_\pm\in \mathcal{P}(H)$ with $c_-\ne c_+$ we let $\widehat{\mathcal{M}}(c_-,c_+;H,J)$ be the \emph{space of Floer strips} consisting of $u:\R\times [0,1]\to \widehat{W}$ satisfying
\begin{eqnarray*}
\begin{cases}
\p_s u+J_t(u)(\p_t u -X_H(u))=0,  \\
\displaystyle\lim_{s\to \pm\infty}u(s,t)=c_\pm(t),\\
u(s,0),u(s,1)\in \widehat{L}.	
\end{cases}
\end{eqnarray*}
This space carries a free $\R$-action which is given by translation in the $s$-variable. Its quotient
$$
\mathcal{M}(c_-,c_+;H,J):=\widehat{\mathcal{M}}(c_-,c_+;H,J)/\R
$$ is called the {\it moduli space of Floer strips} from $c_-$ to $c_+$. The action decreases along Floer strips, and a standard argument (see \cite{FHS} and \cite{RS2}) shows that $\mathcal{M}(c_-,c_+;H,J)$ is a smooth manifold of dimension $|c_-|-|c_+|-1$ for a  generic $J\in \mathcal{J}$.

Let $\tau\in \R\cup \{\infty\}$. The {\it filtered wrapped Floer chain complex} is defined by
$$
CF_*^{<\tau}(H)=\bigoplus_{\substack{ c\in \mathcal{P}(H)\text{ with }|c|=* \\ \mathcal{A}_H(c)<\tau}}\Z_2\langle c\rangle.
$$
The {\it Floer differential} is defined by counting rigid Floer strips
$$
\p: CF_*^{<\tau}(H)\to CF_{*-1}^{<\tau}(H), \quad \p(c_-)=\sum_{\substack{ c_+\in \mathcal{P}(H)\\ |c_-|-|c_+|=1}}\#_{2}\mathcal{M}(c_-,c_+;H,J)\cdot c_+,
$$
where $\#_{2}$ denotes the count modulo two. For a  generic choice of $J$, the differential is well-defined and satisfies $\partial \circ \partial=0$. We define the {\it filtered wrapped Floer homology of the Hamiltonian $H$} as the homology of the chain complex
$$
HF_*^{<\tau}(H,J)=H_*(CF_*^{<\tau}(H),\partial).
$$
For generic pairs $(H_\pm, J_\pm)$ with $H_-\le H_+$, one can define a  \emph{continuation map} $HF_*^{<\tau}(H_-,J_-)\to HF_*^{<\tau}(H_+,J_+)$
by counting rigid $s$-dependent Floer strips, see \cite[Section~3.2]{Rit} for the construction. Using continuation maps one shows that $HF_*^{<\tau}(H,J)$ does not depend on the choice of $J$ and only depends on the slope of $H$, see \cite[Lemma~3.1]{Rit}. We therefore omit $J$ from the notation. We define the {\it filtered wrapped Floer homology of a Lagrangian $L$} by
$$
HW_*^{<\tau}(L;W):=\lim_{\underset{H}{\longrightarrow}} HF_*^{<\tau}(H),
$$
where the direct limit is taken over admissible Hamiltonians using continuation maps. If $\tau=\infty$ we call $HW_*(L;W):=HW_*^{<\infty}(L;W)$ the \emph{wrapped Floer homology} of $L$.

\subsubsection{Positive wrapped Floer homology}
Fix $\epsilon>0$ which is smaller than the minimum length of Reeb chords from $\p L$ to $\p L$. We define the \emph{positive wrapped chain complex} as the quotient chain complex
$$
CF_*^{+}(H):=CF_*(H)\big /CF_*^{<\epsilon}(H)
$$
equipped with the induced differential. Following the usual limit procedure, the resulting homology group, denoted by $HW^+_*(L;W)$, is called \emph{positive wrapped Floer homology} of the Lagrangian $L$. There exists a short exact sequence
$$
0\to CF_*^{<\epsilon}(H) \to CF_*(H) \to CF_*^+(H) \to 0,
$$
which induces a long exact sequence
$$
\cdots \to HF_*^{<\epsilon}(H) \to HF_*(H) \to HF_*^+(H) \to \cdots.
$$
Taking the direct limit with respect to $H$, we obtain the tautological exact sequence
\begin{equation}\label{eq: viterboLES}
\cdots \to H_{*+n}(L,\p L) \to HW_*(L;W) \to HW_*^+(L;W) \to \cdots,
\end{equation}
where  we have used the fact that $HW^{<\epsilon}_*(L;W)$ is isomorphic to $H_{*+n}(L,\p L)$, see \cite[Proposition~1.3]{Vit}.

\subsection{Vanishing of wrapped Floer homology} \label{sec: vanishing wfh sectioin} In this section, we show that (non-equivariant) wrapped Floer homology vanishes under a displaceability assumption. In \cite{Kang3}, using the notion of leaf-wise intersections, it is proved that the symplectic homology of a Liouville domain vanishes if its boundary is displaceable. We carry out this idea in wrapped Floer homology. 
 
\subsubsection{Displaceability} In what follows, $(W, \lambda)$ is a Liouville domain, $\Sigma=\p W$ is a contact type boundary, $L$ is an admissible Lagrangian, and   $\mathcal{L} := L \cap \Sigma$  is a Legendrian submanifold in $\Sigma$. The following notion of displaceability is taken from \cite{Merry}. 

\begin{definition}\label{def: disp}
The hypersurface $\Sigma$ is \emph{displaceable from $\widehat L$ in $\widehat W$} if there exists a compactly supported Hamiltonian diffeomorphism $\phi$ on $\widehat W$ such that $\phi(\Sigma) \cap \widehat L = \emptyset$.
\end{definition}

\begin{example}\label{examp:hypersurfaceids}
Starshaped hypersurfaces $\Sigma$ in $\R^{2n}$ are displaceable from the real Lagrangian $\widehat L = \{(q, p) \in \R^{2n}\;|\; p = 0\}$.
\end{example}

\begin{proposition}\label{prop: Lagselfdisp}
If $\Sigma$ is displaceable from $\widehat L$ in $\widehat W$, then $W$ is displaceable from $\widehat L$ in $\widehat W$ by the same displacing Hamiltonian diffeomorphism.
\end{proposition}

\begin{proof}
Let $\phi : \widehat W \rightarrow \widehat W$ be a Hamiltonian diffeomorphism displacing $\Sigma$. Note that the  hypersurface $\Sigma \subset \widehat W$ separates $\widehat W$ into two connected components: the bounded piece $\text{int} (W)$ and the unbounded piece $\widehat W \setminus W$. The image $\phi(\Sigma)$ also separates $\widehat W$ into the bounded piece $\phi(\text{int}(W))$ and the other. Since $\widehat L$ is connected and unbounded, $\widehat L$ cannot entirely be contained in the bounded piece $\phi(\text{int} (W))$. Since $\phi(\text{int} (W))$ is bounded by $\phi(\Sigma)$, it follows that if $\phi(\text{int} (W)) \cap \widehat L \neq \emptyset$, then $\phi(\Sigma) \cap \widehat L \neq \emptyset$.
\end{proof}


\subsubsection{Filtered wrapped Floer homology and a tweaked action functional}

Let $H_{\tau}: \widehat W \rightarrow \R$ be an admissible Hamiltonian of slope $\tau$ and take  any  \emph{compactly supported} Hamiltonian $F \in C_c^{\infty}([0,1] \times \widehat W)$.  Following \cite{Merry}  we define a \emph{tweaked action functional} $\mathcal{A}_{H_{\tau}}^F : \mathscr{P} \rightarrow \R$ by
$$
\mathcal{A}_{H_{\tau}}^F(c) = \ell(c(1)) - \ell(c(0)) - \int_0^1 c^*\widehat \lda - \int_0^1 \beta(t){H_{\tau}}(c(t))dt - \int_0^1 \dot \chi(t) F(\chi(t),c(t))dt,
$$
where $\beta: [0,1] \rightarrow \R$ is a non-negative smooth function with $\beta(t) = 0$ for $t \in [\frac{1}{2},1]$ and $\int_0^1 \beta(t) dt = 1$, and $\chi: [0,1] \rightarrow [0,1]$ is a smooth monotone function such that $\chi(\frac{1}{2}) = 0$ and $\chi(1) = 1$. The following lemma is straightforward.

\begin{lemma}
The differential of $\mathcal{A}_{H_{\tau}}^F$ is given by
$$
d\mathcal{A}_{H_{\tau}}^F(c)(\zeta) = \int_0^1 \widehat \ow(\dot c(t) - X_{\beta(t){H_{\tau}}}(c(t)) - \dot \chi (t) X_{F}(\chi(t), c(t)), \zeta(t)) dt\quad \text{for $\zeta\in T_c\mathscr{P}$}.
$$
Consequently, $c \in \crit(\mathcal{A}_{H_{\tau}}^F)$ if and only if $c \in \mathscr{P}$ solves the equation
\begin{equation}\label{eq: eqcrit}
\dot c(t) = \beta(t) X_{H_{\tau}}(c(t)) + \dot \chi(t) X_{F}(\chi(t), c(t)).
\end{equation}
\end{lemma}


\begin{remark}
Critical points of the tweaked action functional $\mathcal{A}_{H_{\tau}}^F$ are closely related to the notion of \emph{relative leaf-wise intersections} in \cite{Merry}.
\end{remark}

The following lemma allows us to define the Floer homology of $\mathcal{A}_{H_{\tau}}^F$ for generic $F$. The proof is analogous to \cite[Appendix~A]{AFleafwise}, see also \cite[Theorem~2.28]{Merry}.

\begin{lemma}
For a generic $F$ in the $C^{\infty}$-topology, the tweaked functional $\mathcal{A}_{H_{\tau}}^F$ is Morse.
\end{lemma}
Denote the Floer homology of the tweaked action functional $\mathcal{A}_{H_{\tau}}^F$ by $HF_*(\mathcal{A}_{H_{\tau}}^F)$. 

\begin{proposition}
$HF_*(\mathcal{A}_{H_{\tau}}^F) \cong HW_*^{<\tau}(L;W)$.
\end{proposition}

\begin{proof}
It follows from a standard argument using continuation maps with respect to a homotopy between $F$ and the constant function $0$ that $HF_*(\mathcal{A}_{H_{\tau}}^F) \cong HF_*(H_{\tau})$; We refer to \cite[Section~2.2]{Kang3} for details. On the other hand, since $H_{\tau}$ is an admissible Hamiltonian of slope $\tau$, we have $HF_*(H_{\tau}) \cong HW_*^{<\tau}(L; W)$.
\end{proof}


\subsubsection{Vanishing of wrapped Floer homology} \label{sec: vanHW} We shall prove the following vanishing property.
\begin{theorem}\label{thm: incvan}
If $\Sigma$ is displaceable from $\widehat L$ in $\widehat W$, then the inclusion map
$$
\iota_*: HW^{<\tau}_*(L; W) \rightarrow HW^{<\tau + e(L)}_*(L; W)
$$
is a zero map for all $\tau \in \R$, where $e(L)$ is defined as in \eqref{defofeL}.
\end{theorem}

An immediate corollary is the following.

\begin{corollary}\label{cor: displaceable}
If $\Sigma$ is displaceable from $\widehat L$ in $\widehat W$, the wrapped Floer homology $HW_*(L; W)$ vanishes.
\end{corollary}

\begin{remark}\label{rmk: vanishingSHandHW}
If $\Sigma$ is displaceable \emph{from itself}, then we already know that $HW_*(L;W)$ vanishes for any admissible Lagrangian $L$. In this case, the symplectic homology $SH_*(W)$ vanishes as  shown in \cite{Kang3}. By the result in \cite[Theorem~6.17]{Rit} the wrapped Floer homology $HW_*(L;W)$ is a module over $SH_*(W)$ for any admissible Lagrangian $L$. It follows that the wrapped Floer homology $HW_*(L;W)$ vanishes.
\end{remark}

For a given compactly supported Hamiltonian $F \in C^{\infty}_c([0,1] \times \widehat W)$, we define
$$
\mathfrak{f}: = \max_{t\in [0,1]} \Supp(r\circ F(t,\cdot)).
$$
Here $r$ denotes the cylindrical coordinate, and if $F$ is supported only in $W$, then we put $\mathfrak{f} = 0$ by convention. In the rest of this section, we   assume that $\Sigma$ is displaceable from $\widehat L$ and that $\phi_F$ is a displacing Hamiltonian diffeomorphism.

\begin{lemma}\label{lem: cylcorcrit}
If $c \in \crit (\mathcal{A}_{H_{\tau}}^F)$, then we have $r \circ c(0) \in (1, \mathfrak{f}]$. In particular, if $\mathfrak{f}=0$ the set $\crit(\mathcal{A}_{H_\tau}^F)$ is empty.
\end{lemma}

\begin{proof}
Note that $c$ is a solution of the equation \eqref{eq: eqcrit}. Suppose $r \circ c(0) > \mathfrak{f}$. Then $c$ is a Reeb chord of period $\tau \not\in \Spec(\Sigma, \alpha, \mathcal{L})$, which contradicts the choice of $\tau$. Suppose $r \circ c(0) \leq 1$, in other words $c(0) \in L$. Then $\phi_{H_{\tau}}^{t} (c(0))  \subset W$ for all $t\in \R$ since $H_{\tau}$ is admissible. We therefore have $\phi_F(\phi_{H_{\tau}}^t(c(0))) \in \widehat L$ for some $t \in \R$. In particular, $\phi_F(W) \cap \widehat L \neq \emptyset$. This contradicts Proposition~\ref{prop: Lagselfdisp}.
\end{proof}

We set $\displaystyle \|F\|_{-}: = \int_0^1 \min_{y \in \widehat W} F(t, y) dt.$ 

\begin{lemma}\label{lem: chotau}
For each $a > 0$, there exists $\tau=\tau(a)>0$ as large as one likes such that there is no critical point $c$ of $\mathcal{A}_{H_{\tau}}^F$ with $\mathcal{A}_{H_{\tau}}^F(c) < a + \|F\|_-$. Consequently, the filtered chain complex $CF^{<a + \|F\|_-}_*(\mathcal{A}_{H_{\tau}}^F)$ and hence the homology $HF^{<a + \|F\|_-}_*(\mathcal{A}_{H_{\tau}}^F)$ vanishes.
\end{lemma}

\begin{proof}
For notational convenience, we denote $\tilde F (t, \cdot) : = \dot \chi(t) F(\chi(t), \cdot)$, and we choose without loss of generality $H_{\tau}(r, y) = \tau r - \tau$ for $r >1$. By Lemma~\ref{lem: cylcorcrit}, if $c \in \crit (\mathcal{A}_{H_{\tau}}^F)$, then $r \circ c(0) \in (1, \mathfrak{f}]$. In particular, $c$ is a solution of the equation $\dot c = - \tau \beta(t) R(c) + X_{\tilde F}(t, c)$ where $R$ denotes the Reeb vector field. Abbreviating $r_0: = r \circ c(0)$ we compute 
\begin{align*}
\mathcal{A}_{H_{\tau}}^F(c)&= \ell(c(1)) - \ell(c(0)) - \int_0^1 r\alpha\big(- \tau \beta(t) R(c) + X_{\tilde F}(t, c)\big) dt - \int_0^1 (\beta(t)H_{\tau}(c) + \tilde F(t, c)) dt \\
&=\ell(c(1)) - \ell(c(0)) + \tau \cdot r_0 - \int_0^1 r \alpha(X_{\tilde F}(t, c))dt - (\tau \cdot r_0 - \tau) - \int_0^1 \tilde F(t, c)dt\\
&= \ell(c(1)) - \ell(c(0)) + \tau - \int_0^1 (r \alpha(X_{\tilde F}(t,c)) + \tilde F(t, c))dt.
\end{align*}
Since $r \circ c \leq \mathfrak{f}$, the term $\int_0^1 (r \alpha(X_{\tilde F}(t,c)) + \tilde F(t, c))dt$ is bounded from above by a constant $C_F > 0$ which does not depend on $c$. It follows that
$$
\mathcal{A}_{H_{\tau}}^F(c) \geq \ell(c(1)) - \ell(c(0)) + \tau - C_F.
$$
We then take $\tau= \tau(a)>0$ such that $\ell(c(1)) - \ell(c(0)) + \tau - C_F \geq a + \|F\|_-$. This completes the proof.
\end{proof}

We now prove the vanishing property. This will basically be done by showing that the inclusion $\iota_*: HF^{< a}_*(\mathcal{A}_{H_{\tau(a)}}) \rightarrow HF^{<a+ \|F\|}_*(\mathcal{A}_{H_{\tau(a)}})$ factors through $HF^{<a + \|F\|_-}_*(\mathcal{A}_{H_{\tau}}^F)$ which vanishes for suitable $\tau = \tau(a)$. Define the \emph{displacement energy of $\Sigma$ from $\widehat L$ in $\widehat W$} by
\begin{equation}\label{defofeL}
e(L): = \inf\{\|F\|: F \in C^{\infty}_c([0,1] \times \widehat W), \; \phi_F(\Sigma) \cap \widehat L = \emptyset \}.
\end{equation}

\begin{proof}[Proof of Theorem~\ref{thm: incvan}]
Choosing a generic homotopy $F^s$ such that $F^0 \equiv 0$ and $F^1 =\dot{\chi}(t)F(\chi(t),\cdot)$, we have continuation maps on filtered homology groups
\begin{eqnarray}
&&\Phi^{<a}_*: HF_*^{< a}(\mathcal{A}_{H_{\tau}}) \rightarrow HF_*^{< a + \|F\|_-}(\mathcal{A}_{H_{\tau}}^F) \nonumber\\
&&\Psi^{<a}_*: HF_*^{< a}(\mathcal{A}_{H_{\tau}}^F) \rightarrow HF_*^{< a + \|-F\|_-}(\mathcal{A}_{H_{\tau}})\nonumber
\end{eqnarray}
for each $a >0$. We remark that the action shift by $\|F\|_-$ (or $\|-F\|_-$) is due to the energy consumption property of parametrized Floer solutions which we count to define the continuation maps, see Remark~\ref{rem: energyparam}. By composing the two maps, we have a map (note that $\|-F\|_- + \|F\|_- = \|F\|$)
$$
\Psi_*^{<a + \|F\|_-} \circ \Phi_*^{<a} : HF_*^{< a}(\mathcal{A}_{H_{\tau}}) \rightarrow HF_*^{< a + \|F\|}(\mathcal{A}_{H_{\tau}}),
$$
and by the usual homotopy of homotopies argument in Floer theory, we have that
$$
\Psi_*^{<a + \|F\|_-} \circ \Phi_*^{<a} = \iota_*
$$
where $\iota_*$ is the map induced by inclusion at the chain level.

Taking $\tau = \tau(a)$ as in Lemma~\ref{lem: chotau}, we have $HF_*^{< a + \|F\|_-}(\mathcal{A}_{H_{\tau}}^F)=0$. Therefore the map
$$
\iota_*: HF_*^{< a}(\mathcal{A}_{H_{\tau}}) \rightarrow HF_*^{< a + \|F\|}(\mathcal{A}_{H_{\tau}})
$$
vanishes for all $F$ and $a$. We finally take the direct limit along $\tau  \rightarrow \infty$, and conclude that the inclusion
$$
\iota_*: HW_*^{<a}(L; W) \rightarrow HW_*^{<a+e(L)}(L; W)
$$
also vanishes.
\end{proof}

\begin{remark}\label{rem: energyparam}
For a solution of the parametrized Floer equation $w: \R \times [0,1] \rightarrow \widehat W$, the energy
$$
E(w) : = \int_{-\infty}^{\infty} \|\partial_s w\|^2 ds
$$
satisfies  
$$
E(w) \leq \mathcal{A}_{\beta \cdot {H_{\tau}}}(w_{-}) - \mathcal{A}_{{H_{\tau}}}^F(w_{+}) + \|F\|_-.
$$
\end{remark}

\section{Equivariant wrapped Floer homology} \label{sec: ewf}

\subsection{$\Z_2$-complexes} \label{sec: z2cpx} As an algebraic preliminary to develop equivariant theories, we briefly outline the notion of $\Z_2$-complex. The parallel notion of an $S^1$-complex can be found in \cite{BO17}.
\begin{definition}
A $\Z$-graded chain complex $(C_*, \partial)$ with $\Z_2$-coefficients is called a \emph{$\Z_2$-complex} if it admits the additional datum of a sequence of maps
$$
\phi = (\phi_0 = \partial, \phi_1, \phi_2, \dots)
$$
such that $\phi_j$ has degree $j-1$, i.e., $\phi_j : C_k \rightarrow C_{k +j-1}$, and satisfies the relations
\begin{equation}\label{eq: phi_rel}
\forall k \geq 0, \quad \sum_{i+j = k} \phi_i \circ \phi_j = 0.
\end{equation}
\end{definition}

\begin{remark}
The first two terms of the relation \eqref{eq: phi_rel} are
$$
\partial \circ \partial = 0, \quad \partial \circ \phi_1 + \phi_1 \circ \partial = 0.
$$
\end{remark}
Given a $\Z_2$-complex $(C_*, \partial)$, we form the \emph{$\Z_2$-equivariant chain complex} $(C^{\Z_2}_*, \partial^{\Z_2})$ by
$$
C^{\Z_2}_* := \Z_2[w] \otimes_{\Z_2} C_*,
$$
where $w$ is a formal variable of degree one, and
\begin{equation}\label{eq: ez2differentialalg}
\partial^{\Z_2}(w^l \otimes x) : = \sum_{j=0}^l w^{l-j}\phi_j(x).
\end{equation}
We note that the differential formally splits into
$$
\partial^{\Z_2} = \partial + w^{-1}\phi_1 + w^{-2}\phi_2 + \cdots.
$$
It follows immediately from the relations \eqref{eq: phi_rel} that we have $\partial^{\Z_2} \circ \partial^{\Z_2} = 0$.
\begin{definition}
The \emph{$\Z_2$-equivariant homology} of the $\Z_2$-complex $(C_*, \phi)$ is the homology of the $\Z_2$-equivariant chain complex $(C_*^{\Z_2},\p^{\Z_2})$:
$$
H^{\Z_2}_*(C_*, \phi): = H_*(C^{\Z_2}_*, \partial^{\Z_2}).
$$
\end{definition}

\subsubsection{Leray--Serre type spectral sequence} \label{sec: lstss} For each $p \in \Z$ we denote the $\Z_2$-module of polynomials in $w$ of degree less than or equal to $p$ by $\Z_2[w]/\{w^{p+1}=0\}$. We set $F_pC_*^{\Z_2}=0$ for $p<0$. Note that the tensor product 
$$
F_p C^{\Z_2}_* :=\Z_2[w]/\{w^{p+1}=0\}  \otimes_{\Z_2} C_*
$$
is a subgroup in $C_*^{\Z_2}$. The following lemma is straightforward from the definition.
\begin{lemma}
The subgroups $\{F_p C^{\Z_2}_* \}_{p \in \Z}$ define  a filtration on $(C^{\Z_2}_*, \partial^{\Z_2})$ which is bounded from below and is exhaustive.
\end{lemma}

By a standard theorem of homological algebra, we have a spectral sequence $\{E^r_{p,q}\}$ which converges to the $\Z_2$-equivariant homology $H_*^{\Z_2}(C_*)$.  Denoting the induced filtration on $C_k^{\Z_2}$ by $F_p C_k^{\Z_2}$ for each $k \geq 0$, the $E^0$-term is given by 
$$
E^0_{p,q} : = F_p C^{\Z_2}_{p+q}/ F_{p-1} C^{\Z_2}_{p+q},
$$
and the differential $d^0_{p,q}$ is the one  induced by $\partial^{\Z_2}$ on the   quotient $E^0_{p,q}$.

\begin{lemma}\label{lem: alglsss}
The $E^1$-term of the spectral sequence is given by
$$
E^1_{p,q} = \begin{cases} H_q(C_*), & p \geq 0, \\ 0, & \text{$p<0$.}  \end{cases}
$$
The differential $d^1_{p,q} : E^1_{p,q} \rightarrow E^1_{p-1,q}$ is given by $d^1 = \phi_1$.
\end{lemma}

\begin{proof}
Observe that we can identify $E^0_{p,q}$ with $\Z_2\langle w^p\rangle \otimes C_q$ using the fact that we are over the field $\Z_2$. For an element $w^p \otimes x \in \Z_2\langle w^p\rangle \otimes C_q$, we have
$$
\partial^{\Z_2}(w^p \otimes x) = \sum_{j=0}^p w^{p-j} \otimes \phi_j(x),
$$
and this implies, passing to the quotient, that
$$
d^0_{p,q}(w^p \otimes x) = w^p \otimes \phi_0(x) = w^p \otimes \partial(x).
$$
In other words,
$$
d^0_{p,q}	 = \id \otimes \partial,
$$
so that $E^1_{p,q} = \Z_2\langle w^p\rangle \otimes_{\Z_2} H_q(C_*) = H_q(C_*)$. 

The same computation shows that 
$$
d^1_{p,q}(w^p \otimes x) = w^{p-1} \otimes \phi_1(x).
$$
Under the identification $\Z_2\langle w^p\rangle \otimes_{\Z_2} H_q(C_*) = H_q(C_*)$, we can write $d^1 = \phi_1$ as asserted.
\end{proof}

\begin{remark}
Since $H_*(\R P^{\infty}) = H_*(B\Z_2) = \Z_2[w]$, one may rephrase the $E^1$-term as follows:
$$
E^1_{p,q} = H_p(B\Z_2) \otimes_{\Z_2} H_q(C_*) 
$$ 
\end{remark}

\subsection{Equivariant Morse homology}\label{sec: eMH} Before defining equivariant wrapped Floer homology, we describe the construction of a finite dimensional model, namely, equivariant Morse homology following \cite{SS}. This construction is the so-called family Morse homology theory, see \cite{Hut}. Throughout this paper, we use $\Z_2$-coefficients in (equivariant) homology.

\subsubsection{$\Z_2$-spaces and the Borel constructions}\label{sec: Borel}
A topological space endowed with a topological $\Z_2$-action is called a \emph{$\Z_2$-space}. Note that a topological space $X$ is a $\Z_2$-space if and only if  it admits an involution. It is well-known that $\displaystyle B\Z_2=\R P^\infty=\lim_{\underset{N}{\longrightarrow}}\R P^N$ and $\displaystyle E\Z_2=S^\infty=\lim_{\underset{N}{\longrightarrow}} S^N$. For a $\Z_2$-space $X$, the group $\Z_2$ acts on the product $X\times E\Z_2$ diagonally and its quotient space
$$
X_\text{Borel}:=X\times_{\Z_2}E\Z_2
$$ is called the {\it Borel construction}. If $X$ is a point, then $X_\text{Borel}=E\Z_2/\Z_2=B\Z_2=\R P^\infty$. In the Morse setup, our approach to the equivariant theory corresponds to considering the fibration
$$
X \longhookrightarrow X_\text{Borel}\longrightarrow \R P^\infty,
$$
where the projection map is induced by the projection onto the second factor $pr_2:X\times E\Z_2\to E\Z_2$. The \emph{equivariant homology of a $\Z_2$-space $X$} is defined as the singular homology of its Borel construction:
$$
H_*^{\Z_2}(X):=H_*(X_\text{Borel}).
$$
If the $\Z_2$-action is free, i.e., if the involution has no fixed points, then we have
$$
H_*^{\Z_2}(X)=H_*(X/\Z_2).
$$
In practice, we approximate $E\Z_2=S^\infty$ by finite-dimensional spheres $S^N$ on which $\Z_2$ still acts freely.  It is known that 
$$
H_*^{\Z_2}(X)=\lim_{\underset{N}{\longrightarrow}}H_*(X\times_{\Z_2}S^N),
$$
where the direct limit is taken over morphisms induced by the equivariant inclusions $X\times_{\Z_2} S^N\hookrightarrow X\times_{\Z_2} S^{N+1}$.
 
\subsubsection{$\Z_2$-equivariant Morse chain complex}\label{subsec:eMH}
Let $M$ be a closed manifold with a $\Z_2$-action, induced by an involution $\iota: M \rightarrow M$. The quotient of the unit-sphere $S^N\subset \R^{N+1}$ by the antipodal map $\sigma(z)=-z$ is $\R P^N$. We take the perfect Morse function $f: \R P^{N} \rightarrow \R$ whose lift $\widetilde{f} : S^N \rightarrow \R$ to the covering $S^N$ is given by $\widetilde{f}(z) = \sum_{j=0}^{N} j|z_j|^2$. The function $f$ has precisely one critical point, say $z^{(j)}$, of each index $0 \leq j \leq N$. We take a family of functions $h_z : M \rightarrow \R$ parametrized by $z \in S^{N}$, which yields a function $h: M \times S^N \rightarrow \R$ such that
\begin{itemize}
\item $h$ is $\Z_2$-invariant in the sense that $h_{-z} \circ \iota = h_z$;
\item $h+\tilde{f}$ is Morse, and if $(x, z) \in \crit(h+\tilde{f})$, then $z \in \crit(\widetilde f)$ (and hence $x \in \crit(h_z))$;  
\item for each $z \in \crit(\tilde f)$, the corresponding function $h_z$ is Morse;
\item $h_z$ is independent of $z$ near each critical point of $\tilde{f}$. 
\end{itemize}
A construction of such an $h_z$ is illustrated in  Example~\ref{ex: prodham} where a similar construction in Floer theory is given. The Morse case is similar but much simpler.
\begin{remark}
Note that $h_z: M \rightarrow \R$ itself need not be $\Z_2$-invariant for a fixed $z \in S^N$.
\end{remark}
Let $g_{S^N}$ be the round metric on $S^N$. It is $\Z_2$-invariant and Morse-Smale with respect to $\tilde{f}$. Take a $\Z_2$-invariant family of metrics $g = \{g_z\}_{z \in S^N}$ on $M$, i.e., $\iota^*g_{-z} = g_z$, which is locally constant near the critical points of $\tilde f$. We then have a $\Z_2$-invariant metric $ g_z \oplus g_{S^N}$ on $M \times S^N$, which we by abuse of notation still denote by $g$. 
\begin{remark}
For a fixed $z \in S^N$, the metric $g_z$ is not necessarily $\Z_2$-invariant.
\end{remark}
The critical points of $h+\tilde{f}$ always come in pairs due to  the $\Z_2$-invariance:  $p=(x, z) \in \crit(h+\tilde{f})$ if and only if $\overline{p}:=(\iota(x), -z) \in \crit(h+\tilde{f})$. We denote by $Z_p=  \{ p , \overline{p}   \}$  the $\Z_2$-pair of $p \in \crit(h)$.  Its index is defined as  
$$
|Z_p| : = \ind(x; h_z) + \ind(z; \tilde f), \quad p = (x,z).
$$
We define the \emph{equivariant Morse chain complex} by
$$
CM^{\Z_2, N}_*(h, g) : = \bigoplus_{\substack{Z_p \subset  \crit(h+\tilde{f}) \\ |Z_p|=*}} \Z_2\langle Z_p \rangle.
$$
Given two critical points $p_{\pm} = (x_{\pm}, z_{\pm}) \in \crit(h+\tilde{f})$, we denote by $\widehat{\mathcal{M}}(Z_{p_-}, Z_{p_+}; h, g)$ the space of pairs $(u,v)$, where  $u: \R \rightarrow M$ and $v: \R \rightarrow S^N$ are solutions to the system of ordinary differential equations
$$
\begin{cases}
\p_su(s) + \nabla h_{v(s)}(u(s)) = 0, \\
\p_sv(s) + \nabla \tilde f(v(s)) = 0,
\end{cases}
$$
with asymptotic conditions
$$
\lim_{s \rightarrow \pm \infty}(u(s), v(s)) \in Z_{p_\pm}.
$$
\begin{remark}
Note that the above system of differential equations is \emph{not} the negative gradient flow equation of a Morse function $h:M\times S^N\to \R$.	
\end{remark}
For $Z_{p_-} \neq Z_{p_+}$, denote the quotient of $\widehat{\mathcal{M}}(Z_{p_-}, Z_{p_+}; h, g)$ by the $\R$-action by $\mathcal{M}(Z_{p_-}, Z_{p_+}; h, g)$. Note that this moduli space carries a free $\Z_2$-action. We denote its $\Z_2$-quotient by $\mathcal{M}_{\Z_2}(Z_{p_-}, Z_{p_+}; h, g)$. By a standard argument, for a generic $g$ the moduli space $\mathcal{M}_{\Z_2}(Z_{p_-}, Z_{p_+}; h, g)$ is a smooth manifold of dimention $|Z_{p_-}|-| Z_{p_+}|-1$. In particular, if $|Z_{p_-}|-| Z_{p_+}| =1$, then this space is a compact zero-dimensional manifold and hence a finite set. We   define the \emph{equivariant Morse differential} by
$$
\p^{\Z_2}(Z_{p_-})=\sum_{|Z_{p_-}|-|Z_{p_+}|=1}\#_{2}\mathcal{M}_{\Z_2}(Z_{p_-},Z_{p_+};h,g)\cdot Z_{p_+}.
$$
Denote  by $HM_*^{\Z_2, N}(M)$ the homology of the equivariant Morse chain complex $(CM^{\Z_2, N}_*(h, g), \p^{\Z_2})$. As the notation indicates, it can be shown that this group does not depend on the choice of $(h, g)$. Finally, the \emph{equivariant Morse homology} is defined by   taking the direct limit with respect to the maps induced by  the equivariant inclusions $S^N\hookrightarrow S^{N+1}$:
$$
HM_*^{\Z_2}(M) : = \varinjlim_{N} HM_*^{\Z_2, N}(M).
$$
As in the non-equivariant case, the equivariant Morse homology is isomorphic to the equivariant singular homology. See \cite[Theorem~5.1]{Hut} and \cite[Proposition~2.5]{BO17} for this result in the symplectic homology setup.

\subsubsection{Periodic family and $\Z_2$-complex structure} \label{sec: perMorcompl} In this section, we re-interpret the differential~$\p^{\Z_2}$ as a $\Z_2$-equivariant differential of a $\Z_2$-complex in the sense of Section~\ref{sec: z2cpx}. The point is that we can canonically identify   fibers of $M_{\text{Borel}}$ over a critical point of $f$ with each other. To do this, we use the canonical shift on $\R P^{\infty}$ given by the infinite shift
$$
\tau:\R P^\infty\to \R P^\infty, \quad [z_0:z_1: \cdots]\mapsto [0:z_0:z_1:\cdots].
$$
Since $\R P^\infty$ is no longer a smooth manifold, let us clarify this symmetry in a finite approximation following ideas in \cite[Section~2.3]{BO17}. From now on, we write $h_N:=h:M\times S^N\to \R$ for $N\ge 1$ where $h$ is given in Section~\ref{subsec:eMH}.

Consider the two embeddings 
\begin{eqnarray*}
&&\tilde{\tau}_0:S^N\to S^{N+1},\quad (z_0,\dots,z_N)\mapsto (z_0,\dots,z_N,0) \\
&&\tilde{\tau}_1:S^N\to S^{N+1},\quad (z_0,\dots,z_N)\mapsto (0,z_0,\dots,z_N). 
\end{eqnarray*}
Note that $\tilde{\tau}_0$ is the canonical embedding and $\tilde{\tau}_1$ is an embedding given by the shift. They induce embeddings
\begin{equation}\label{eq: shiftemb}
\widetilde{\mathcal{T}}_j:M \times S^N\rightarrow M \times S^{N+1}, \quad (x,z) \mapsto (x, \tilde{\tau}_j(z)).
\end{equation}
We put an additional condition on $\{(h_N, g_N)\}_{N\ge 1}$, which we call \emph{periodicity}:
$$
h_N= \widetilde{\mathcal{T}}^*_0 h_{N+1}=\widetilde{\mathcal{T}}^*_1 h_{N+1}, \quad g_N = \widetilde{\mathcal{T}}^*_0g_{N+1} = \widetilde{\mathcal{T}}^*_1g_{N+1}
$$
hold  for all $N\ge 1$. 
\begin{remark}\label{rem: periodicseq}
A periodic family of functions $\{h_N\}$ can always be constructed from a given Morse function $h_0: M \rightarrow \R$. One can use the same recipe as in Example~\ref{ex: prodham}. Indeed, if we choose a family of local slices and bump functions $\beta:S^N\to [0,1]$ in a periodic manner, the resulting family $\{h_N\}$ is periodic. The same argument applies to construct a periodic family of metrics $\{g_N\}$.
\end{remark}
We also observe that $\{(\tilde{f}_N,g_{S^N})\}_{N\ge 1}$ is \emph{periodic} meaning that for all $N\ge 1$:
\begin{itemize}
	\item $\im(\tilde{\tau}_0)$ and $\im(\tilde{\tau}_1)$ are invariant under the gradient flow of $\tilde{f}_{N+1}$;
	\item $\tilde{f}_N=\tilde{\tau}_0^*\tilde{f}_{N+1}=\tilde{\tau}_1^*\tilde{f}_{N+1}+const$, and $g_{S^N}=\tilde{\tau}_0^*g_{S^{N+1}}=\tilde{\tau}_1^*g_{S^{N+1}}$.
\end{itemize}
This periodicity implies that the gradient flow of $\tilde{f}_N$ is preserved under $\tilde{\tau}_0$ and $\tilde{\tau}_1$.  A nice feature of the periodicity condition is that it simplifies our equivariant Morse chain complex.

Recall that the perfect Morse function $f_N:\R P^N\to \R$ has precisely one critical point $z^{(j)}$ of each index $0\le j \le N$. Its preimage under the projection is  denoted by $\{ \pm z^{(j)} \}\subset S^N$. Fix a Morse function $h_0 : = h_{N,z^{(0)}}: M \rightarrow \R$. By the periodicity, we can assume that $z^{(j)}$ are chosen such that each Morse function $h_{N,z^{(j)}}:M\to \R$ on the fiber $M\times \{z^{(j)}\}$ is given by $h_0$. Note that $h_{N,-z^{(j)}}$ is automatically determined by the $\Z_2$-invariance:   $h_{N,-z^{(j)}}=\iota^*h_{N,z^{(j)}}$. We choose a metric $g_0$ on $M$ which is Morse-Smale with respect to $h_0$. We then have a canonical identification between critical points of $h_{N,z^{(j)}}$ and $h_0$ for all $j\ge 1$. For a critical point $x \in \crit(h_0)$ we formally denote the corresponding critical point of $h_{N,z^{(j)}}$ by $w^jx$, where as before $w$ denotes a formal variable of degree $1$.  As a result, the equivariant Morse chain group is simplified under this identification to
\begin{equation}\label{eq: simplymorsecomplex}
CM_*^{\Z_2, N}(h_N, g_N) = \Z_2[w]/\{w^{N+1}=0\} \otimes_{\Z_2} CM_*(h_0, g_0).
\end{equation}
 Indeed, every generator in $CM_*^{\Z_2,N}(h_N,g_N)$ is of the form $Z_p$ for some $p=(w^jx,z^{(j)})$ with $x\in \crit(h_0)$ and $j\ge 0$. Then the generator $Z_p$ corresponds uniquely to $w^j \otimes x$ in the identification \eqref{eq: simplymorsecomplex}. We also simplify the moduli spaces  which we count for $\p^{\Z_2}$ in the following way. Since our Morse data are periodic, there is a one-to-one correspondence
\begin{equation}\label{lem: permoduli}
\mathcal{M}_{\Z_2}(Z_{p_-},Z_{p_+}; h_N, g_N) \cong \mathcal{M}_{\Z_2}(Z_{q_-}, Z_{q_+}; h_N, g_N),
\end{equation}
where $p_- = (w^jx_-, z^{(j)})$, $p_+ = (w^kx_+, z^{(k)})$, $q_- = (w^{j-k}x_-, z^{(j-k)})$, and $q_+ = (x_+, z^{(0)})$ for some $j > k$.
Since the periodicity condition still leaves enough freedom to achieve  all  necessary transversality of moduli spaces, we can define a family of maps $\phi_{\ell}: CM_*(h_0,g_0) \rightarrow CM_{* + ({\ell}-1)}(h_0,g_0)$~by 
\begin{equation}\label{eq: z2comhomo}
\phi_{\ell} (x_-) :=  \sum_{\substack{x_+  \text{ with}\\ \text{ind}(x_-;h_0) - \text{ind}(x_+;h_0) +{\ell}=1  }} \#_{2} \mathcal{M}_{\Z_2}(w^{\ell} \otimes x_-, x_+; h_N, g_N) \cdot x_+.	
\end{equation}
By analyzing the boundary of the 1-dimensional moduli spaces, we get the relations \eqref{eq: phi_rel} for $0\le k\le N$.
\begin{corollary}\label{cor:diffequi}
Under the identification \eqref{eq: simplymorsecomplex} the equivariant Morse differential $\p^{\Z_2}$ is given by \eqref{eq: ez2differentialalg}.
\end{corollary}

\begin{proof} Choose $Z_{p_{-}}$ (resp.\ $Z_{p_+}$) which corresponds to $w^j \otimes x_-$ (resp.\ $w^k \otimes x_+$). In view of the correspondence \eqref{lem: permoduli} we can identify the moduli space  $\mathcal{M}_{\Z_2} (Z_{p_-}, Z_{p_+} ;h_N, g_N)$ with   $\mathcal{M}_{\Z_2} (w^{j-k} \otimes x_-,   x_+ ;h_N, g_N)$. Moreover, the identity   $|Z_{p_-}|-|Z_{p_+}|=1$ can be rephrased as 
\begin{align*}
1 &= |Z_{p_-}|-|Z_{p_+}|\\
&=\left( \text{ind}(x_-;h_0)+\text{ind}(z^{(j-k)};\tilde{f}_N) \right) - \left( \text{ind}(x_+;h_0)+\text{ind}(z^{(0)};\tilde{f}_N) \right)\\
&= \text{ind}(x_-;h_0) - \text{ind}(x_+;h_0) + (j-k).
\end{align*}
Then the differential can be written as
$$
\p^{\Z_2}(w^j \otimes x_-) =  \sum_{k=0}^j   \sum_{\substack{x_+  \text{ with}\\ \text{ind}(x_-;h_0) - \text{ind}(x_+;h_0) +j-k=1  }}   \#_{2}   \mathcal{M}_{\Z_2} (w^{j-k} \otimes x_-,   x_+ ;h_N, g_N) \cdot  w^k \otimes x_+.
$$
The differential $\p^{\Z_2}$ is therefore given by
$$
\p^{\Z_2}(w^{j} \otimes  x) = \sum_{k=0}^{j} w^{k} \otimes \phi_{j-k}(x) = \sum_{k=0}^{j} w^{j-k} \otimes \phi_{k}(x) .
$$
This completes the proof of the corollary.
\end{proof}
Since for a fixed $j$, the maps $\phi_j$ defined in \eqref{eq: z2comhomo} for varying $N\ge j$ coincide, we can take a direct limit of chain complexes $\{(CM^{\Z_2,N}_*(h_N,g_N),\{\phi_j\}_{0\le j\le N})\}$ as $N\to \infty$, with respect to the maps induced by the equivariant inclusions $S^N\hookrightarrow S^{N+1}$. We denote by 
$$
CM_*^{\Z_2}(h,g)=\Z_2[w]\otimes_{\Z_2}CM_*(h_0,g_0)
$$
the direct limit of chain complexes equipped with maps $\{\phi_j\}_{j\ge 0}$.
Then the Morse chain complex $CM_*(h_0, g_0)$ is a $\Z_2$-complex with the collection of maps $\{\phi_j\}$. Therefore by definition the (limit) chain complex $(CM_*^{\Z_2}(h,g), \p^{\Z_2})$ is the equivariant chain complex of the $\Z_2$-complex $(CM_*(h_0, g_0), \p, \{\phi_j\})$. Its homology is isomorphic to $HM_*^{\Z_2}(M)$ as the homology functor commutes with direct limits. By Lemma~\ref{lem: alglsss}, we immediately obtain the following assertion which gives the classical Leray--Serre spectral sequence.
\begin{corollary}
	There is a spectral sequence $\{E^r\}_{r\ge 0}$ converging to $HM_*^{\Z_2}(M)$ with the first page given by
	$$
	E^1_{p,q}=\begin{cases}
		HM_q(M), & p\ge 0, \\
		0, & \text{otherwise}.
	\end{cases}
	$$
\end{corollary}

\subsection{Definition of equivariant wrapped Floer homology} \label{sec: defEWFH}
In this section, we introduce the equivariant wrapped Floer homology of an admissible Lagrangian $L$ in a Liouville domain $(W,\lambda)$ when it admits an involution $\I:W\to W$. We assume that one of the following holds:
\begin{quote}
	{\bf (Symplectic involution)}
	\begin{itemize}
		\item 	$\I$ is \emph{exact symplectic}, i.e., $\I^*\lambda=\lambda$, and
		\item  $L$ is $\I$-invariant.
	\end{itemize}
\end{quote}
\begin{quote}
	{\bf (Anti-symplectic involution)} 
	\begin{itemize}
		\item 	$\I$ is \emph{exact anti-symplectic}, i.e., $\I^*\lambda=-\lambda$, and
		\item $L=\Fix(\I)$ and $\p L=\Fix(\I|_{\p W})$ are nonempty.	
	\end{itemize}
\end{quote}
In the latter case, the triple $(W,\lambda,\I)$ is called a  {real Liouville domain} and $L=\Fix(\I)$ is called a \emph{real Lagrangian}. Note that real Lagrangians are  admissible and $\I$-invariant.
We mainly follow ideas in \cite{SS}, where equivariant Lagrangian Floer homology (with a symplectic involution) is constructed using family Morse homology theory. See also \cite{BO17} for the construction of $S^1$-equivariant symplectic homology and \cite{UrsMoon} for the case of an anti-symplectic involution.

\subsubsection{Geometric setup}
We observe that $\I$ restricts to an involution $\I|_{\p W}$ on the contact manifold $(\p W,\alpha:=\lambda|_{\p W})$ which is 
\begin{quote}
\begin{itemize}
	\item  a (strict) \emph{contact involution}, i.e., $\I|_{\p W}^*\alpha=\alpha$ if $\I$ is exact symplectic;
	\item an (strict) \emph{anti-contact involution}, i.e., $\I|_{\p W}^*\alpha=-\alpha$ if $\I$ is exact anti-symplectic.
\end{itemize}
\end{quote}
Note that a Legendrian $\p L$ is $\I|_{\p W}$-invariant. We then extend $\I$ to an exact (anti-)symplectic involution $\widehat{\I}$ defined on the completion $\widehat{W}$:
$$
\widehat{\I}(x)=\begin{cases}
	\I(x) & \text{for $x\in W$},\\
	(r,\I|_{\p W}(y)) & \text{for $x=(r,y)\in [1,\infty)\times \p W$}.
\end{cases}
$$
One can see that $\widehat{L}$ is $\widehat{\I}$-invariant. The path space $\mathscr{P}$ defined in \eqref{eq: pathspace} admits an involution
\begin{equation}\label{eq: involutionpathspace}
\I_{\mathscr{P}}:\mathscr{P}\longrightarrow \mathscr{P},\quad \I_{\mathscr{P}}(c(t)):=\begin{cases}
\widehat{\I}(c(t)) & \text{if $\I$ is symplectic}, \\
\widehat{\I}(c(1-t)) & \text{if $\I$ is anti-symplectic}.
\end{cases}	
\end{equation}
By abuse of notation, we abbreviate by $\I$ the various involutions on $\widehat{W}$, $\p W$, and $\mathscr{P}$.

\subsubsection{A family of action functionals}
Fix $N\in \N$. For a family of Hamiltonians
$$
H:\widehat{W}\times S^N\longrightarrow \R
$$
we denote $H_z:=H(\cdot,z)$ for $z\in S^N$. We say that a family of Hamiltonians $H:\widehat{W}\times S^N\to \R$ is \emph{$\Z_2$-invariant} if it is invariant under the diagonal $\Z_2$-action on $\widehat{W}\times S^N$, i.e., 
\begin{equation}\label{eq : Z2invHam}
H_{-z}(\I x)=H_z(x)\quad\text{for all $x\in \widehat{W}$ and  $z\in S^N$}.
\end{equation}
Since $\widehat{L}$ is exact, there exists $\ell \in C^\infty(\widehat{L})$ such that $\widehat{\lambda}|_{\widehat{L}}=d\ell$ and $\ell$ is  $\I$-invariant up to a constant. Given a $\Z_2$-invariant Hamiltonian family $H:\widehat{W}\times S^N\to \R$, we consider the corresponding  \emph{family of action functionals} $\mathcal{A}:\mathscr{P}\times S^N\to \R$ given by
$$
\mathcal{A}(c,z):=\mathcal{A}_{H_z}(c)=\ell(c(1))-\ell(c(0))-\int_0^1c^*\widehat{\lambda}-\int_0^1H_z(c(t))dt.
$$
One sees that $\mathcal{A}$ is $\Z_2$-invariant with respect to the diagonal $\Z_2$-action on $\mathscr{P}\times S^N$.  The differential of $\mathcal{A}$ is given by
$$
d\mathcal{A}(c,z)\cdot(\zeta,\ell ) = \int_0^1 \widehat{\ow}(\dot{c}(t)-X_{H_z}(c(t)),\zeta(t))dt-\int_0^1\frac{\p H}{\p z}(c(t),z)dt \cdot \ell,
$$
where $\zeta\in T_c\mathscr{P}$ and $\ell\in T_zS^N$. Hence, $(c,z)\in \mathscr{P}\times S^N$ is a critical point of $\mathcal{A}$ if and only if 
\begin{equation}\label{eq: critpt}
c\in \mathcal{P}(H_z)\quad\text{and}\quad \int_0^1 \frac{\p H}{\p z}(c(t),z)dt=0.	
\end{equation}
We write $\mathcal{P}(H):=\crit(\mathcal{A})$ for the set of critical points of $\mathcal{A}$. 
\subsubsection{A family of almost complex structures}
A family of $\widehat{\ow}$-compatible almost complex structures $J=\{J_z^t\}_{t\in [0,1],z\in S^N}$ on $\widehat{W}$ is called \emph{$\Z_2$-invariant} if the following holds:
\begin{itemize}
	\item if $\I$ is symplectic it satisfies
$$
J_{-z}^t=\I^*J^t_z:=\I_*\circ J^t_z\circ \I_* \quad \text{for all $t,z$},
$$
	\item if $\I$ is anti-symplectic it satisfies
$$
-J_{-z}^{1-t}=\I^*J_z^t\quad \text{for all $t,z$}.
$$
\end{itemize}
Such a $J$ is called \emph{admissible} if $J_z=\{J_z^t\}_{t\in [0,1]}\in \mathcal{J}$ for all $z\in S^N$, where $\mathcal{J}$ is defined in \eqref{eq: space_almcpxstr}. We denote by $\mathcal{J}_N^{\Z_2}$ the set of $\Z_2$-invariant admissible families of $\widehat{\ow}$-compatible almost complex structures. The set of $\mathcal{J}_N^{\Z_2}$ is nonempty and contractible, see \cite{MS17}. For an admissible $J$, we consider the $\Z_2$-invariant $S^N$-family of $L^2$-metrics on $\mathscr{P}$ given by
$$
\langle \zeta_1,\zeta_2 \rangle_z := \int_0^1 \widehat{\ow}(\zeta_1(t),J_z^t\zeta_2(t))dt\quad \text{for $z\in S^N$ and $\zeta_1,\zeta_2\in T_c\mathscr{P}$}.
$$
As in Section~\ref{subsec:eMH}  this together with a $\Z_2$-invariant metric $g$ on $S^N$  gives rise to a $\Z_2$-invariant metric on $\mathscr{P}\times S^N$. The $L^2$-gradient of $\mathcal{A}$ at $(c,z)\in \mathscr{P}\times S^N$ is given by
\begin{equation}\label{eq: gradient}
\begin{split}
	\nabla\mathcal{A}(c,z) &=J_z^t(\dot{c}-X_{H_z}(c))-\int_0^1\frac{\p H}{\p z}(c(t),z)dt \\
	&=\nabla \mathcal{A}_{H_z}(c)-\int_0^1\frac{\p H}{\p z}(c(t),z)dt,
\end{split}
\end{equation}	
where $\nabla \mathcal{A}_{H_z}$ denotes the $L^2$-gradient of the functional $\mathcal{A}_{H_z}$ with respect to the metric $\langle \cdot, \cdot \rangle_z$. Since $\mathcal{A}$ is $\Z_2$-invariant, the set $\mathcal{P}(H)$ is $\Z_2$-invariant. Given $p\in \mathcal{P}(H)$, we denote by $Z_p$ the \emph{critical $\Z_2$-pair} of $p$. In particular, we have $Z_p=Z_{\I p}$, see Figure \ref{fig: z_2orbits}.

\begin{figure}[h]
\begin{subfigure}{0.48\textwidth}
   \centering
\begin{tikzpicture}[scale=0.2]

\begin{scope}[yscale=1,xscale=1,xshift=2cm]
 \draw [thick] plot [smooth,tension=0.7] coordinates {(7,0)
(7,0.5)
(7,3)
(9,3)
(10,5)
(8,9)
(6,9)
(5,7)
(6,5)
(5,2)
(3,3)
(2.5,6)
(1,6)
(0,6)
(0,5)
(1,3)
(2,1.3)
(2,0.3)
(2,0)
};
\draw [->,thick] (9.2,7.2)--(8.8,7.9) ; 
\draw [fill] (7,0) circle [radius=0.3];
 \draw [fill] (2,0) circle [radius=0.3];
\node at (3,9)  {$c$};
\end{scope}

\begin{scope}[yscale=-1,xscale=-1,xshift=2cm]
 \draw [thick] plot [smooth,tension=0.7] coordinates {(7,0)
(7,0.5)
(7,3)
(9,3)
(10,5)
(8,9)
(6,9)
(5,7)
(6,5)
(5,2)
(3,3)
(2.5,6)
(1,6)
(0,6)
(0,5)
(1,3)
(2,1.3)
(2,0.3)
(2,0)
};
\draw [->,thick] (9.2,7.2)--(8.8,7.9) ; 
\draw [fill] (7,0) circle [radius=0.3];
\draw [fill] (2,0) circle [radius=0.3];
\node at (3,9)  {$\I(c)$};
\end{scope}

\draw [dashed] (-13,0)--(13,0);
\node at (13,0) [right]{$\widehat{L}$};
\draw [fill] (0,0) circle [radius=0.2];
\draw [->] (2,0.5) arc [x radius = 2cm, y radius = 2cm, start angle = 0, end angle = 180];
\node at (0,3.5) {$\I$};

\end{tikzpicture}

  \caption{$\mathcal{I}$ is symplectic.}
\end{subfigure}
\begin{subfigure}{0.48\textwidth}
  \centering
\begin{tikzpicture}[scale=0.2]
\draw [dashed] (-5,0)--(15,0);

\begin{scope}[yscale=1,xscale=1]
\draw [thick] plot [smooth,tension=0.7] coordinates {(7,0)
(7,0.5)
(7,3)
(9,3)
(10,5)
(8,9)
(6,9)
(5,7)
(6,5)
(5,2)
(3,3)
(2.5,6)
(1,6)
(0,5)
(1,3)
(2,1.3)
(2,0.3)
(2,0) };
\draw [->,thick] (9.2,7.2)--(8.8,7.9) ;
\draw [fill] (7,0) circle [radius=0.3];
 \draw [fill] (2,0) circle [radius=0.3];
\node at (3,9)  {$c$};
\end{scope}

\begin{scope}[yscale=-1,xscale=1]
\draw [thick] plot [smooth,tension=0.7] coordinates {(7,0)
(7,0.5)
(7,3)
(9,3)
(10,5)
(8,9)
(6,9)
(5,7)
(6,5)
(5,2)
(3,3)
(2.5,6)
(1,6)
(0,5)
(1,3)
(2,1.3)
(2,0.3)
(2,0) };
\draw [<-,thick] (9.2,7.2)--(8.8,7.9) ;
\node at (3,9)  {$\I(c)$};
\end{scope}

\draw [<->] (12,-2)--(12,2) ;

\draw [dashed] (-5,0)--(12,0);
\node at (13,3) {$\I$};
\node at (14,0) [right]{$\widehat{L}=\Fix(\I)$};

\end{tikzpicture}
  \caption{$\mathcal{I}$ is anti-symplectic.}
\end{subfigure}
\caption{An illustration of $\Z_2$-pairs: the antipodal involution and complex conjugation on $\C$.}
\label{fig: z_2orbits}
\end{figure}
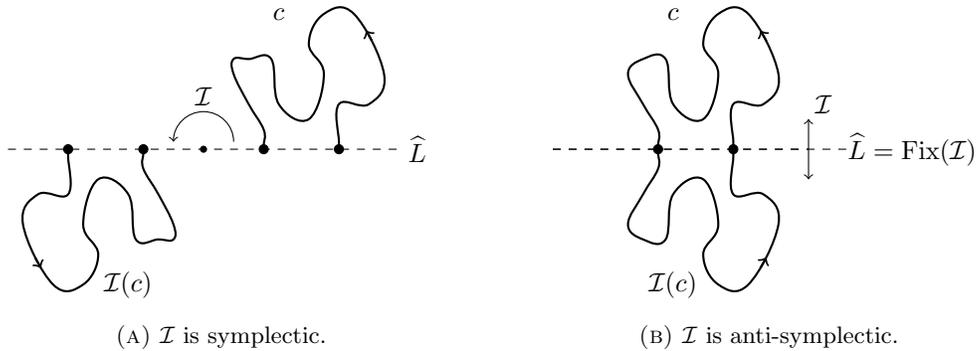

\begin{definition}
	A $\Z_2$-invariant Hamiltonian family $H:\widehat{W}\times S^N\to \R$ is called weakly admissible if $H_z\in \mathcal{H}$ for all $z\in S^N$ with constant slope not depending on $z$.
\end{definition}
We denote by $\mathcal{H}_N^{\Z_2}$ the space of $\Z_2$-invariant weakly admissible families of Hamiltonians. Note that the set $\mathcal{H}_N^{\Z_2}$ is nonempty.
\subsubsection{Hessian and non-degeneracy}
Let $\nabla$ be the $t$-dependent Levi-Civita connection with respect to the metric $\widehat{\ow}(\cdot,J_z^t\cdot)\oplus g$ on $\widehat{W}\times S^N$.
\begin{lemma}
	The Hessian of $\mathcal{A}$ at a critical point $p=(c,z)\in \mathcal{P}(H)$ is given by
\begin{equation}\label{eq: asympop}
\begin{split}
	\Hess_{(c,z)}\mathcal{A}:T_c\mathscr{P}\oplus T_z S^N &\longrightarrow  L^2(c^*T\widehat{W})\oplus T_zS^N \\
	(\zeta,\ell) &\longmapsto  \left( \begin{matrix}
J_z^t(\nabla_t\zeta-\nabla_\zeta X_{H_z}-(D_z X_{H_z})\cdot \ell) \\
-\int_0^1 \nabla_\zeta\frac{\p H}{\p z}dt-\int_0^1\nabla_\ell \frac{\p H}{\p z}dt
	\end{matrix} \right).
\end{split}
\end{equation}

\end{lemma}
\begin{proof}
Take $c_s\in \mathscr{P}$ and $z_s\in S^N$, $s\in (-\epsilon,\epsilon)$ such that $(d/ds)c_s=\zeta$ and $(d/ds)z_s=\ell$. 
Using \eqref{eq: gradient}, we verify that
	\begin{eqnarray*}
		\Hess_{(c,z)}\mathcal{A}(\zeta,\ell) &=& \nabla_s \nabla \mathcal{A}(c_s,z_s) |_{s=0} \\
		&=& \nabla_s\left[ \nabla \mathcal{A}_{H_{z_s}}(c_s)-\int_0^1\frac{\p H}{\p z}(c_s,z_s)dt \right]|_{s=0} \\
		&=& \nabla_\zeta\nabla\mathcal{A}_{H_{z_s}}(c_s)|_{s=0}+\nabla_\ell\nabla \mathcal{A}_{H_{z_s}}(c_s)|_{s=0} -\nabla_s\int_0^1\frac{\p H}{\p z}(c_s,z_s)dt|_{s=0} \\
		&=& \Hess_c\mathcal{A}_{H_z}(\zeta) +\nabla_\ell(J^t_{z_s}(c)(\dot{c}-X_{H_{z_s}}(c)))|_{s=0}\\
		&&-\int_0^1\nabla_\zeta \frac{\p H}{\p z}(c_s,z_s)|_{s=0}dt-\int_0^1\nabla_\ell \frac{\p H}{\p z}(c_s,z_s)|_{s=0}dt\\
		&=& \Hess_c\mathcal{A}_{H_z} (\zeta) + (\nabla_\ell J_z^t)(\underbrace{\dot{c}-X_{H_z}(c)}_{=0})-J_z^t(c)\nabla_\ell X_{H_z}(c)\\
		&&-\int_0^1\nabla_\zeta \frac{\p H}{\p z}(c,z)dt-\int_0^1\nabla_\ell \frac{\p H}{\p z}(c,z)dt\\
		&=& \Hess_c\mathcal{A}_{H_z}(\zeta) -J^t_z(c)D_zX_{H_z}(\ell)-\int_0^1\nabla_\zeta \frac{\p H}{\p z}(c,z)dt-\int_0^1\nabla_\ell \frac{\p H}{\p z}(c,z)dt.
	\end{eqnarray*}
Here $\Hess_c\mathcal{A}_{H_z}$ denotes the Hessian of $\mathcal{A}_{H_z}$ at $c$. This completes the proof.
\end{proof}
\begin{definition}
A  $\Z_2$-pair of critical points $Z_p\subset \mathcal{P}(H)$ is called \emph{non-degenerate} if the kernel of the Hessian $\Hess_{(c,z)}\mathcal{A}$ is trivial for some (and hence any) $(c,z)\in Z_p$.
\end{definition}

\subsubsection{Admissible family of Hamiltonians}\label{sec:Produclikeham}
Let  $f$ and $\tilde{f}$ be defined as in Section~\ref{subsec:eMH}. From now on we fix a local slice for each critical $\Z_2$-pair of $\tilde{f}$ and assume that all local slices are disjoint. Since $\Z_2$ is finite, local slices are nothing but $\Z_2$-invariant open neighborhoods where the $\Z_2$-action is given by the antipodal map $\sigma$ on $S^{N}$.

Let $U$ be the union of all local slices, which is an open neighborhood of $\crit(\tilde{f})$. Choose an open  set $U_0$ in $S^N$ such that $U_0$ and $\sigma(U_0)$ are disjoint and $U=U_0\sqcup \sigma(U_0)$. Intuitively, one may think of $U_0$ as a neighborhood of $\crit(f)$ in $\R P^N$. 
\begin{definition}
	A family of Hamiltonians $H:\widehat{W}\times S^N\to \R$ is called \emph{admissible} if it satisfies~that
	\begin{enumerate}
	\item $H\in \mathcal{H}_N^{\Z_2}$;
	\item \label{eq: prod_prop2} Critical points of $\mathcal{A}_{H+\tilde{f}}$ lie over critical points of $\tilde{f}$. More precisely, if $(c,z)\in\crit(\mathcal{A}_{H+\tilde{f}})$, then we have $z\in \crit(\tilde{f})$;
	\item \label{eq: prod_prop3} For each $z\in \crit(\tilde{f})$, the corresponding Hamiltonian $H_z$ is non-degenerate;
	\item \label{eq: prod_prop4} All critical $\Z_2$-pairs in $\crit(\mathcal{A}_{H+\tilde{f}})$ are non-degenerate;
	\item \label{eq: prod_prop5} $H_z$ is independent of $z$ along each local slice.
	\item  \label{eq: prod_prop6}The inequality
\begin{equation}\label{eq: prod_estimate}
	\|\p_z H\|< \min_{S^N\setminus U} \|\nabla \tilde{f}\|
\end{equation}  
holds.
\end{enumerate}
\end{definition}
We denote by $\mathcal{H}_f \subset \mathcal{H}_N^{\Z_2}$ the subclass of admissible families of Hamiltonians.
\begin{remark}
	Condition \eqref{eq: prod_prop2} above is equivalent to the condition
	$$
	\crit(\mathcal{A}_{H+\tilde{f}})=\bigcup_{z\in \crit(\tilde{f})}\mathcal{P}(H_z)\times \{z\}.
	$$ 
Condition \eqref{eq: prod_prop3}  actually follows from the other conditions, but for simplicity we included it. Condition  \eqref{eq: prod_prop6}  implies the action decreasing property of solutions that is used to define the $\Z_2$-equivariant differential, see Lemma~\ref{lam: actdecreasing}.
\end{remark}
The space $\mathcal{H}_f$ is nonempty in view of the following explicit example.
\begin{example}\label{ex: prodham}
We construct an admissible family of Hamiltonians $H\in \mathcal{H}_{f}$. Fix an admissible non-degenerate Hamiltonian $H_0:\widehat{W}\to \R$. We first define the $\Z_2$-invariant Hamiltonian $H':\widehat{W}\times U\to \R$ by 
$$
H'(x,z):=\begin{cases}
	H_0(x) & \text{if $z\in U_0$,} \\
	H_0(\I x) & \text{if $z \in  \sigma(U_0).$}
\end{cases}
$$
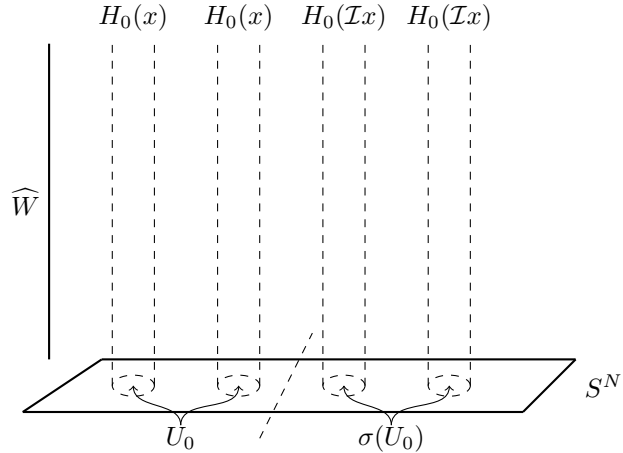
\begin{figure}[h]
\begin{center}
\begin{tikzpicture}[scale=0.7]


\draw [thick] (0,-3)--(0,3);

\draw [thick] (1,-3)--(10,-3);

\draw [thick] (-0.5,-4)--(9,-4);

\draw [thick] (1,-3) to (-0.5,-4);
\draw [thick] (10,-3) to (9,-4);

\draw [dashed] (2,-3.5) arc [x radius = 0.4cm, y radius = 0.2cm, start angle = 0,  end angle = 360];

\draw [dashed] (4,-3.5) arc [x radius = 0.4cm, y radius = 0.2cm, start angle = 0,  end angle = 360];

\draw [dashed] (6,-3.5) arc [x radius = 0.4cm, y radius = 0.2cm, start angle = 0,  end angle = 360];

\draw [dashed] (8,-3.5) arc [x radius = 0.4cm, y radius = 0.2cm, start angle = 0,  end angle = 360];

\draw [dashed] (1.2,-3.5) to (1.2,3);
\draw [dashed] (2,-3.5) to (2,3);
\draw [dashed] (3.2,-3.5) to (3.2,3);
\draw [dashed] (4,-3.5) to (4,3);
\draw [dashed] (5.2,-3.5) to (5.2,3);
\draw [dashed] (6,-3.5) to (6,3);
\draw [dashed] (7.2,-3.5) to (7.2,3);
\draw [dashed] (8,-3.5) to (8,3);
\draw [dashed] (4,-4.5) to (5,-2.5);




\draw [->](2.5,-4.25) to [out=90,in=-90](1.6,-3.5);
\draw [->](2.5,-4.25) to [out=90,in=-90](3.6,-3.5);

\draw [->](6.5,-4.25) to [out=90,in=-90](5.6,-3.5);
\draw [->](6.5,-4.25) to [out=90,in=-90](7.6,-3.5);

\node at (0,0) [left]{$\widehat{W}$};
\node at (10,-3.5) [right]{$S^N$};
\node at (1.6,3.5) {$H_0(x)$};
\node at (3.6,3.5) {$H_0(x)$};
\node at (5.6,3.5) {$H_0(\I x)$};
\node at (7.6,3.5) {$H_0(\I x)$};
\node at (2.5,-4.5) {$ U_0 $};
\node at (6.5,-4.5) {$ \sigma(U_0) $};


\end{tikzpicture}
\end{center}
\caption{An illustration of $H'$}
\label{fig: construction_H}
\end{figure}
\noindent See Figure~\ref{fig: construction_H}. To be able to  extend $H'$ to a globally defined family $H$ with constant slope, we choose  a $\Z_2$-invariant cut-off function $\beta:S^N\to \R$ such that
$$
\beta(z)=\begin{cases}
	1 & \text{on $U'$,} \\
	0 & \text{outside of $U$,}
\end{cases}
$$
where $U'\subset U$ is a  $\Z_2$-invariant neighborhood of $\crit(\tilde{f})$ such that $\overline{U'}\subset U$. Also choose a cut-off function $\eta:\widehat{W}\to \R$ such that
$$
\eta(x)=\begin{cases}
	1 & \text{on $[1,\infty)\times \p W$,} \\
	0 & \text{on $W\setminus \big((1-\delta,1]\times \p W\big)$},
\end{cases}
$$
for sufficiently small $\delta>0$. We then define the $\Z_2$-invariant Hamiltonian $H:\widehat{W}\times S^N\to \R$ by
$$
H(x,z):=\epsilon\big[\beta(z)H'(x,z)+(1-\beta(z))\eta(x)H_0(x)\big],
$$
where $\epsilon>0$ is chosen so small that the inequality \eqref{eq: prod_estimate} holds (with $U'$ instead of $U$). By construction we have $H\in \mathcal{H}_N^{\Z_2}$ and it satisfies condition \eqref{eq: prod_prop5} for $U'$. Since $H_0$ is non-degenerate, condition \eqref{eq: prod_prop3} is also fulfilled. \\

{\bf Claim 1.} \emph{Condition \eqref{eq: prod_prop2} holds}.\\
Let $\tilde{H}=H+\tilde{f}$. We observe that $\p_z \tilde{H}=\nabla\tilde{f}$ on $\widehat{W}\times U'$. By \eqref{eq: critpt}, we obtain 
$$
\crit(\mathcal{A}_{\tilde{H}})\supset \bigcup_{z\in \crit(\tilde{f})}\mathcal{P}(H_z)\times \{z\}.
$$
To see the reverse, let $z\in S^N\setminus U'$. Using \eqref{eq: prod_estimate}, we verify that 
$$
\int_0^1 \frac{\p \tilde{H}}{\p z}(c(t),z)dt= \int_0^1 \frac{\p H}{\p z}(c(t),z)+\nabla \tilde{f}(z)dt\ne 0 \quad\text{for $c\in \mathcal{P}(H_z)$}.
$$
Using \eqref{eq: critpt} again we obtain Claim 1.\\

{\bf Claim 2.} \emph{Condition \eqref{eq: prod_prop4} holds}.\\
Fix $(c,z)\in \crit(\mathcal{A}_{\tilde{H}})$. Let $(\zeta,\ell)\in T_c\mathscr{P}\oplus T_zS^N$ such that
$$
\Hess_{(c,z)}\mathcal{A}(\zeta,\ell)=\left( \begin{matrix}
J_z(\nabla_t\zeta-\nabla_\zeta X_{\tilde{H}_z}-(D_z X_{\tilde{H}_z})\cdot \ell) \\
-\int_0^1 \nabla_\zeta\frac{\p \tilde{H}}{\p z}dt-\int_0^1\nabla_\ell \frac{\p \tilde{H}}{\p z}dt
\end{matrix} \right)=\left(	\begin{matrix}
		0 \\
		0
	\end{matrix}\right).
$$
We need to show that $(\zeta,\ell)=0$. We first verify that
\begin{eqnarray*}
	\int_0^1 \nabla_\zeta\frac{\p\tilde{ H}}{\p z}dt+\int_0^1\nabla_\ell \frac{\p \tilde{H}}{\p z}dt &=& \int_0^1\underbrace{\nabla_\zeta \frac{\p \tilde{f}}{\p z}}_{=0}dt + \int_0^1\nabla_\ell \frac{\p \tilde{f}}{\p z}dt \\
	&=&\nabla_\ell \nabla\tilde{f}\\
	&=&\Hess_z \tilde{f}(\ell),
\end{eqnarray*}
where $\Hess_z\tilde{f}$ is the Hessian of $\tilde{f}$ at $z$.
Since $\tilde{f}$ is non-degenerate, $\ell=0$. Further,
\begin{eqnarray*}
J_z(\nabla_t\zeta-\nabla_\zeta X_{\tilde{H}_z}-(D_z X_{\tilde{H}_z})\cdot \ell )&=& J_z(\nabla_t \zeta-\nabla_\zeta X_{\tilde{H}_z}) =\Hess_c \mathcal{A}_{\tilde{H}_z}(\zeta).
\end{eqnarray*}
Since $\tilde{H}_z=H_z+const$ is non-degenerate, we have $\zeta=0$. This completes the proof of Claim 2.
\end{example}

\begin{definition}
Let $H\in \mathcal{H}_f$. The \emph{Maslov index} of $p=(c,z)\in \crit (\mathcal{A}_{H+\tilde{f}})$ is defined by
$$
\mu(p):=\mu(c)-\ind(z;-\tilde{f}),
$$
where $\ind(z;-\tilde{f})$ denotes the \emph{Morse index} of $z$ with respect to $-\tilde{f}$. The \emph{Maslov index} of a critical $\Z_2$-pair $Z_p$ is defined by
$$
\mu(Z_p):=\mu(p') \quad \text{for $p'\in Z_p$}.
$$
	
\end{definition}
\subsubsection{Moduli spaces}
Fix $H\in \mathcal{H}_f$. Let $J\in \mathcal{J}_N^{\Z_2}$ be such that $J_z$ is independent of $z$ along each local slice and $g$ a $\Z_2$-invariant metric on $S^N$ which is Morse-Smale with respect to $\tilde{f}$. Given $p_\pm =(c_\pm, z_\pm) \in \crit(\mathcal{A}_{H+\tilde{f}})$ we denote by $\widetilde{\mathcal{M}}(Z_{p_-},Z_{p_+};H,J)$ the space of pairs $(u, z)$, where $u:\R \times [0,1]\to \widehat{W}$ and  $z:\R \to S^N$  are solutions to the system of equations
\begin{equation}\label{eq: moduli}
	\begin{cases}
	\p_s u+J_{z(s)}^t(u)(\p_t u - X_{H_{z(s)}}(u))=0,\\
	\dot{z}-\nabla \tilde{f}(z) =0,
\end{cases}
\end{equation}
with asymptotic and boundary conditions
$$
\lim_{s\to \pm\infty}(u(s,\cdot),z(s))\in Z_{p_\pm},\quad u(s,0),u(s,1)\in \widehat{L}.
$$
If $Z_{p_-}\ne Z_{p_+}$ let $\mathcal{M}(Z_{p_-},Z_{p_+};H,J):=\widetilde{\mathcal{M}}(Z_{p_-},Z_{p_+};H,J)/\R$ be the quotient by the $\R$-translation in the $s$-variable. This space admits a free $\Z_2$-action
$$
(u(s,t),z(s))\mapsto \begin{cases}
(\I(u(s,t)),-z(s)) & \text{if $\I$ is symplectic},\\
	(\I(u(s,1-t),-z(s)) & \text{if $\I$ is anti-symplectic},
\end{cases}
$$
which is induced by the involution \eqref{eq: involutionpathspace} on the path space equipped with the antipodal map on~$S^N$. The $\Z_2$-quotient of $\mathcal{M}(Z_{p_-},Z_{p_+};H,J)$ is denoted by 
$$
\mathcal{M}_{\Z_2}(Z_{p_-},Z_{p_+};H,J).
$$
\begin{remark}
Since the system of differential equations \eqref{eq: moduli} is not the  negative gradient flow equation of $\mathcal{A}_{H+\tilde{f}}$, it is not obvious that the action decreases along solutions to   equation \eqref{eq: moduli}. However, it will turn out that with   condition \eqref{eq: prod_estimate} we have the action decreasing property, in Section~\ref{sec: actdecreas}.
\end{remark}
\subsubsection{Transversality issues} This section is devoted to prove the following. 
\begin{proposition}
	For generic $J\in \mathcal{J}^{\Z_2}_N$, the moduli space $\mathcal{M}_{\Z_2}(Z_{p_-},Z_{p_+};H,J)$ is a smooth manifold of dimension $\mu(Z_{p_+})-\mu(Z_{p_-})-1$.
\end{proposition}

For the proof, we shall translate the system \eqref{eq: moduli} to a  parametrized continuation equation for which the analysis is well-known. See \cite[Section~2.2]{BO17} for similar situation in the symplectic homology setup.

Let $\widehat{\mathcal{M}}(Z_{z_-},Z_{z_+};-\tilde{f},g)$ be the space of \emph{parametrized} gradient flow lines on $S^N$ for $-\tilde{f}$ from $Z_{z_-}$ to $Z_{z_+}$, where $Z_{z_\pm}$ are critical $\Z_2$-pairs of $-\tilde{f}$. Since $g$ is Morse-Smale, the space $\widehat{\mathcal{M}}(Z_{z_-},Z_{z_+};-\tilde{f},g)$ is a smooth manifold of dimension $\ind(Z_{z_-};-\tilde{f})-\ind(Z_{z_+};-\tilde{f})$, where the index is given by $\ind(Z_z;-\tilde{f}):=\ind(z';-\tilde{f})$ for any $z'\in Z_z$. From now on, we think of
$$
\Gamma:=\widehat{\mathcal{M}}(Z_{z_-},Z_{z_+};-\tilde{f},g)
$$ as a  parameter space. Consider a parametrized data of pairs given by
\begin{equation}\label{eq: paradatawrap}
(H_v^s,J_v^s):=(H_{v(s)},J_{v(s)})\quad \text{for $s\in \R$ and $v\in \Gamma$}.
\end{equation}
Note that $(H_v^s,J_v^s)$ is independent of $s$ for $|s|\gg 1$. We define the \emph{parametrized moduli space of continuation trajectories}
$$
\mathcal{P}(Z_{p_-},Z_{p_+};\{H_v^s\},\{J_v^s\})
$$
as the space consisting of pairs $(u,v)$ with $u:\R\times [0,1]\to \widehat{W}$ and $v\in \Gamma$ satisfying
$$
\begin{cases}
	\p_su+J_{v}^{s,t}(u)(\p_t u -X_{H_{v}^s}(u))=0, \\
     u(s,j)\in \widehat{L}\quad\text{for $j=0,1$}, \\
	\displaystyle \lim_{s\to \pm \infty} (u(s,t),v(s))\in Z_{p_\pm}.
\end{cases}
$$
It  follows  immediately from \eqref{eq: paradatawrap} that $\mathcal{P}(Z_{p_-},Z_{p_+};\{H_v^s\},\{J_v^s\})=\widetilde{\mathcal{M}}(Z_{p_-},Z_{p_+};H,J)$.
\begin{theorem}\label{thm: equivwraptran}
	For generic $J\in \mathcal{J}_N^{\Z_2}$ the space $\mathcal{P}(Z_{p_-},Z_{p_+};\{H_v^s\},\{J_v^s\})$ is a smooth manifold of dimension $\mu(Z_{p_+})-\mu(Z_{p_-})$.	
\end{theorem}
\begin{proof}
We follow a standard argument in \cite{FHS} and  point out the different parts.  Although $J$ is $\Z_2$-invariant, we have much freedom to perturb it as the parameter space $\Gamma$ is additionally involved. We denote by $\mathcal{J}^{\Z_2,\ell}_N$ the completion of $\mathcal{J}^{\Z_2}_N$ with respect to the $C^\ell$-topology. Let $J\in \mathcal{J}^{\Z_2,\ell}_N$. We consider the section of a suitable Banach bundle with the factor $\mathcal B$ of the base a $W^{1,p}$-space and the fibers of ${\mathcal E}$ an $L^p$-space with $p >2$,
\begin{eqnarray*}
	\mathcal{F}:\mathcal{B}\times \Gamma &\longrightarrow& \mathcal{E} \\
	(u,v) &\longmapsto& \mathcal{F}_v(u):=\p_s u+J_v^{s,t}(u)(\p_t u-X_{H_v^s}(u)).
\end{eqnarray*}
Then
$$
\mathcal{F}^{-1}(0)=\mathcal{P}(Z_{p_-},Z_{p_+};\{H_v^s\},\{J_v^s\}).
$$
For the construction of such a Banach bundle we refer to \cite[Section~4.1]{BO_MB}, where the same construction appears in Morse--Bott symplectic homology. Our case is much simpler as the asymptotics are $\Z_2$-pairs which are 0-dimensional. In particular, we do not need to work with a weighted Sobolev space.
We now consider the linearized operator of $\mathcal{F}$ at $(u,v)\in \mathcal{F}^{-1}(0)$ given by
\begin{equation}
\begin{split}
&	D_{(u,v)}\mathcal{F}	 : T_u\mathcal{B}\oplus T_v\Gamma \longrightarrow  L^p(u^*T\widehat{W}) 	\\
&D_{(u,v)}\mathcal{F}(\zeta,\ell)=\underbrace{\nabla_s \zeta +J_v^s(\nabla_t \zeta-\nabla_t X_{H_v^s})+\nabla_\zeta J_v^s(\p_t u -X_{H_v^s})}_{=D_u\mathcal{F}_v(\zeta)}+\nabla_\ell \big(J_v^s(\p_t u-X_{H_v^s})\big),
\end{split}
\end{equation}
where $D_u\mathcal{F}_v$ is the usual Floer operator.
Note that $D_{(u,v)}\mathcal{F}$ is a Fredholm operator with index
\begin{eqnarray*}
	\ind D_{(u,v)}\mathcal{F} &=& \ind D_u\mathcal{F}_v+\dim \Gamma \\
	&=& \mu(c_+)-\mu(c_-)+\ind(z_-;-\tilde{f})-\ind(z_+;-\tilde{f})\\
	&=& \mu(p_+)-\mu(p_-).
\end{eqnarray*}
Now consider the section
\begin{eqnarray*}
	\mathcal{F}_\text{univ}:\mathcal{B}\times \Gamma\times \mathcal{J}^{\Z_2,\ell}_N&\longrightarrow& \mathcal{E} \\
	(u,v,J) &\longmapsto& \p_s u+J_v^s(u)(\p_t u-X_{H_v^s})
\end{eqnarray*}
and define  the universal moduli space to be the zero locus
$$
\mathcal{P}_\text{univ}:=\mathcal{F}_\text{univ}^{-1}(0).
$$
We claim that $\mathcal{P}_\text{univ}$ is a Banach manifold of class $C^\ell$. Given this, we can apply the parametric transversality theorem and Taubes' trick to complete the proof. 

To see the claim, consider the linearization of $\mathcal{F}_\text{univ}$ given by
\begin{equation}
	\begin{split}
 & D_{(u,v,J)}\mathcal{F}_\text{univ} :	 T_u\mathcal{B}\oplus T_v\Gamma\oplus T_J\mathcal{J}^{\Z_2,\ell}_N \longrightarrow  L^p(u^*T\widehat{W}) \\		
 & D_{(u,v,J)}\mathcal{F}_\text{univ}(\zeta,V,Y) = D_{(u,v)}\mathcal{F}(\zeta,V)+D_{(u,v,J)}\big(J_v^s(u)(\p_t u-X_{H_v^s})\big)(Y).
	\end{split}
\end{equation}
Since $D_{(u,v)}\mathcal{F}$ is Fredholm, the cokernel of $D_{(u,v,J)}\mathcal{F}_\text{univ}$ has a finite dimension. Hence, the image of $D_{(u,v,J)}\mathcal{F}_\text{univ}$ is closed. We claim that the image is dense. It suffices to show that the annihilator of $\im (D_{(u,v,J)}\mathcal{F}_\text{univ})$ is trivial. Suppose that there exists $\eta\in L^q(u^*T\widehat{W})$ with $1/p+1/q=1$, which annihilates the image of $D_{(u,v,J)}\mathcal{F}_\text{univ}$. Then $\eta$ is of class $C^\ell$ and 
\begin{equation}\label{eq: transann2}
	\int_{\R\times [0,1]}\langle \eta,Y_t(u)(\p_t u-X_{H_v^s}(u))\rangle ds\wedge dt=0\quad \text{for all $Y\in T_J\mathcal{J}^{\Z_2,\ell}_N$}.	
\end{equation}
Assume  by contradiction  that there is a point $z_0=(s_0,t_0)\in R(u)$ with  $\eta(z_0)\ne 0$ where $R(u)$ is the set of regular points defined in \cite[Section~4]{FHS}. Choose $Y_0\in \text{End}(T_{u(z_0)}\widehat{W},J_{t_0},\widehat{\ow})$ such that $\langle \eta,Y_{t_0}(u)(\p_t u-X_{H_v^s}(u)) \rangle>0$, see \cite[Section~8]{SZ}. Now we want to extend $Y_0$ to $\tilde{Y}\in T_J\mathcal{J}^{\Z_2,\ell}_N$ such that the left hand side of \eqref{eq: transann2} is strictly positive, which gives the contradiction. This is the place that requires a specific argument. This $\Z_2$-invariant extension can be obtained by a construction as in Example~\ref{ex: prodham}. To show that the extension $\tilde{Y}$ is independent of $z$ along each local slice, we refer to \cite[Remark~5.2]{FHS}. By the maximum principle, a Floer strip $u$ lies in a compact region in~$\widehat{W}$ so that we can find the extension $\tilde{Y}$ even under the requirement that $J$ is admissible. This proves the universal moduli space $\mathcal{P}_\text{univ}$ is a Banach manifold of class $C^\ell$ as claimed and proves the theorem.
\end{proof}

\subsubsection{Chain complex}
We define the \emph{$\Z_2$-equivariant chain groups} by
$$
CF_*^{\Z_2,N}(H,J):=\bigoplus_{Z_p\subset \mathcal{P}(H+\tilde{f})} \Z_2\langle Z_p\rangle
$$
with grading given by $|Z_p|:=-\mu(Z_p)-n/2\in \Z$. For generic $J\in \mathcal{J}_N^{\Z_2}$ we define the   \emph{$\Z_2$-equivariant differential} $\p^{\Z_2}:CF_*^{\Z_2,N}(H,J)\to CF_{*-1}^{\Z_2,N}(H,J)$ by the formula
$$
\p^{\Z_2}(Z_{p_-}):=\sum_{|Z_{p_-}|-|Z_{p_+}|=1}\#_{2}\mathcal{M}_{\Z_2}(Z_{p_-},Z_{p_+};H,J)\cdot Z_{p_+}.
$$
By a standard argument, we have $\p^{\Z_2}\circ \p^{\Z_2}=0$. We denote by 
$$
HF_*^{\Z_2,N}(H,J)
$$
the homology group associated to the chain complex $(CF_*^{\Z_2,N}(H,J),\p^{\Z_2})$. By the usual direct limit procedure, we obtain
$$
HW_*^{\Z_2}(L;W):=\lim_{\underset{N}{\longrightarrow}} \lim_{\underset{H}{\longrightarrow}} HF_*^{\Z_2,N}(H,J),
$$ 
where the limit $\displaystyle \lim_{\underset{H}{\longrightarrow}}$ is taken over admissible families of Hamiltonians using continuation maps, and the limit $\displaystyle \lim_{\underset{N}{\longrightarrow}}$ is taken using morphisms induced by the equivariant inclusions $S^N\hookrightarrow S^{N+1}$.
We call $HW_*^{\Z_2}(L;W)$ the \emph{$\Z_2$-equivariant wrapped Floer homology group} of $L$ in $W$.
A standard continuation argument shows that this homology is independent of the choice of data $(H,f,J,g)$.
\subsection{Periodic family of Floer data}\label{sec: periodicfamham} In this section, we choose a special kind of $\Z_2$-invariant admissible families of Hamiltonians, parallel to Section~\ref{sec: perMorcompl}. We write $\widetilde{\mathcal{T}}_j:\widehat{W}\times S^N\to \widehat{W}\times S^{N+1}$ for the embeddings defined in \eqref{eq: shiftemb}, with   $M$ being replaced by $\widehat{W}$.
\begin{definition}
A sequence $\{(H_N,J_N)\}$ of pairs with $H_N\in \mathcal{H}_N^{\Z_2}$ and $J_N\in \mathcal{J}_N^{\Z_2}$ is called \emph{periodic} if the following conditions hold for all $N\ge 1$:
\begin{itemize}
	\item $H_N=\widetilde{\mathcal{T}}_0^*H_{N+1}=\widetilde{\mathcal{T}}_1^*H_{N+1}$;
	\item $J_N=\widetilde{\mathcal{T}}_0^*J_{N+1}=\widetilde{\mathcal{T}}_1^*J_{N+1}$.
\end{itemize}
\end{definition}

\begin{remark}
To construct a periodic sequence $\{(H_N, J_N)\}$, we follow the ideas in Example~\ref{ex: prodham}, see also Remark~\ref{rem: periodicseq}.
\end{remark}
Suppose that a periodic sequence $\{(H_N,J_N)\}$ is given. The same procedure as in Section~\ref{subsec:eMH} leads to the following canonical identification, which is analogous to \eqref{eq: simplymorsecomplex}. One should be aware that the role of $\tilde{f}$ is replaced by $-\tilde{f}$ as we used the (positive) gradient flow of $\tilde{f}$ in the definition of the moduli spaces \eqref{eq: moduli}. Take $H_0:=H_{N,z^{(0)}}\in \mathcal{H}$ where $z^{(0)}$ is a critical point of $\tilde{f}$ of index 0, and a generic $J_0\in \mathcal{J}$. 
\begin{lemma}\label{lem: idenper}
There is a canonical identification
$$
CF^{\Z_2, N}_*(H_N, J_N) = \Z_2[w]/\{w^{N+1}=0\} \otimes_{\Z_2} CF_*(H_0,J_0)
$$
where $w$ is a formal variable with $|w| =1$.
\end{lemma}

\begin{proof}
Let $Z_{p}$ be a generator of $CF^{\Z_2, N}_*(H_N, J_N)$. As a representative of $Z_p$ we take $p = (w^jc, z^{(j)})$ where $z^{(j)}$ is a critical point of $-\tilde{f}$ with index $j$. Define the map
\begin{eqnarray*}
CF^{\Z_2, N}_*(H_N, J_N) &\rightarrow& \Z_2[w]/\{w^{N+1}=0\} \otimes_{\Z_2} CF_*(H_0,J_0)	\\
Z_p = [(w^jc, z^{(j)}] &\mapsto& w^j \otimes c.
\end{eqnarray*}
By the canonical choice of the representative of $Z_p$, this map is well-defined. The inverse is given by $w^j \otimes c \mapsto Z_p$, where $p = (w^jc, z^{(j)})$, which completes the proof.
\end{proof}

\subsection{Leray--Serre type spectral sequence}

As in Lemma~\ref{lem: permoduli},  the periodicity of $\{(H_N, J_N)\}$ gives rise to the following lemma.
\begin{lemma}
There is a one-to-one correspondence 
$$
\mathcal{M}_{\Z_2}(Z_{p_-},Z_{p_+}; H_N, J_N) \leftrightarrow \mathcal{M}_{\Z_2}(Z_{q_-}, Z_{q_+}; H_N, J_N)
$$
where $p_- = (w^jc_-, z^{(j)})$, $p_+ = (w^kc_+, z^{(k)})$, and $q_- = (w^{j-k}c_-, z^{(j-k)})$, $q_+ = (c_+, z^{(0)})$.
\end{lemma}
\noindent Under the identification in Lemma~\ref{lem: idenper}, the differential $\p^{\Z_2}$ reads
\begin{align*}
\p^{\Z_2}(Z_{p_-}) &= \sum_{Z_{p_+}} \#_{2} \mathcal{M}_{\Z_2}(Z_{p_-}, Z_{p_+}; H_N, J_N) \cdot Z_{p_+}\\
&= \sum_{k=0}^{j}\sum_{Z_{p_+}} \#_{2} \mathcal{M}_{\Z_2}(w^{j-k} \otimes c_-, c_+; H_N, J_N) \cdot w^k \otimes c_+.
\end{align*}
We introduce the following maps on $CF_*(H_0,J_0)$:
\begin{equation}\label{eq:mapztwo}
\phi_j(c_-) = \sum_{c_+}\#_2\mathcal{M}_{\Z_2}(w^j \otimes c_-, c_+; H_N, J_N) \cdot c_+.
\end{equation}
Then the $\Z_2$-equivariant differential  $\p^{\Z_2}$   satisfies, under the identification in Lemma~\ref{lem: idenper},
$$
\p^{\Z_2}(w^{\ell} \otimes c) = \sum_{j=0}^{\ell} w^{j} \phi_{\ell-j}(c) = \sum_{j=0}^{\ell} w^{\ell-j} \phi_{j}(c).
$$
Like in the Morse case (Section~\ref{sec: perMorcompl}), we take the direct limit of chain complexes
$$
\{(CF_*^{\Z_2,N}(H_N,J_N), \{\phi_j\}_{0\le j\le N})\}
$$
as $N\to \infty$ using the morphisms induced by the equivariant inclusions $S^N\hookrightarrow S^{N+1}$, and write
$$
CF_*^{\Z_2}(H,J)
$$
for the limit chain complex equipped with maps $\{\phi_j\}_{j\ge 0}$. Summarizing, we have:
\begin{proposition}
The wrapped Floer chain complex $CF_*(H_0,J_0)$ admits a $\Z_2$-complex structure given by $\{\phi_j\}$ defined in \eqref{eq:mapztwo}, and the resulting $\Z_2$-equivariant chain complex is exactly the complex $CF_*^{\Z_2}(H, J)$.  
\end{proposition}
A direct consequence from Lemma~\ref{lem: alglsss}, after taking direct limits over $H$, is the following.
\begin{corollary}[Leray--Serre type spectral sequence in $HW_*^{\Z_2}$] \label{cor: lsssfull}
There is a spectral sequence $\{E^r\}_{r\geq 0}$ converging to $HW_*^{\Z_2}(L;W)$ with first page given by
$$
E^1_{p,q} = \begin{cases} HW_q(L;W), & p \geq 0, \\ 0, & \text{otherwise.}
\end{cases}
$$
\end{corollary}

\subsection{Positive equivariant wrapped Floer homology} Since the equivariant differential $\p^{\Z_2}$ in principle counts parametrized Floer solutions, it is not obvious that the action decreases along Floer solutions. However, it turns out that we can obtain the action decreasing property under additional assumptions on the  Hamiltonians, see Lemma~\ref{lam: actdecreasing}. We shall then define the positive equivariant wrapped Floer homology $HW^{\Z_2, +}_*(L;W)$ which plays a crucial role for our applications in Section~\ref{secapplication}.
 
\subsubsection{Action filtration}\label{sec: actdecreas} Let $(H,J)$ be an admissible $\Z_2$-invariant family of Floer data. By definition, $H$ satisfies the inequality \eqref{eq: prod_estimate}. Let $(u, z)$ be a solution of the Floer equation \eqref{eq: moduli} from $Z_{p_-}$ to $Z_{p_+}$ with $p_\pm = (c_\pm, z_\pm)\in \mathcal{P}(H)$. Let $U$ be the union of local slices given in Section~\ref{sec:Produclikeham}.

\begin{lemma}[Action decreasing property]\label{lam: actdecreasing}
The action decreases along the solution $(u, z)$, i.e.,
$$
\mathcal{A}_{H+ \tilde f_N}(c_-, z_-) \geq \mathcal{A}_{H + \tilde f_N}(c_+, z_+).
$$
\end{lemma}

\begin{proof} 
Recall that
$$
\mathcal{A}_{H+\tilde f_N}(c, z) = \ell(c(1)) - \ell(c(0)) - \int_0^1 c^* \widehat{\lda} - \int_0^1 (H + \tilde f_N)(c(t), z) dt.
$$
We shall prove the lemma by showing that
$$
\frac{\p}{\p s} \mathcal{A}_{H+\tilde f_N} (u(s,\cdot), z(s)) \leq 0
$$
for all $s \in \R$. We first compute that 
\begin{align*}
 &\frac{\p}{\p s} \left\{ \ell(u(s, 1)) - \ell(u(s, 0)) - \int_0^1 u(s, \cdot)^* \widehat{\lda}  \right\} \\
	=\;&  \widehat{\lda}(\p_s u(s, 1)) - \widehat{\lda}(\p_s u(s, 0)) - \int_0^1  d (\widehat{\lda}(\p_s u)) - \int_0^1   \widehat{\ow}(\p_s u, \p_t u ) dt   \\
	=\;& \widehat{\lda}(\p_s u(s, 1)) - \widehat{\lda}(\p_s u(s, 0)) -\{\widehat{\lda}(\p_s u(s, 1)) - \widehat{\lda}(\p_s u(s, 0))\} - \int_0^1   \widehat{\ow}(\p_s u, \p_t u ) dt   \\
	=\;& - \int_0^1 \widehat{\ow}(\p_s u, \p_t u) dt.
\end{align*}
We then compute the remaining term as follows:
\begin{align*}
	&   \frac{\p}{\p s} \int_0^1 (H + \tilde f_N)(u(s, t), z(s)) dt\\
	=\;&  \frac{\p}{\p s} \int_0^1 H_{z(s)}(u(s, t)) + \tilde f_N (z(s)) dt \\
	=\;&   \frac{\p}{\p s} \int_0^1 H_{z(s)}(u(s, t)) dt + \int_0^1 d \tilde f_N (z(s)) \cdot \nabla \tilde{f}_N(z(s)) dt \\
	=\;&   \int_0^1 dH_{z(s)}(\p _s u(s, t)) dt  + \int_0^1 \p_z H_z(u(s, t)) \cdot \nabla 
	\tilde{f}_N dt + \|\nabla \tilde f_N\|^2\\
	=\;&  - \int_0^1 \widehat{\ow}(\p_s u, X_{H_{z(s)}}(u))dt  + \int_0^1 \p_z H_z(u(s, t)) \cdot \nabla 
	\tilde{f}_N dt + \|\nabla \tilde f_N\|^2.
\end{align*}
Combining these two identities we obtain
\begin{align*}
  & \frac{\p}{\p s} \mathcal{A}_{H+\tilde f_N} (u(s,\cdot), z(s))\\ 
	=\;& -\int_0^1 \widehat{\ow}(\p_s u , \p_t u -X_{H_{z(s)}}) dt - \int_0^1 \p_z H_z(u(s, t)) \cdot \nabla 
\tilde{f}_N dt - \|\nabla \tilde f_N\|^2\\
	=\;& -\int_0^1 \|\p_s u\|^2 dt  - \int_0^1 \p_z H_z(u(s, t)) \cdot \nabla 
	\tilde{f}_N dt - \|\nabla \tilde f_N\|^2. 
\end{align*}
Now the claim follows from the following  observations below:
\begin{itemize}
\item In $U$, since $H_z$ is constant in $z$, we have $\p_z H_z = 0$.
\item By the inequality \eqref{eq: prod_estimate}, on $S^N\setminus U$ we have 
$$
 - \int_0^1 \p_z H_z(u(s, t)) \cdot \nabla 
\tilde{f}_N dt - \|\nabla \tilde f_N\|^2  \leq 0.   
$$
\end{itemize}
This completes the proof.
\end{proof}

\subsubsection{Positive equivariant chain complex} Let $(H, J)$ be an admissible $\Z_2$-invariant Floer datum. For $\tau \in \R$ define the subgroup  
$$
CF^{\Z_2, N, < \tau}_*(H, J) \subset CF^{\Z_2, N}_*(H, J)
$$ which is generated by elements with action value less than $\tau$. By Lemma~\ref{lam: actdecreasing}, this subgroup forms a subcomplex of $(CF^{\Z_2, N}_*(H, J), \p^{\Z_2})$. 

Let $\epsilon > 0$ be smaller than the lengths of any Reeb chord on the boundary $\p W$. We define the \emph{positive equivariant chain complex} as the quotient
$$
CF^{\Z_2, +, N}_*(H, J) : = CF^{\Z_2, N}_*(H, J)/ CF^{\Z_2, N, < \epsilon}_*(H, J)
$$
with the induced differential, still denoted by $\p^{\Z_2}$. Its homology is denoted by $HW^{\Z_2, +, N}_*(H, J)$. We take the direct limit over $N$ and $H$, and call the resulting homology the \emph{positive equivariant wrapped Floer homology}, denoted by $HW^{\Z_2, +}_*(L;W)$.

\subsubsection{A tautological exact sequence in $HW^{\Z_2}_*$} The action filtration on the equivariant chain complex $CF^{\Z_2}_*(H, J)$ yields the following short exact sequence of chain complexes:
$$
0\to CF_*^{\Z_2,N, < \epsilon}(H, J) \to CF_*^{\Z_2,N}(H, J) \to CF_*^{\Z_2, +,N}(H, J) \to 0 .
$$ 
The induced long exact sequence, after taking the direct limit, is 
$$
\cdots \to HW_*^{\Z_2, <\epsilon}(L;W) \to HW_*^{\Z_2}(L;W) \to HW_*^{\Z_2, +}(L;W) \to \cdots.
$$
Assuming that $H(\cdot,z)$ is $C^2$-small in the interior of $W$ for all $z\in S^N$, a standard argument~\cite[Proposition~1.3]{Vit} yields 
$$
HW_*^{\Z_2, <\epsilon}(L;W) \cong H^{\Z_2}_{*+n}(L, \p{L}),
$$
where the $\Z_2$-action on $L$ is  induced by the involution $\I$ on $W$. We thus obtain the following tautological exact sequence for $HW_*^{\Z_2}(L;W)$. 

\begin{proposition}\label{prop: tlspositive}
There exists a long exact sequence
$$
\cdots \to H_{*+n}^{\Z_2}(L,\p L) \to HW_*^{\Z_2}(L;W) \to HW_*^{\Z_2,+}(L;W) \to H_{*+n-1}^{\Z_2}(L,\p L) \to \cdots.
$$
\end{proposition}  

\subsubsection{Action filtration via $\mathcal{A}_{H_0}$} Recall that $(H,J)$ has the properties described in Section~\ref{sec: periodicfamham}.
The positive part $CF^{\Z_2, +, N}_*(H, J)$ can also be simplified as in Lemma~\ref{lem: idenper} using a periodic family. However, the identification may not restrict to the positive part since positivity of the action $\mathcal{A}_{H+\tilde{f}}(Z_p)$ does not in general imply positivity of $\mathcal{A}_{H_0}(c)$, where  $w^j \otimes c$ is the generator corresponding to $Z_p$. Notice that our previous constructions and results also hold for $\tilde f$ replaced by $\delta \tilde{f}$ with any $\delta >0$.

\begin{lemma}\label{lem: actiongovern} Under the assumptions and notations in Lemma~\ref{lem: idenper}, the action $\mathcal{A}_{H+\delta\tilde{f}}(w^jc, z^{(j)})$ is arbitrarily close to $\mathcal{A}_{H_0}(c)$, provided that $\delta>0$ is sufficiently small.
\end{lemma}

\begin{proof}
Note that the action $\mathcal{A}_{H+\delta\tilde f}$ is defined by
$$
\mathcal{A}_{H+\delta\tilde f}(w^jc, z^{(j)}) = \ell(c(1)) - \ell(c(0)) - \int_0^1 c^* \widehat{\lda} - \int_0^1 (H +\delta \tilde f)(c(t), z^{(j)}) dt.
$$
Since the family $H$ is periodic, we may replace $H$ by $H_0$ as
$$
\mathcal{A}_{H_0}(c) = \ell(c(1)) - \ell(c(0)) - \int_0^1 c^* \widehat{\lda} - \int_0^1 H_0(c(t)) dt.
$$
Now the difference between $\mathcal{A}_{H+\delta \tilde{f}}(w^jc, z^{(j)})$ and $\mathcal{A}_{H_0}(c)$ is bounded by
$$
\left|\int_0^1 \delta\tilde f(z^{(j)}) dt\right| \leq |\delta \tilde f(z^{(j)})|\le \delta N.
$$
Since $N$ is fixed, the lemma follows.
\end{proof}

\begin{remark} \label{rmk: final conditions on H and f}
We have made several assumptions on the family $H$ and $\delta\tilde f$ so far. Namely,
\begin{itemize}
\item The family $H$ is admissible;
\item $\{H_N\},\{f_N\}$ are periodic;
\item $\delta>0$ is sufficiently small in the sense of Lemma~\ref{lem: actiongovern}. 
\end{itemize}
By abuse of the notation, we write $\tilde{f}$ instead of $\delta \tilde{f}$. All these conditions can be achieved compatibly by the construction in Example~\ref{ex: prodham}. First, choose $\delta>0$ sufficiently small in the sense of Lemma~\ref{lem: actiongovern}. Then, take $H_0$ which is sufficiently small in the region where Hamiltonian chords  occur. Finally, construct a family $H = \{H_z\}$ satisfying the inequality \eqref{eq: prod_estimate} following the scheme in Example~\ref{ex: prodham}.
\end{remark}
From now on, unless stated otherwise, we assume all the properties listed in Remark~\ref{rmk: final conditions on H and f}. Then the identification in Lemma~\ref{lem: idenper} restricts to the positive part. Moreover, we can follow the same recipe as in Section~\ref{sec: periodicfamham} to provide a $\Z_2$-complex structure for the positive parts $HW^{\Z_2, +}_*(L;W)$ and $HW_*^+(L;W)$. As a result we have an analogous Leray--Serre type spectral sequence in positive equivariant wrapped Floer homology:

\begin{proposition}[Leray--Serre type  spectral sequence in $HW^{\Z_2, +}_*$]
There is a spectral sequence $\{E^r\}_{r\geq 0}$ converging to $HW_*^{\Z_2,+}(L;W)$ with first page 
$$
E^1_{p,q} = \begin{cases} HW^+_q(L;W), & p \geq 0, \\ 0, & \text{otherwise.}
\end{cases}
$$
\end{proposition}

\subsection{Morse--Bott spectral sequence} \label{sec: mbss} For our applications we need to know that the positive equivariant wrapped Floer homology is generated by $\Z_2$-pairs of Reeb chords; this is not entirely clear from the definition. We do this by constructing a Morse--Bott spectral sequence in $HW^{\Z_2, +}_*$. A similar technique was used in \cite{Gutt} for positive equivariant symplectic homology.

Suppose that the contact form $\alpha$ on the boundary $\p W$ is of {\it Morse--Bott type}, for the definition we refer to \cite{KKL}. To have the well-defined Maslov indices of Reeb chords we assume that the Maslov class $\mu_\mathcal{L}:\pi_2(\Sigma, \mathcal{L})\to \Z$ of the Legendrian $\mathcal{L}=\Fix(\rho)$ vanishes and that $\pi_1(\Sigma,\mathcal{L})=0$.

We arrange the action spectrum  of Reeb chords as  
\begin{equation}\label{eq: spectrum}
\Spec(\Sigma, \alpha, \mathcal{L}) = \{0 < T_1 < T_2 < \cdots\}. 
\end{equation}
For each $T \in \Spec(\Sigma, \alpha, \mathcal{L})$, denote the Morse--Bott submanifold  of length $T$ by  
$$
\mathcal{L}_T : = \{y \in \mathcal{L} \;|\; \phi_R^T(y) \in \mathcal{L}\},
$$
where $\phi_R^t$ is the flow of the Reeb vector field associated to the contact form $\alpha$. This forms a closed submanifold of $\mathcal{L}$ by definition of the Morse--Bott condition. For each Morse--Bott submanifold $\mathcal{L}_{T}$ we define the Maslov index $\mu(\mathcal{L}_{T})$ as the Maslov index of a Reeb chord starting at $\mathcal{L}_{T}$. Note that the Maslov index can be different at different components of $\mathcal{L}_T$. Starting from a Morse function $h_0: \mathcal{L}_T \rightarrow \R$, we form a periodic family of Morse functions $h = \{h_z\}$ in the sense of Section~\ref{sec: perMorcompl} following the recipe given in Example~\ref{ex: prodham}. This allows us to consider the $\Z_2$-complex structure on the Morse chain complex $(CM_*(h_0, g_0), \phi_j)$, together with the identification \eqref{eq: simplymorsecomplex}:
$$
CM_*^{\Z_2, N}(h, g) = \Z_2[w]/\{w^{N+1}=0\} \otimes_{\Z_2} CM_*(h_0, g_0).
$$
Denote the corresponding homology groups by $HM_*(\mathcal{L}_T)$ and $HM_*^{\Z_2, N}(\mathcal{L}_T)$. We do the same procedure for all Morse--Bott submanifolds; take a Morse function, form a periodic family, and so~on.

Let $H_0: \widehat W \rightarrow \R$ be an admissible Hamiltonian such that $H_0$ is $C^2$-small on $W$ and   depends only on the variable $r$ for $(r,y)\in (1-\epsilon, \infty)$. Since $\alpha$ is of Morse--Bott type, $H_0$ is automatically Morse--Bott. Perturb $H_0$ in a standard way using the Morse function $h_0$ on each submanifold $\mathcal{L}_T$; for an explicit description of the perturbation we refer to \cite{KvK}. For notational convenience, we denote the perturbed function still by $H_0$.  

Following the recipe in Example~\ref{ex: prodham}, we construct a periodic family $\{(H_N,J_N)\}_{N\ge 1}$ starting from $H_0$, which satisfies all the conditions in Remark~\ref{rmk: final conditions on H and f}. This yields the identification 
$$
CF^{\Z_2, +, N}_*(H_N, J_N) = \Z_2[w]/\{w^{N+1}=0\} \otimes_{\Z_2} CF_*^+(H_0, J_0)
$$
as in Lemma~\ref{lem: idenper} and Lemma~\ref{lem: actiongovern}. We also have the corresponding $\Z_2$-complex structure, say $(CF_*^+(H_0, J_0), \p, \phi_j)$. By Lemma~\ref{lam: actdecreasing} and Lemma~\ref{lem: actiongovern}, we can filter the complex $CF_*^{\Z_2, +, N}(H_N,J_N)$ by action values according to the spectrum \eqref{eq: spectrum}. The $E^0_{p*}$-terms of the resulting spectral sequence is freely generated by those generators of $CF_*^{\Z_2, +, N}(H_N, J_N)$ whose action is (approximately) equal to $T_p$. 

Note that, by Lemma~\ref{lem: actiongovern}, the generators in $CF_*^{\Z_2, +, N}(H_N, J_N)$ of action (approximately) $T_p$ correspond to the generators in $CF_*^+(H_0, J_0)$ of action (approximately) $T_p$. Consequently, we have
$$
E^0_{p*} = \Z_2[w]/\{w^{N+1}=0\} \otimes_{\Z_2} CF_*^{loc}(\mathcal{L}_{T_p}),
$$
where $CF_*^{loc}(\mathcal{L}_{T_p}; H_0, J_0)$ denotes the local chain complex near $\mathcal{L}_{T_p}$. It is well-known that the generators and Floer trajectories for $CF_*^{loc}(\mathcal{L}_{T_p}; H_0, J_0)$ are in one-to-one correspondence to those for the Morse chain complex $CM_*(h_0, g_0)$ up to a degree shift; see  \cite{Pozniak} for example. As a result, we have the following  identifications 
$$
E^0_{p*} = \Z_2[w]/\{w^{N+1}=0\} \otimes CF_*^{loc}(\mathcal{L}_{T_p}) =  \Z_2[w]/\{w^{N+1}=0\} \otimes CM_*(h_0,g_0) = CM^{\Z_2, N}_*(h_N,g_N).
$$
Actually, the correspondence between Morse and Floer trajectories is also valid for parametrized ones. This allows us to identify the $\Z_2$-complex structures on $CF_*^{loc}(\mathcal{L}_{T_p})$ and $CM_*(h_0, g_0)$. Therefore, the differential $d^0: E^0 \rightarrow E^0$ of the 0-page is identified with the $\Z_2$-equivariant differential~$\p^{\Z_2}$ of $CM_*^{\Z_2, N}(h_N,g_N)$. It follows that, up to a degree shift,
$$
E^1_{p,*} = HM^{\Z_2, N}_*(\mathcal{L}_{T_p}).
$$
The degree shift here is the same as that of the non-equivariant part; we refer to the Morse--Bott spectral sequence for the non-equivariant case in \cite{KKL}. After taking the direct limits, we conclude the following.

\begin{theorem}[Morse--Bott spectral sequence in $HW_*^{\Z_2, +}$]
There is a spectral sequence $\{E^r\}_{r\ge 0}$ converging to $HW_*^{\Z_2, +}(L;W)$ whose first page is given by
\begin{equation}\label{eq: mbss}
E^1_{p,q} = 
\begin{cases}
H_{p+q -\sh(\mathcal{L}_{T_p})+\frac{1}{2}\dim L}^{\Z_2}(\mathcal{L}_{T_p}),  & \text{$p>0$}, \\
0, & \text{otherwise,}	
\end{cases}
\end{equation}
where $\sh(\mathcal{L}_{T_p}) := \mu(\mathcal{L}_{T_p}) -\frac{1}{2}(\dim \mathcal{L}_{T_p} -1)$.
\end{theorem}

\subsection{Generators of $HW^{\Z_2, +}_*$ are Reeb chords} \label{sec: genetorsof eHW} At the beginning of the construction in Section~\ref{sec: mbss}, we may assume that $\alpha$ is \emph{chord-non-degenerate}, i.e., every Reeb chord is non-degenerate. In this case, because of the $\Z_2$-symmetry, the non-degeneracy means that each Morse--Bott component of $\mathcal{L}_{T_p}$ consists of two points. The corresponding Reeb chords form a $\Z_2$-pair. The $\Z_2$-equivariant homology of the Morse--Bott submanifold $\mathcal{L}_{T_p}$ is therefore given by
$$
H_*^{\Z_2}(\mathcal{L}_{T_p};\Z_2) = \begin{cases}
\Z_2^r & \text{for $*=0$},\\
0 & \text{otherwise}.
\end{cases}
$$
Here $r$ denotes the number of $\Z_2$-pairs in $\mathcal{L}_{T_p}$. It follows that the $E^1$-page of the Morse--Bott spectral sequence \eqref{eq: mbss} is expressed by $\Z_2$-pairs of Reeb chords. We conclude the following.
\begin{corollary}\label{cor: genwhfpsotive}
Suppose the contact form $\alpha$ on $\p W$ is chord-non-degenerate. If $HW^{\Z_2,+}_k(L;W)\ne 0$ for some $k \in \Z$, then there exists a $\Z_2$-pair of Reeb chords, say $\{c,\I(c)\}$, whose Maslov index is given by
\begin{equation}\label{eq: orbitexist}
\mu(c)-\frac{n}{2}+\frac{1}{2}=k.
\end{equation}
\end{corollary}
When $\I$ is an anti-symplectic involution, a $\Z_2$-pair of Reeb chords corresponds to a symmetric periodic Reeb orbit, see Figure \ref{fig:symmintro}. Hence,  $HW^{\Z_2,+}_*(L;W)\ne 0$ implies the existence of a symmetric periodic Reeb orbit.


\subsection{Computing positive equivariant wrapped Floer homology}\label{sec:compuewfgh} The vanishing property of wrapped Floer homology in Section~\ref{sec: vanHW} allows us to compute the positive equivariant wrapped Floer homology in many cases. We use the algebraic tools developed in Section~\ref{sec: ewf}. 

\begin{proposition}
Suppose that $HW_*(L; W)$ vanishes. Then the positive equivariant wrapped Floer homology group is given by
\begin{equation}\label{eq: computepewfh}
HW_*^{\Z_2, +}(L; W) \cong H_{* + n -1}^{\Z_2}(L, \p L) 
\end{equation}
\end{proposition}

\begin{proof}
By the Leray--Serre type spectral sequence in Corollary~\ref{cor: lsssfull}, if $HW_*(L;W)$ vanishes, then $HW^{\Z_2}_*(L;W)$ also vanishes. Recall from Proposition~\ref{prop: tlspositive} that we have a long exact sequence
$$
\cdots \to H_{*+n}^{\Z_2}(L,\p L) \to HW_*^{\Z_2}(L;W) \to HW_*^{\Z_2,+}(L;W) \to \cdots
$$
from which it follows that
$$
HW_*^{\Z_2, +}(L; W) \cong H_{* + n -1}^{\Z_2}(L, \p L).
$$
This completes the proof.
\end{proof}
Since $HW_*(L;W)$ vanishes in the displaceable case (Corollary~\ref{cor: displaceable}), an immediate corollary is the following.
\begin{corollary}\label{cor: diseqhom}
If $\Sigma$ is displaceable from $\widehat L$ in $\widehat W$, then $HW^{\Z_2, +}_*(L;W)$ is given by \eqref{eq: computepewfh}.
\end{corollary}
Note that the $\Z_2$-action on $(L, \p L)$ is induced  by the involution $\I$ on the ambient space $W$. So the further computation of the equivariant homology $H_{*}^{\Z_2}(L, \p L)$ depends on $\I$.
\subsubsection{Case of symplectic involutions}
We derive a simple version of a localization property on equivariant homology theory which is enough for our applications. 

\begin{proposition}\label{prop: nonvanishingeqhom}
Suppose that the $\Z_2$-action by $\I$ is  free on the boundary $\p L$ and that the fixed point set $\Fix(\I|_L)$ is non-empty. Then $H_{k}^{\Z_2}(L, \p L)\ne 0$ for $k \geq n$.
\end{proposition}    
This follows directly from the following more general observation.

\begin{lemma}
Let $X^n$ be a compact manifold with boundary equipped with an involution $I$. Suppose that $I$ acts freely on $\p X$ and that $\Fix(I)$ is non-empty. Then $H^{\Z_2}_k(X, \p X)\ne 0$ for $k \geq n$
\end{lemma}	
	

\begin{proof}
We first claim that $H_k^{\Z_2}(X)\ne 0$  for $k \geq 0$. Take a fixed point $x \in \Fix(I)\subset  X\setminus \p X$. We consider the inclusion $i:\{x\}\hookrightarrow X$. This map is clearly $\Z_2$-equivariant, so we obtain a section $\overline{i}:\{x\}_\text{Borel}=\R P^\infty \to X_\text{Borel}$ of the bundle $\pi:X_\text{Borel}\to \R P^\infty$. Since $\pi \circ \overline{i}=\id$, it follows that the induced map
$$
\overline{i}_*:H_*(\R P^\infty)\to H_*(X_\text{Borel})= H_*^{\Z_2}(X)
$$
is injective. In particular $H_k^{\Z_2}(X)\ne0$ for $k \geq 0$ as we claimed.

Consider the long exact sequence in homology for the pair $(X_\text{Borel}, \p X_\text{Borel})$, namely
$$
\cdots\rightarrow H_k(\p X_\text{Borel}) \rightarrow H_k(X_\text{Borel}) \rightarrow H_k(X_\text{Borel}, \p X_\text{Borel}) \rightarrow H_{k-1}(\p X_\text{Borel}) \rightarrow\cdots.
$$
Since the $\Z_2$-action is free on the boundary $\p X$, we have
$$
H_k^{\Z_2}(\p X) = H_k(\p X_\text{Borel}) = H_k(\p X/\Z_2).
$$
In particular, since $\dim (\p X/\Z_2) = n-1$, we have $H_k(\p X_\text{Borel}) = 0$ for $k \geq n$. Applying this together with the previous claim to the long exact sequence, we find that $H_k(X_\text{Borel}, \p X_\text{Borel}) \cong H_k(X_\text{Borel})\ne 0$ for $k \geq n+1$, and the inclusion $0 \neq H_n(X_\text{Borel}) \hookrightarrow H_n(X_\text{Borel}, \p X_\text{Borel})$ is injective so that also $H_n(X_\text{Borel}, \p X_\text{Borel})\ne 0$. This completes the proof.
\end{proof}

Applying Proposition~\ref{prop: nonvanishingeqhom} to the displaceable case computed in Corollary~\ref{cor: diseqhom}, we conclude the following.

\begin{theorem}\label{cor: sympinvcomp}
Suppose that $\Sigma$ is displaceable from $\widehat L$ in $\widehat W$ and that $\I$ is an exact symplectic involution on $W$. If the  involution $\I|_L$ acts freely on the boundary and the fixed point set $\Fix(\I|_L)$ is non-empty, then $HW^{\Z_2, +}_k(L;W)\ne 0$ for all $k \geq 1$.
\end{theorem}

\begin{example}
Consider the real Lagrangian $\R^n$ which is the fixed point set of complex conjugation on $\C^n$. Consider the antipodal map $\I$ as an exact symplectic involution on $\C^n$. Note that $S^{2n-1} = \p B^{2n}$ is displaceable from $\R^n$ in $\C^n$,  and that the fixed point set $\Fix(\I|_{B^n})$ of the antipodal map on $B^n$ is a singleton set (the origin). It follows from Theorem~\ref{cor: sympinvcomp} that $HW_k^{\Z_2, +}(\R^n; \C^n)\ne 0$ for all $k\ge 1$.
\end{example}
\subsubsection{Case of anti-symplectic involutions}
Since the anti-symplectic involution $\I$ on $W$ restricts to the identity map on $L$, we have
$$
H_*^{\Z_2}(L,\p L)\cong H_*(L,\p L)\otimes_{\Z_2} H_*(\R P^\infty).
$$
By Corollary~\ref{cor: diseqhom} and Lefschetz duality, we obtain the following.
\begin{theorem}
	\label{cor: antisympcomp}
	Suppose that $\Sigma$ is displaceable from $\widehat{L}$ in $\widehat{W}$ and that $L$ is a real Lagrangian, i.e., $\I$ is exact anti-symplectic. Then we have
	$$
	HW_*^{\Z_2,+}(L;W)\cong \Big(H_*(L,\p L)\otimes_{\Z_2} H_*(\R P^\infty) \Big)[n-1],
	$$
	where $[n-1]$ means the downward shift in grading by $n-1$. In particular, if $L$ is orientable, then $HW_k^{\Z_2,+}(L;W)\ne 0$ for all $k\ge 1$.
\end{theorem}

\begin{example}
	Consider again the real Lagrangian $\R^n$ in $\C^n$ and the complex conjugation $\I$ which is an anti-symplectic involution on $\C^n$. Then one checks that
	$$
	H_*(L,\p L)\cong \tilde{H}_*(S^n)=\begin{cases}
		\Z_2 & \text{if $*=n$}, \\
		0 &\text{otherwise}.
	\end{cases}
	$$
By Theorem~\ref{cor: antisympcomp}, $HW_k^{\Z_2,+}(\R^n;\C^n)\ne 0$ for all $k\ge 1$.
\end{example}

\section{On the minimal number of  symmetric periodic Reeb orbits}\label{secapplication}
  Let $(\Sigma,\alpha,\rho)$ be a real contact manifold. Assume that $\mathcal{L}:=\Fix(\rho)$ is nonempty and that the Maslov class $\mu_{\mathcal{L}}:\pi_2(\Sigma,\mathcal{L})\to \Z$ of the Legendrian $\mathcal{L}$ vanishes.   Abbreviate by $R=R_{\alpha}$ the associated Reeb vector field and by $\xi=\text{ker} (\alpha)$ the contact structure. Since the involution $\rho$ is anti-contact, one can   associate to each  Reeb chord  $c : [0,T] \rightarrow \Sigma$  a \textit{symmetric periodic orbit}
\begin{equation*}
c^2 (t):= \begin{cases} c(t) &  \; t \in [0,T], \\ \rho \circ c(2T-t) & \; t\in [T,2T].     \end{cases}
\end{equation*}
We apply the equivariant wrapped Floer homology to the study of the minimal number of symmetric periodic Reeb orbits on the real contact manifold $(\Sigma, \alpha,   \rho)$. In order to do this, we will develop an index theory on iterations of symmetric periodic Reeb orbits.



\subsection{Index of symmetric periodic Reeb orbits\nopunct}\label{sec:indexsym}

\subsubsection{The Robbin-Salamon index} \label{sec: RSindex}
We briefly recall the definition of the Robbin-Salamon index of a path of Lagrangians. For more details we refer to \cite{RSindex} and  \cite[Chapter 10]{book}.



Let $\mathscr{L}(n)$ be the Lagrangian Grassmannian, i.e., the space of Lagrangian subspaces of $(\C^n ,\omega_0)$, where $\omega_0= \tfrac{i}{2}\sum_{j=1}^n dz_j \wedge d \overline{z}_j$ is the standard symplectic form. In coordinates $z_j = q_j + ip_j$, the form $\omega_0$ equals $\sum_{j=1}^n dq_j \wedge dp_j$. We fix $\Lambda_0 \in \mathscr{L}(n)$.   Without loss of generality we may assume that $\Lambda_0 =\R^n \times \left\{ 0 \right\}  \subset \C^n$. Any Lagrangian complement $\Lambda_1$, i.e., $\Lambda_1 \in \mathscr{L}(n)$ such that $\Lambda_0 \cap \Lambda_1 = \left \{ 0 \right \}$, can then be written as
\begin{equation}\label{eq:lagcom}
\Lambda_1= \left \{  (Bx, x) : x \in \R^n  \right\}
\end{equation}
for some real symmetric $n\times n$ matrix $B$. Choose a smooth path $\Lambda: (-\epsilon, \epsilon) \rightarrow \mathscr{L}(n)$ such that 
$$
\Lambda(t) = \left \{    (x,A(t)x) : x\in \R^n     \right\},
$$
where $A$ is a path of symmetric $n \times n$ matrices satisfying $A(0)=0$ so that $\Lambda(0)=\Lambda_0$. Abbreviate ${\widehat{\Lambda}}=\dot{\Lambda}(0)$. Since $\Lambda_0 \cap \Lambda_1 =\left \{0 \right\}$, for each  $(x,0)  \in \Lambda_0$ there exists a unique $w_x(t) \in \Lambda_1$ such that $(x,0) + w_x(t) \in \Lambda(t)$, provided that $\epsilon>0$ is small enough. We then define the quadratic form
$$
Q^{\Lambda_0,\Lambda_1}_{\Lambda(t)} ( (x,0)) := \omega_0 ( (x,0), w_x(t) ).
$$
By means of \eqref{eq:lagcom} for each $ t \in (-\epsilon, \epsilon)$  we find a unique $ y(t)  \in \R^n$ such that $w_x(t)=(By(t), y(t))$.   It follows that
\begin{equation}\label{eq:yA}
(x,0) + w_x(t) = ( x + By(t), y(t))\in \Lambda(t) \quad \Rightarrow \quad y(t) = A(t)(x+By(t) ).
\end{equation}
Then the quadratic form $Q^{\Lambda_0,\Lambda_1}_{\Lambda(t)} $ becomes 
$$
Q^{\Lambda_0,\Lambda_1}_{\Lambda(t)} ( (x,0)) = \omega_0((x,0), (By(t), y(t) ) ) = \langle x, y(t)\rangle.
$$
Differentiating with respect to $t$ at $t=0$ gives rise to
\[
Q^{\Lambda_0,\Lambda_1}_{\widehat{\Lambda}} ( (x,0)):= \frac{d}{dt}\bigg|_{t=0} Q^{\Lambda_0,\Lambda_1}_{\Lambda(t)} ( (x,0)) =  \langle x, \dot{y}(0)\rangle
= \langle x, \dot{A}(0)x\rangle,
\]
where the last equality follows from \eqref{eq:yA} together with $y(0)=0$. Note that the expression of $Q^{\Lambda_0,\Lambda_1}_{\widehat{\Lambda}} $ does not depend on the choice of $B$, or equivalently, $\Lambda_1$. Thus, we can write $ Q_{\widehat{\Lambda}} =Q^{\Lambda_0,\Lambda_1}_{\widehat{\Lambda}} $.

We are ready to define the Robbin-Salamon index. Fix $V\in \mathscr{L}(n)$ and   a path $\Lambda:[0,T]\rightarrow \mathscr{L}(n)$. For each $ t \in [0,T]$, the \textit{crossing form} associated to $\Lambda$ and $V$ is defined to be
$$
C(\Lambda,V,t):= Q_{\dot{\Lambda}(t)} |_{V\cap \Lambda(t)}. 
$$
If $C(\Lambda,V,t)$ is nonsingular, then $t$ is called a \textit{regular crossing}. Suppose that the path $\Lambda$ admits only regular crossings. Then the \textit{Robbin-Salamon index} of $\Lambda$ with respect to $V$ is defined to be
$$
\mu_{\RS}(\Lambda, V) := \frac{1}{2}\text{sgn} C(\Lambda, V,0) + \sum_{\substack{t: \text{ crossing} \\ 0<t<T}} \text{sgn}C(\Lambda,V,t)+\frac{1}{2}\text{sgn} C(\Lambda, V,T).
$$
Since regular crossings are isolated, the sum is finite. This index has the  following properties: 
\begin{itemize}
\item (Homotopy)  Two paths $\Lambda_0$, $\Lambda_1$ are homotopic with same endpoints if and only if they have the same Robbin-Salamon index.
\item (Catenation) For $a<c<b$, it holds that $\mu_{\RS}(\Lambda|_{[a,c]}, V) +\mu_{\RS}(\Lambda|_{[c,b]}, V) =\mu_{\RS}(\Lambda|_{[a,b]}, V)$.
\item (Reversal) If $\Lambda:[0,1]\rightarrow \mathscr{L}(n)$ is a path of Lagrangians, then $\mu_{\RS}(\Lambda, V) =- \mu_{\RS}(\overline{\Lambda}, V)$, where $\overline{\Lambda}(t):=\Lambda(1-t).$
\item (Naturality) If $\Phi$ is a path of symplectic matrices and if $V_0, V_1 \in \mathscr{L}(n)$, then $\mu_{\RS}(\Phi V_0, V_1) = - \mu_{\RS}(\Phi ^{-1}V_1, V_0)$. If $\Phi$ is anti-symplectic, we have  $\mu_{\RS}(\Phi V_0, V_1) =  \mu_{\RS}(\Phi^{-1} V_1, V_0)$.
\end{itemize}

\subsubsection{The spectral flow} \label{secspectral} Let 
$
\mathcal{S} = C^{\infty}([0,T], {\rm Sym}(2n) )
$
be the space of smooth paths of symmetric $ 2n  \times  2n $ matrices endowed with the metric
$$
d(S_0 , S_1 ) = \int_0^T || S_0(t) - S_1(t) ||\ dt, \quad S_0 , S_1 \in \mathcal{S}.
$$
We introduce the   Hilbert space
$$
W_{\R^{n}}^{1,2}( [0,T], \C^{n}) = \left \{  v \in W^{1,2}([0,T],\C^{n} ) \; : \; v(0), v(T) \in \R^{n}          \right \}.
$$ 
For $ S \in \mathcal{S}$, consider the bounded linear operator between Hilbert spaces
\begin{equation}\label{eq:buondedlinear}
L_S : W_{\R^{n} }^{1,2}( [0,T], \C^{n} ) \rightarrow L^{2}( [0,T], \C^{n} ), \quad v \mapsto -J_0 \p_t v - S v.
\end{equation}
It is well-known that the operator $L_S$ is a symmetric Fredholm operator of index zero. Moreover, its spectrum 
$$
\mathfrak{S}(L_S) = \left \{ \lambda \in \C \; : \; L_S- \lambda \text{Id is non-invertible} \right\}
$$
has the following properties:
\begin{itemize}
\item It consists precisely of the eigenvalues;
\item It is real and discrete;
\item The eigenvalues can accumulate only at $\infty$.
\end{itemize}
Note that if $S\equiv 0$, then the operator $L_0$ has $\mathfrak{S}(L_0) = \left\{ \pi k \; : \; k \in\Z\right\}$ and each eigenvalue $\pi k$ has multiplicity $n$. An important fact is that by  Kato's perturbation theory \cite{Kato1}  the eigenvalues of $L_S$ vary continuously under perturbations of $S$.  It follows that there exist continuous mappings  
$$ 
\lambda_k  : \mathcal{S} \rightarrow \mathfrak{S} := \left \{ ( S, \lambda ) \in \mathcal{S} \times \R : \lambda \in \mathfrak{S}(L_S) \right\},  \quad k \in \Z
$$
having the following properties:
\begin{itemize}
\item (Continuity) For every $k\in \Z$, $S\mapsto \lambda_k(S)$ is a continuous map;
\item (Monotonicity) For every $k \in \Z$, $\lambda_k (S) \leq \lambda_{k+1}(S)$;
\item  (Eigenvalue) For each  $S$, $\mathfrak{S}(L_S)=\left\{ \lambda_k(S) :  k \in \Z\right\}$;
\item (Multiplicity) If $\lambda \in \mathfrak{S}(L_S)$, then $\# \left\{ k \in \Z  :  \lambda_k(S) =\lambda \right\} = \text{multiplicity of $\lambda$}$;
\item (Normalization) For $k\in \left \{ 1,2,\dots,n\right\}$, $\lambda_k(0) =0$.
\end{itemize}

\noindent Associated to $S \in \mathcal{S}$ is the path of symplectic matrices $\Phi_S :[0,T]\rightarrow \text{Sp}(2n)$, i.e., the solution of the Cauchy problem $\p_t \Phi_S(t) = J_0 S(t)\Phi_S(t)$, $\Phi_S(0) = \id$. The following theorem relates the Robbin-Salamon index of the path $\Phi_S$ with respect to the real Lagrangian $\R^n$ to the spectrum of~$L_S$ in the non-degenerate case. For the proof, we refer the reader to \cite[Theorem~11.2.3]{book}.

\begin{theorem} \label{thm:spec1} Assume that $\Phi_S(T) \R^{n}  \cap \R^{n}  = \left \{ 0 \right \}.$ Then we have $$
\mu_{ {\rm RS}}( \Phi_S \R^{n} , \R^{n} ) = \max \left\{ k \in \Z : \lambda_k(S) <0 \right\} - \frac{n}{2}.
$$
\end{theorem}

Similarly, one can consider the Hilbert space
$$
W_{i \R^{n} }^{1,2}( [0,T], \C^{n} ) = \left \{  v \in W^{1,2}([0,T],\C^{n}  ) \; : \; v(0), v(T) \in i \R^{n}          \right \}
$$ 
and  the bounded linear operator $\widetilde{L}_S : W_{i \R^{n} }^{1,2}( [0,T], \C^{n} ) \rightarrow L^{2}( [0,T], \C^{n} )$ defined by the same formula as $L_S$. This operator satisfies the same properties as $L_S$ and hence, there exist continuous mappings  $\widetilde{\lambda}_k$, $k \in \Z$ having the same properties as $\lambda_k$.  Consequently, if $\Phi_S(T) i \R^{n}  \cap i \R^{n}  = \left \{ 0 \right \}$, then we have
\begin{equation}
\mu_{ {\rm RS}}( \Phi_S i\R^n, i\R^n) = \max \left\{ k \in \Z : \widetilde{\lambda}_k(S) <0 \right\} - \frac{n}{2}.
\end{equation}

\subsubsection{Index of  Reeb chords and iterations}\label{sec:indexiteration}

 Let $c:[0,T]\rightarrow\Sigma$ be a Reeb chord and $c^2$ the associated symmetric periodic Reeb orbit. We define further iterations of $c$  by
$$
c^{2\ell-1}(t) := \begin{cases} c(t) & \; t \in [0,T], \\ \rho \circ c( 2T -t ) & \; t \in [T, 2T], \\ c( t - 2T) & \; t \in [2T, 3T], \\ \rho \circ c( 4T - t ) & \; t \in [3T,4T], \\ \quad \quad \vdots \\ c(  t - (2\ell-2)T   ) & \; t \in [(2\ell-2)T, (2\ell-1)T],       \end{cases}
$$
and
$$
c^{2\ell}(t) := \begin{cases} c^{2 \ell-1}(t) & \; t \in [0,(2\ell-1)T], \\   \rho \circ c( 2\ell T - t  ) & \; t \in [(2\ell-1)T, 2\ell T]   .    \end{cases}
$$

 Choose a  $d\alpha|_{\xi}$-compatible almost complex structure $J$ on $\xi$ which is $\rho$-anti-invariant:
$$
\rho^*J:= (D\rho|_{\xi})^{-1} \circ J \circ D\rho|_{\xi} = -J.
$$
Given a contractible Reeb chord $c:[0,T]\rightarrow \Sigma$, let $u :D \rightarrow \Sigma$ be a capping disk for  the symmetric periodic Reeb orbit $c^2$ such that $u\circ I=\rho\circ u$ where 
$$
I= \begin{pmatrix} \id_{\R^{n-1}} & 0 \\ 0 & -\id_{\R^{n-1}} \end{pmatrix}.
$$
We choose a symmetric unitary trivialization 
$$
\widetilde{\Psi} : D \times \C^{n-1} \rightarrow u^*\xi 
$$
satisfying
\begin{itemize}
\item $\Psi(z)^* d\alpha|_{\xi} = \omega_0$;
\item $J(u(z))\Psi(z) = \Psi(z)J_0$; and
\item $D\rho|_{\xi_{u(z)}} \circ \Psi(z) = \Psi(\overline{z}) \circ I$, 
\end{itemize}
where $\Psi(z):= \widetilde{\Psi}(z, \cdot)$,
$$
J_0= \begin{pmatrix} 0 & -\id_{\R^{n-1}} \\ \id_{\R^{n-1}} & 0\end{pmatrix}.
$$
In particular, the Legendrian $\mathcal{L}$ becomes $L_0 =\R^{n-1} \times \left\{0\right\}$ which is a Lagrangian subspace of $(\C^{n-1},\omega_0)$. We fix $L_1 :=  \left\{0\right\} \times \R^{n-1} = J_0  L_0$ which is a Lagrangian complement of $L_0$ in~$\C^{n-1}$.   For the existence of such a trivialization, we refer to \cite[Lemma~3.10]{UrsJungsoo}. We write $
\Psi(t):= \Psi(e^{2\pi i  t/2T})$. Associated to the symmetric periodic Reeb orbit $c^2$ is the path of symplectic matrices ${\Phi}^2 : [0,2T] \rightarrow \text{Sp}(2n-2)$ 
$$
{\Phi}^2(t):= \Psi(t)^{-1} \circ D\phi_R^t (c^2(0) )|_{\xi} \circ \Psi(0).
$$
Abbreviating $\Phi = \Phi^2|_{[0,T]}$, one represents the  paths of symplectic matrices corresponding to the iterations:
\begin{equation}\label{eqiterationodd}
\Phi^{2\ell-1}(t)= \begin{cases} \Phi(t) & t \in [0,T], \\ I \Phi(2T-t) \Phi(T)^{-1} I \Phi(T) & t \in [T,2T], \\ \quad \quad \vdots \\I\Phi(2(\ell-1)T -t)I \Phi(2T)^{\ell-1} & t \in [(2\ell-3)T, (2\ell-2)T], \\ \Phi ( t-2(\ell-1)T ) \Phi(2T)^{\ell-1} & t \in [(2\ell-2)T, (2\ell-1)T],  \end{cases}
\end{equation}
and 
\begin{equation}\label{eqiterationeven}
\Phi^{2\ell}(t)= \begin{cases} \Phi^{2\ell-1}(t)   & t \in [0,(2\ell-1)T], \\   I \Phi(2\ell T-t) I \Phi(2T)^\ell & t \in [(2\ell-1)T, 2\ell T],\end{cases}
\end{equation}
where $\Phi(2T) =I \Phi(T)^{-1} I \Phi(T)$. 

The Maslov indices of a non-degenerate Reeb chord $c$ are defined by
$$
\mu_I  (c) := \mu_{\RS}(\Phi  L_0, L_0) , \quad \mu_{-I}(c) := \mu_{\RS}(\Phi  L_1, L_1).\footnote{We borrow the notations $\mu_I, \mu_{-I}$ from \cite[Section~3]{UrsJungsoo}.}
$$
Recall that $c$ is non-degenerate (as a chord) if and only if    $\Phi (  T ) L_0 \cap L_0 =\left\{ 0 \right \}$. If we write $\Phi$ as a block matrix
\begin{equation*}\label{blocksdkumatrix}
\Phi = \begin{pmatrix} X & Y \\ Z & W \end{pmatrix}
\end{equation*}
with respect to the Lagrangian splitting $\C^{n-1} = L_0 \oplus L_1$, then this condition is equivalent to the block  $Z(T)$ being invertible.  Similarly, $\Phi(T)L_1 \cap L_1 = \left \{ 0 \right \}$ is equivalent to the block $Y(T)$ being  invertible.

 We have defined the Maslov indices $\mu_I, \mu_{-I}$ of non-degenerate Reeb chords via the Robbin-Salamon indices. We now extend the definitions to degenerate Reeb chords. If we extend them by the same formula, i.e., via the Robbin-Salamon index, then the indices are neither lower nor upper semi-continuous. In other words, under a small perturbation, the indices of a degenerate Reeb chord may jump up or down.   Following  \cite{HWZ} we  extend the definitions  to degenerate Reeb chords via the spectral flow.

\begin{definition} Let $c$ be a contractible Reeb chord on a real contact manifold $(\Sigma^{2n-1}, \alpha, \rho)$. Abbreviate by $\Phi(t) \in \text{Sp}(2n-2)$, $t \in [0,T]$ the corresponding path of symplectic matrices. Then its Maslov indices $\mu_I, \mu_{-I}$ are defined to be
\begin{align*}
\mu_I(c) &:= \max \left\{ k \in \Z : \lambda_k(S) <0 \right\} - \frac{n-1}{2} ,\\
{\mu}_{-I}(c)&:=  \max \left\{ k \in \Z : \widetilde{\lambda}_k(S) <0 \right\} - \frac{n-1}{2},
\end{align*}
where $\lambda_k , \widetilde{\lambda}_k $ are the continuous mappings defined in Section~\ref{secspectral} and $S(t):= -J_0 \dot{\Phi}(t) \Phi(t)^{-1}$, $t \in [0,T]$ is the path of $(2n-2) \times (2n-2)$ symmetric matrices. 
\end{definition}
 
  It is worth mentioning that with these definitions the indices become lower semi-continuous, i.e., under a small perturbation, they may jump up, but not down.  Moreover, in view of Theorem~\ref{thm:spec1}, in the non-degenerate case we have
$$
\mu_I(c) = \mu(c),
$$
where $\mu(c)$ is defined as in Definition \ref{defmaslov}.

In the following we say that a contact form $\alpha$ is 
\begin{itemize}
\item    {non-degenerate} if all periodic Reeb orbits are non-degenerate;
\item   \text{$L_i$-non-degenerate} if for each Reeb chord $c:[0,T] \rightarrow \Sigma$, the corresponding path of symplectic matrices $\Phi$ satisfies   $\Phi(T)L_i \cap L_i = \left \{ 0 \right \}$, $i=0,1$.
\end{itemize}

\begin{proposition}\label{prop:nondeg} {\rm(\cite[Proposition C]{no}, \cite[Propositions 3.6 and  3.7]{UrsJungsoo})}  Let $c: [0,T] \rightarrow \Sigma$ be a Reeb chord as above. Then the iteration $c^2$ is non-degenerate {\rm (\textit{as a periodic orbit})} if and only if $c$ is both $L_0$- and $L_1$-non-degenerate. In particular, in this case we have
$$
\mu_{\CZ}(c^2) = \mu _I(c) + \mu_{-I}(c),
$$
where $\mu_{\CZ}$ denotes the Conley-Zehnder index of a periodic orbit.
\end{proposition}

\subsection{Real dynamical convexity} \label{sec:realdyn}In this section, we introduce the notion of real dynamical convexity.

\begin{definition}\label{def:realdc} A real contact manifold $(\Sigma^{2n-1}, \alpha, \rho)$  is said to be \textit{real dynamically convex}  if   the Maslov class $\mu_{\mathcal{L}} : \pi_2(\Sigma, \mathcal{L})\rightarrow\Z$, where $\mathcal{L}=\text{Fix}(\rho)$, vanishes and  every contractible Reeb chord $c$ satisfies  the lower bound
$$\mu_I (c), \; \mu_{-I} (c)\geq   \frac{n+1}{2}.$$
\end{definition}
In order to explain the reason for this terminology, let $\Sigma$ be a   strictly convex hypersurface in $\R^{2n} \cong\C^n$ equipped with   the standard contact form 
$$
\alpha_0 = \frac{1}{2} \sum_{j=1}^{n} q_j dp_j - p_j dq_j \bigg|_{\Sigma}.
$$
The corresponding Reeb vector field and the contact structure are denoted by $R=R_{\alpha_0}$ and $\xi = \ker (\alpha_0)$, respectively. We assume that $\Sigma$ is invariant under complex conjugation  
$${I}= \begin{pmatrix} \id_{\R^{n}} & 0 \\ 0 & -\id_{\R^{n}} \end{pmatrix}$$
of $(\C^n, \omega_0)$, where $\omega_0 = \sum_{j=1}^n dq_j\wedge dp_j$.

\begin{theorem}\label{thm:chorddynami} Let the triple  $(\Sigma, \alpha_0, {I})$ be described as above.  Then it  is real dynamically convex. 
\end{theorem}
\begin{proof} We imitate the strategy of the proofs  of \cite[Theorem~3.4]{HWZ} and \cite[Theorem~12.2.1]{book}. Without loss of generality, we may assume that $\Sigma$ encircles the origin.  Choose a Hamiltonian $H:\C^n \rightarrow\R$ satisfying
\begin{itemize}
\item (Hypersurface) $H^{-1}(1)=\Sigma$ and $H(z)<1$ for $z\in \text{int}(K)$, where $K$ is a strictly convex and compact subset of $\C^n$ with $\p K=\Sigma$;
\item (Invariance) $H({I}z)=H(z)$, $z\in \C^n$;
\item (Homogeneity) $H(rz) =r^2H(z)$, $z\in\C^n\setminus \left \{ 0 \right \}$, $r \geq 0$.
\end{itemize}
Let $X_H$ denote the Hamiltonian vector field associated to $H$ defined implicitly by the relation $\iota_{X_H}\omega_0=-dH$. Using homogeneity of $H$, one see that 
$$
R(z) =X_H(z), \quad z\in \Sigma.
$$
In particular, for a Reeb chord $c:[0,T] \rightarrow \Sigma$, it holds that $\phi_H^t(c(0))=c(t)$, where $\phi_H^t$ denotes the Hamiltonian flow of $H$. Recall that the path of symplectic matrices ${\Phi}(t) := D \phi_H^t( c(0)) \in \text{Sp}(2n)$, $t \in [0,T]$, satisfies
$$
 \dot{{\Phi}}(t) = J_0 \mathcal{H}(c(t)) {\Phi}(t),
$$
where $\mathcal{H}$ is the Hessian of $H$ which is positive-definite due to the strict convexity of   the hypersurface~$\Sigma$.  By considering $\tfrac{1}{T}H$, from now on we may assume that $T=1$.  

Let $u: D \rightarrow \Sigma$ be a capping disk of a symmetric periodic Reeb orbit $c^2$. Choose a symmetric unitary trivialization $\Psi_{\xi} :  u^* \xi \rightarrow D \times \C^{n-1}$.  Abbreviating $\eta_z = \text{span}\left \{ z, X_H(z) \right \}$, $z \in \C^n$, we obtain the symplectic splitting $u^* \C^n = u^* \eta \oplus u^* \xi$ of the vector bundle $u^* \C^n \rightarrow D$.  Note that  $u^*\eta$ admits a canonical unitary trivialization  $\Psi_{\eta} : u^* \eta \rightarrow D\times \C$ given by
$$
 u (z) \mapsto 1 \quad \text{and} \quad X_H(u(z)) \mapsto i.
$$
This gives rise to  the unitary trivialization 
$$
\Psi:= \Psi_{\eta} \oplus \Psi_{\xi} : u^* \C^n \rightarrow D \times \C^n.
$$
In this trivialization, the path ${\Phi} : [0,1] \rightarrow \text{Sp}(2n)$  is given by the path of $2n\times2n$ symplectic matrices  $\widehat{\Phi} (t)=   \widehat{\Phi}_{\eta} \oplus \widehat{\Phi}_{\xi}(t)$, where    $\widehat{\Phi}_{\eta} := \Psi_{\eta}(c(t)) \circ {\Phi}(t)|_ {\eta}  \circ \Psi_{\eta}( c(0))^{-1} = \id_{\C}.$  The catenation property of the Robbin-Salamon index then implies that $\mu_{\RS}(\widehat{\Phi} \Lambda_0, \Lambda_0) = \mu_{\RS}(\widehat{\Phi}_{\xi} L_0, L_0) ,$ where $\Lambda_0 = \R^n \times \left\{0 \right\} $. On the other hand,  since $D$ is contractible, the trivialization $\Psi : u^* \C^n=D \times \C^n \rightarrow D \times \C^n$  is homotopic as a bundle map to the identity map of $D\times \C^n$. It thus follows from homotopy invariance of the Robbin-Salamon index that $\mu_{\RS}(\widehat{\Phi}\Lambda_0, \Lambda_0)= \mu_{\RS}(\Phi \Lambda_0, \Lambda_0) $ and similarly  $\mu_{\RS}(\widehat{\Phi}\Lambda_1, \Lambda_1)=\mu_{\RS}(\Phi \Lambda_1,\Lambda_1),$ 
where $\Lambda_1=\{0\}\times \R^n$.

\medskip

\noindent \textbf{Claim.} Consider the   path of Lagrangian subspaces $\text{Gr}(\Phi)$ of $( \C^n \times \C^n, \Omega= (-\omega) \times \omega)$. For each $t_0 \in [0,1]$, the quadratic form $Q=Q_{\partial_t \text{Gr}(\Phi)|_{t=t_0}}$ is given by 
$$
Q ( x, \Phi(t_0)x) = \left< \Phi(t_0)x, \mathcal{H}(c(t_0)) \Phi(t_0)x\right>.
$$
\emph{Proof of the claim.} A typical Lagrangian complement to the graph Lagrangian $\text{Gr}(\Phi(t_0))$ is given by $\text{Gr}(-\Phi(t_0))$. By definition, the quadratic form $Q$ is given by
$$
Q( x, \Phi(t_0)x) = \Omega\left( (x, \Phi(t_0)x), ( \dot{y}(t_0), - \Phi(t_0)\dot{y}(t_0) ) \right),
$$
where $y(t)$ is obtained as follows: for $t \in (t_0 - \epsilon, t_0 + \epsilon)$, we find $y(t) \in \C^n$ satisfying that
$$
( x , \Phi(t_0) x) + ( y(t), -\Phi(t_0) y(t) ) \in \text{Gr}(\Phi(t)).
$$
It follows that
$$
\Phi(t_0) (x  - y(t) ) = \Phi(t)( x + y(t))
$$
and hence we obtain that
$$
\dot{y}(t_0) = - \frac{1}{2} \Phi(t_0)^{-1} \dot{\Phi}(t_0)x,
$$
where we have used the fact that $y(t_0)=0$. We then conclude that
\begin{align*}
Q(x, \Phi(t_0)x) &=  - \omega( x, \dot{y}(t_0) ) - \omega( \Phi(t_0)x, \Phi(t_0)\dot{y}(t_0) )\\
&= -2 \omega(x, \dot{y}(t_0)) \\
&=  \omega(x, \Phi(t_0)^{-1} \dot{\Phi}(t_0) x) \\
&= \left< \Phi(t_0)x, \mathcal{H} \Phi(t_0)x \right>
\end{align*}
from which the claim follows.

\bigskip

The claim implies that the crossing form $C(\text{Gr}(\Phi), \Lambda_j \times \Lambda_j, t)= Q|_{\text{Gr}(\Phi(t))\cap (\Lambda_j \times \Lambda_j)}$ is positive-definite, provided that $\text{Gr}(\Phi(t)) \cap (\Lambda_j \times \Lambda_j) \neq \left \{ 0 \right \}$, $j=0,1$. Hence, we obtain that
\begin{equation*}
\mu_{\RS}(\text{Gr}(\Phi), \Lambda_j \times \Lambda_j)= \frac{1}{2}n + \sum_{ 0<t<1} \dim \left( \Phi (t) \Lambda_j \cap \Lambda_j\right) + \frac{1}{2}\dim \left( \Phi(1) \Lambda_j \cap \Lambda_j \right), \; j=0,1.
\end{equation*}
In view of  \cite[Theorem~3.2]{RSindex}, we have
$$
\mu_{\RS}(\Phi \Lambda_0, \Lambda_0) =\mu_{\RS}( \text{Gr}(\Phi), \Lambda_0 \times \Lambda_0 ), \quad \mu_{\RS}(\Phi \Lambda_1, \Lambda_1) =\mu_{\RS}( \text{Gr}(\Phi), \Lambda_1 \times \Lambda_1 ).
$$

We  compute that
\begin{equation*}\label{eq:linearkernel}
D{\phi}_H^1(c(0))R(c(0)) = D{\phi}_H^1(c(0))X_H (c(0)) = X_H( {\phi}_H^1(c(0)) ) = R(c(1))
\end{equation*}
and
\begin{equation*}\label{eq:linearkernel2}
D{\phi}_H^1(c(0))c(0) = {\phi}_H^1(c(0)) = c(1),
\end{equation*}
where in the first equality we used $D\phi_H^t(z)z = {\phi}_H^t(z)$ for all $z$ which follows from  homogeneity of~$H$.  These identities show that 
$$
\dim  \left( \Phi(1) \Lambda_j \cap \Lambda_j\right) =  \dim  \left( \widehat{\Phi}_{\xi}(1) L_j \cap L_j\right) +1.
$$
Combining the discussion so far we obtain that
\begin{equation}\label{eq:dimincex}
\mu_{\text{RS}}(\widehat{\Phi}_{\xi} L_j, L_j ) \geq \frac{n+1}{2} + \frac{1}{2} \dim  \left( \widehat{\Phi}_{\xi}(1) L_j \cap L_j\right), \quad j=0,1.
\end{equation}

\medskip

\textit{Case 1.} The Reeb chord $c$ is non-degenerate.

\noindent 
In view of \eqref{eq:dimincex} it follows from the definitions of the Maslov indices that 
\begin{eqnarray*}\label{ea:nondechord}
 \mu_I (c) &\geq& \frac{n+1}{2}  +  \frac{1}{2}\dim  \left( \widehat{\Phi}_{\xi}(1) L_0 \cap L_0\right),\\
\mu_{-I} (c) &\geq& \frac{n+1}{2}  +  \frac{1}{2} \dim  \left( \widehat{\Phi}_{\xi}(1) L_1 \cap L_1\right).
\end{eqnarray*}
Since $\dim(\widehat{\Phi}_{\xi}(1) L_j \cap L_j)\ge 0$ for $j=0,1$, these give rise to the desired lower bounds.

\medskip

\textit{Case 2.} The Reeb chord $c$ is  degenerate.

\noindent   For $\epsilon>0$ small enough, we choose  a path of symplectic matrices $\Phi_{\epsilon}^j : [0,1+\epsilon] \rightarrow \text{Sp}(2n-2)$ such that $\Phi_{\epsilon}^j(t) = \widehat{\Phi}_{\xi}(t)$ for $ t \in [0,1]$ and there exist no further crossings with respect to $L_j$ in $t \in (1, 1+\epsilon]$, for $j=0,1$.  In particular, the path $\Phi_{\epsilon}^j$ is $L_j$-non-degenerate for $j=0,1$. We   find that
\begin{align*}
\mu_{\text{RS} }(\Phi_{\epsilon}^j L_j, L_j) &= \mu_{\text{RS}}(\widehat{\Phi}_{\xi} L_j , L_j) + \frac{1}{2} \text{sgn} C( \widehat{\Phi}_{\xi} L_j, L_j, 1 ) \\
&= \mu_{\text{RS}}(\widehat{\Phi}_{\xi} L_j , L_j) + \frac{1}{2} \dim \left(\widehat{\Phi}_{\xi} (1 )L_j \cap L_j \right) \\
&\geq \frac{n+1}{2} +   \dim \left(\widehat{\Phi}_{\xi}(1)L_j \cap L_j \right) .
\end{align*}
Abbreviate by $S  (t)$, $t \in [0,1]$ the paths of symmetric matrices associated with $\widehat{\Phi}_{\xi} (t)$. Note that 
$$
\dim \ker L_S = \dim \left( \widehat{\Phi}_{\xi} (1) L_0 \cap L_0\right), \quad \dim \ker \widetilde{L}_S = \dim \left(\widehat{\Phi}_{\xi} (1) L_1 \cap L_1 \right), 
$$
where $L_S, \widetilde{L}_S$ are the bounded linear operators defined as in Section~\ref{secspectral}. It then follows from the continuity of eigenvalues we have
$$
\mu_I(c)   \geq \lim_{\epsilon \rightarrow 0} \mu_{\text{RS}}( \Phi_{\epsilon}^0 L_0, L_0) -  \dim\left( \widehat{\Phi}_{\xi} (1) L_0 \cap L_0 \right) \geq \frac{n+1}{2} 
$$
and similarly
$$
\mu_{-I}(c) \geq \frac{n+1}{2}.
$$
This completes the proof of the theorem.
 \end{proof}


\subsection{Index increasing property} \label{sec:indexincreasing}

We prove the following increasing property of the Robbin-Salamon index which is one of the  main tools in counting symmetric periodic Reeb orbits via equivariant wrapped Floer homology.
\begin{theorem} \label{thm:indexincreasingproperty} Assume that  a real contact manifold $(\Sigma^{2n-1}, \alpha, \rho)$ is  non-degenerate. Let $c$ be a Reeb chord. If $\mu_I(c) \geq k$ for some $k \in \frac{1}{2}\Z$, then $\mu_I(c^{\ell+1}) \geq \mu_I(c^{\ell}) + k - \frac{n-1}{2}$ for every $\ell\in \N$. In particular, if $\alpha$ is real dynamically convex, then there holds $\mu _I(c^{\ell+1}) > \mu  _I(c^{\ell})$ for every $\ell\in \N$. 
\end{theorem}

 Let $c^2 :[0,2T] \rightarrow \Sigma$ be a non-degenerate symmetric periodic Reeb orbit. Denote by $\Phi:[0,T]\rightarrow \text{Sp}(2n-2)$ the path of symplectic matrices corresponding to the Reeb chord $c$. By abuse of notation, we write
 $\Phi(2kT)=\Phi^{2k}(2kT)$ for $k\in \N$. 
 
 With respect to the Lagrangian splitting $\C^{n-1} = L_0 \oplus L_1$, we write
$$
\Phi(T) = \begin{pmatrix} X & Y \\ Z & W \end{pmatrix}, \quad \Phi(2T)= \begin{pmatrix} A & B \\ C & D \end{pmatrix}.
$$
It follows from \cite[Lemma~3.1]{UrsOttoHormander} that
$$
\Phi(2T)^k = \begin{pmatrix} T_k(A) & U_{k-1}(A)B \\ CU_{k-1}(A) & T_k(A^T) \end{pmatrix},
$$
where $T_k$ and $U_{k-1}$ are Chebyshev polynomials of the first and second kind, respectively. For example,
$
T_0(x)=1, T_1(x)=x, U_0(x)=x$, and $U_1(x)=2x$. The following lemma will be used in the proof of Theorem~\ref{thm:indexincreasingproperty}.
\begin{lemma} The blocks $X,Y,Z,W,T_{k}(A)$ and $CU_{k-1}(A)$ are invertible for all $k\ge 1$. 
\end{lemma}
\begin{proof} In \cite[Lemma~3.2]{UrsOttoHormander} it is shown that the block $CU_{k-1}(A)$ is invertible for all $k \geq 1$. In order to show that $T_k(A)$ is also invertible, we consider the $k$th iteration of $c^2$, i.e., $c^{2k} : [0,2kT] \rightarrow \Sigma$. With respect to the Lagrangian splitting above, we write
$$
\Phi(2kT) = \begin{pmatrix} K & L \\ M & N \end{pmatrix} 
$$
and then we have
$$
\Phi(2kT)^2 = \begin{pmatrix} T_2(K) & U_1(K)L \\ M U_1(K) & T_2(K^T) \end{pmatrix}=\begin{pmatrix} T_2(K) & 2KL \\ 2MK & T_2(K^T) \end{pmatrix}.
$$
By \cite[Lemma~3.2]{UrsOttoHormander}, both $M$ and $MK$ are invertible and hence so is $K$. However, in view of the relation $\Phi(2kT) = \Phi(2T)^k$, we have $K = T_k(A)$. This proves the claim. 

Because of the non-degeneracy, $Y$ and $Z$ are invertible. In view of the relation $\Phi(2T) = I \Phi(T)^{-1} I \Phi(T)$, we see that $C= 2X^{T} Z$ and $B=2 W^{T} Y.$  The invertibility of $C$ and $Z$ implies that $X$ is invertible as well. Using the argument given in the proof of   \cite[Lemma~3.2]{UrsOttoHormander}, one can see that $B$ is invertible, implying the invertibility of $W$. This completes the proof. 
\end{proof}

\begin{proof}[Proof of Theorem~\ref{thm:indexincreasingproperty}]
	
  The catenation property of the Robbin-Salamon index implies that
$$
\mu_{\RS}(\Phi^{\ell+1}L_0, L_0) = \mu_{\RS}(\Phi^{\ell}L_0, L_0) + \mu_{\RS}(\Phi^{\ell+1}|_{[\ell T, (\ell+1)T]}L_0, L_0).
$$
We distinguish the following two cases.

\medskip

\textit{Case 1. $\ell=2k$ is even.}\\
In view of  the iteration formulas   \eqref{eqiterationodd}  and \eqref{eqiterationeven} we have
$$
\mu_{\RS}(\Phi^{\ell+1}|_{[\ell T, (\ell+1)T]}L_0, L_0)=\mu_{\RS}(\Phi^{2k+1}|_{[2kT, (2k+1)T]}L_0, L_0) = \mu_{\RS}(\Phi\Phi(2T)^kL_0, L_0).
$$
By the naturality property   of the Robbin-Salamon index we obtain 
$$
 \mu_{\RS}(\Phi\Phi(2T)^kL_0, L_0)  = - \mu_{\RS}(\Phi^{-1}L_0,\Phi(2T)^kL_0)  
$$
from which we see that
$$
\mu_{\RS}(\Phi^{2k+1}L_0, L_0) - \mu_{\RS}(\Phi^{2k}L_0, L_0) = - \mu_{\RS}(\Phi^{-1}L_0,\Phi(2T)^kL_0)  .
$$
Following \cite[Theorem~3.5]{RSindex} we obtain that
$$
\mu_{\RS}(\Phi^{-1} L_0,L_0) - \mu_{\RS}(\Phi^{-1}L_0,\Phi(2T)^kL_0)     =    s( \Phi(2T)^kL_0 , L_0 ; L_0, \Phi(T)^{-1}L_0),
$$
where $s$ denotes the H\"ormander index \cite[Eq.~(2.10)]{Duister}, \cite[Section~3.3]{Hormander}.
Choose $A_0 = CU_{k-1}(A)T_k(A)^{-1}$, $A_1=B_0 =O$ and $B_1 =  -Z^T (W^T)^{-1}  $. The same theorem then implies that
\begin{eqnarray*}
s( \Phi(2T)^kL_0 , L_0 ; L_0, \Phi(T)^{-1}L_0) &=& \frac{1}{2} \text{sgn} (-Z^T (W^T)^{-1} )-  \frac{1}{2} \text{sgn} ( CU_{k-1}(A)T_k(A)^{-1} ) \\
&\quad&- \frac{1}{2} \text{sgn}(-Z^T (W^T)^{-1}  - CU_{k-1}(A)T_k(A)^{-1}) \\
&=& \frac{1}{2} \text{sgn}( - W^T(Z^T)^{-1} - T_k(A)U_{k-1}(A)^{-1}C^{-1}),
\end{eqnarray*}
where in the last identity we used \cite[Lemma~5.2]{RSindex}. For  dimensional reason we then have
$$
|s( \Phi(2T)^kL_0 , L_0 ; L_0, \Phi(T)^{-1}L_0) | \leq \frac{n-1}{2}
$$
from which we conclude that
\begin{align*}
- \mu_{\RS}(\Phi^{-1}  L_0,\Phi(2T)^kL_0) &=   - \mu_{\RS}(\Phi^{-1} L_0,L_0) + s( \Phi(2T)^kL_0 , L_0 ; L_0, \Phi(T)^{-1}L_0)  \\
&=     \mu_{\RS}(\Phi  L_0,L_0) + s( \Phi(2T)^kL_0 , L_0 ; L_0, \Phi(T)^{-1}L_0)  \\
&\geq   \mu_{\RS}(\Phi  L_0,L_0)  - \frac{n-1}{2}.
\end{align*}

\bigskip

\textit{Case 2. $\ell=2k-1$ is odd.}\\
In a similar way, using the naturality and reversal properties of the Robbin-Salamon index, we find
$$
\mu_{\RS}(\Phi^{2k}L_0, L_0) - \mu_{\RS}(\Phi^{2k-1}L_0, L_0) = -  \mu_{\RS}(\Phi ^{-1} L_0,I  \Phi(2T)^kL_0)
$$
and
$$
  \mu_{\RS}(\Phi ^{-1} L_0,L_0)  -  \mu_{\RS}(\Phi ^{-1} L_0, I \Phi(2T)^kL_0) = s(I   \Phi(2T)^kL_0 , L_0 ; L_0, \Phi(T) ^{-1}L_0) .
  $$
Choosing   $A_0 = -CU_{k-1}(A)T_k(A)^{-1}$, $A_1=B_0 =O$ and $B_1 =  -Z^T (W^T)^{-1}  $, the same argument as above shows that 
$$
 - \mu_{\RS}(\Phi^{-1} L_0,I  \Phi(2T)^kL_0) \geq   \mu_{\RS}(\Phi  L_0,L_0)  -\frac{n-1}{2} .
$$
 This proves the theorem.
\end{proof}

\begin{remark} In strictly convex case, a more general version of Theorem~\ref{thm:indexincreasingproperty} was proved by Liu and Zhang \cite[Lemma~7.3]{Liu3}. 
\end{remark}

\subsection{The common index jump theorem} 

In the following we  only consider paths of symplectic $(2n-2) \times (2n-2)$ matrices all of whose  iterations are   both $L_0$- and $L_1$-non-degenerate.   Abbreviate by $ {\mathcal{P}}  ^*(2n-2)$ the set of such paths. In view of  Proposition \ref{prop:nondeg}, one can define the Conley-Zehnder index  of any even iteration of $\Phi \in {\mathcal{P}}  ^*(2n-2)$.

For  $\Phi \in {\mathcal{P}}  ^*(2n-2)$,    the \textit{mean Robbin-Salamon index with respect to $L_0$}  is defined  by 
\begin{equation*}
\widehat{\mu}_{\RS}(\Phi L_0, L_0) = \lim_{\ell \rightarrow \infty} \frac{{\mu}_{\RS}(\Phi^{\ell} L_0, L_0)}{\ell}.
\end{equation*}
 
\begin{lemma} \label{lemma:inmnindexpoaties} Let $\Phi : [0,T] \rightarrow { \rm Sp}(2n-2)$ be an element in $\mathcal{P}^*(2n-2)$. Assume that 
$$
\mu_{\RS}(\Phi L_0, L_0 ), \; \mu_{\RS}(\Phi L_1, L_1 ) \geq \frac{n}{2}.
$$
Then the mean Robbin-Salamon index satisfies
$$
\widehat{\mu}_{\RS}(\Phi L_0, L_0) \geq \frac{1}{2}.
$$

\end{lemma} 
\begin{proof} By \cite[Corollary~5.1]{Liu3} we have
$$
\widehat{\mu}_{\RS}( \Phi L_0, L_0) = \frac{1}{2}\mu_{\CZ}(\Phi^2) + \sum_{\theta \in (0, 2 \pi)} \bigg( \frac{\theta}{2\pi} - \frac{1}{2} \bigg) S_{\Phi^2(2T)}^- (e^{i \theta}),
$$
where $S_{\Phi^2(2T)}^- (e^{i \theta})$ is the splitting number of  the matrix $\Phi^2(2T)$ defined in \cite{long2}. Recall that it satisfies
$$ 
0 \leq S_{\Phi^2(2T)}^- (e^{i \theta}) \leq \dim_{\C} \ker ( \Phi^2(2T)- e^{i \theta} \id ).
$$
We estimate that
\begin{align*}
\widehat{\mu}_{\RS}( \Phi L_0, L_0) &= \frac{1}{2}\mu_{\CZ}(\Phi^2) + \sum_{\theta \in (0, 2 \pi)} \bigg( \frac{\theta}{2\pi} - \frac{1}{2} \bigg) S_{\Phi^2(2T)}^- (e^{i \theta}) \\
& \geq \frac{n}{2}  + \sum_{\theta \in (0,   \pi)} \bigg( \frac{\theta}{2\pi} - \frac{1}{2} \bigg) S_{\Phi^2(2T)}^- (e^{i \theta}) \\
& \geq\frac{n}{2}  -  \frac{1}{2} \sum_{\theta \in (0,   \pi)}   S_{\Phi^2(2T)}^- (e^{i \theta}) \\
& \geq \frac{n}{2} -  \frac{n-1}{2} \\
&= \frac{1}{2},
\end{align*}
where the first resp. last inequality follows  from Proposition \ref{prop:nondeg} resp. the fact that along the half circle $\left \{ e^{i \theta} \in \C : \theta \in (0, \pi) \right \}$ the $(n-1)\times(n-1)$ matrix $\Phi^2(2T)$ has at most $(n-1)$ eigenvalues. This finishes the proof of the lemma.
\end{proof}

The next main tool is the so-called the \textit{common index jump theorem} which is proved in \cite[Theorem~1.5]{Liu3} for a different type of Maslov indices $i_{L_0}$ and $i_{L_1}$. In the non-degenerate case one can show that these indices are related to ours by
$$
i_{L_j} (\Phi) = \mu_{\RS}(\Phi L_j, L_j ) - \frac{n-1}{2} \in \Z, \quad j=0,1.
$$
In the following we only treat the non-degenerate case.

 \begin{theorem} {\rm (\cite[Theorem~1.5]{Liu3})} \label{theoremgeneralcommon}     Let $\Phi_j \in {\mathcal{P}}  ^*(2n-2)$, $j=1,2, \dots, k$, satisfy   $\widehat{\mu}_{\RS}(\Phi _j L_0, L_0)  >0$ for all $j$.     Then  there exist infinitely many vectors $(K, m_1, m_2, \dots, m_k ) \in \N^{k+1}$ such that  
\begin{align*}
    {\mu}_{\RS} ( \Phi_j ^{2m_j  -1 } L_0, L_0) &=K - {\mu}_{\RS}(\Phi_j L_1, L_1),\\
    {\mu}_{\RS} (  \Phi_j ^{2m_j+1} L_0, L_0) &=K + {\mu}_{\RS}(\Phi_j L_0, L_0).
\end{align*} 
\end{theorem}

\begin{proposition}\label{propremark} { \rm (\cite[Remark~5.1]{Liu3})} Let $\Phi_j \in {\mathcal{P}}  ^*(2n-2)$, $j=1,2, \dots, k$, be as in Theorem~\ref{theoremgeneralcommon}. Then on the infinite set of vectors $(K, m_1, \dots, m_k)$ in the assertion of that theorem, the first component $K$ is unbounded.
\end{proposition}
\begin{proof} Assume by contradiction that all vectors $(K, m_1, \dots, m_k)$  have bounded first components.  Since there exist infinitely many such vectors, without loss of generality we may assume that the $m_1$-component is unbounded. By  Theorem~\ref{theoremgeneralcommon} we see that the ${\mu}_{\RS}(\Phi_1^{2m_1+1} L_0, L_0)$'s are bounded, which implies by definition that $\widehat{\mu}_{\RS}(\Phi_1 L_0, L_0 ) =0$. Since we have assumed in particular $\widehat{\mu}_{\RS}(\Phi_1 L_0, L_0 ) >0$, this yields a contradiction.
\end{proof}

\subsection{Symmetric periodic Reeb orbits on displaceable hypersurfaces}  \label{sec:mainappl}
To describe the main results, we recall the setup. Let $(W^{2n}, \lambda, \rho)$ be a real Liouville domain such that $L=\Fix(\rho)$ and $\mathcal{L}=\p L$ are not empty. Assume that the Maslov classes $\mu_{L}: \pi_2 (W,L)\rightarrow \Z$ and $\mu_\mathcal{L}:\pi_2(\Sigma, \mathcal{L})\to \Z$ vanish so that the Maslov indices of Hamiltonian and Reeb chords are well-defined.

\begin{theorem}   \label{theorem:appmain} Assume that the real contact manifold $(\Sigma, \alpha, \rho)$ is  non-degenerate real dynamically convex, and    displaceable from $\widehat L$ in $\widehat W$. We further assume that the Lagrangian $L$ is orientable. Then there exist at least $n$ geometrically distinct, contractible and simple symmetric periodic Reeb orbits on $\Sigma$. 
\end{theorem}
\begin{proof} We follow the strategy of the proof of \cite[Theorem~1.1]{GuttKang}. Without loss of generality, we may assume that there are only finitely many contractible and  simple Reeb chords, say $c_1, \dots, c_k$. By a simple Reeb chord we mean a Reeb chord $c$ such that the corresponding symmetric periodic Reeb orbit $c^2$ is simple.  Since the contact form is real dynamically convex, in view of  Lemma~\ref{lemma:inmnindexpoaties}  we obtain that    $\widehat{\mu}(c_j) :=\widehat{\mu}_{\text{RS}}(\Phi_j L_0, L_0) 	 >0$ for $j=1,2, \dots, k$, where $\Phi_j$ are the corresponding paths of symplectic matrices. We can therefore apply Theorem~\ref{theoremgeneralcommon}  and Proposition~\ref{propremark}: We find a vector  $(K, m_1, m_2, \dots, m_k ) \in \N^{k+1}$ with $K$ as large as we like such that for each $j$ we have, by Proposition~\ref{prop:nondeg},
$$
 {\mu}_I ( c _j^{2m_j-1}) = K -    \mu_{-I}(c_j) , \quad {\mu}_I( c _j^{ 2m_j+1}) = K + {{\mu}}_I(c_j) .
$$
In view of   Theorem~\ref{thm:indexincreasingproperty} and real dynamical convexity, for each $j$ only $c_j^{2m_j}$  can have the index
$$
|c_j^{2m_j}|:=\mu_I(c_j^{2m_j})-\frac{n-1}{2}
$$ in the interval $[K-n+1,K]$.  Therefore,  if $K \geq n+2 $, then by the relation \eqref{eq: orbitexist} and Theorem~\ref{cor: antisympcomp}  $2m_j$-iteration of $c_j$ is a generator of the positive equivariant wrapped Floer homology from which we obtain $n \leq k$.  
This completes the proof of the theorem. 
\end{proof}

Recall   that a starshaped hypersurface  in $\R^{2n} \cong \C^n$, which is invariant under complex conjugation $(q,p) \mapsto (q,-p)$,   is displaceable from the real Lagrangian $\R^n  = \R^n \times \left \{ 0 \right \}    $, see Example~\ref{examp:hypersurfaceids}.  Theorem~\ref{thm:mianfirst} then follows immediately. Another   corollary of the proof of Theorem~\ref{theorem:appmain} is  the following result.

\begin{corollary}\label{cor: mainthm}   In addition to the assumptions of Theorem~\ref{theorem:appmain} we further assume that there exist precisely $n$ geometrically distinct simple symmetric periodic Reeb orbits on $\Sigma$. Then they are contractible and their  indices  are  all different. 
\end{corollary}

In what follows we assume that a real Liouville domain $(W^{2n},\lambda,\rho)$ admits  an exact symplectic involution $\iota$ under which $L$ is $\iota$-invariant and $\rho$ and $\iota$ commute.  In particular, $\rho\circ \iota$ is an additional exact anti-symplectic involution. By abuse of notation, we also use the symbols $\rho$ and $\iota$ for their restrictions to $\Sigma= \partial W.$ Note that $\rho^* \alpha = -\alpha$ and $\iota^*\alpha = \alpha$.   Given a Reeb chord $c:[0,T]\rightarrow \Sigma$, the corresponding symmetric periodic Reeb orbit $c^2$ is called \textit{doubly symmetric}  if $\iota( \text{im}(c^2) ) = \text{im}(c^2)$.

\begin{example} A typical example of a real Liouville domain which admits commuting $\rho$ and $\iota$ is   a compact starshaped hypersurface $\Sigma \subset \R^{2n}$ which is invariant under the complex conjugate and the antipodal map.
\end{example}

\begin{theorem}\label{thm:appli} Under the assumptions of Theorem~\ref{theorem:appmain}, we further assume that $\iota |_{L}$ acts freely on the boundary $\p L$.   Then there exist at least $n + \mathcal{N}(\Sigma)$ geometrically distinct, contractible and simple symmetric periodic Reeb orbits on $\Sigma$, where $2\mathcal{N}(\Sigma)$ is the number of geometrically distinct simple symmetric periodic  Reeb orbits   which are not doubly symmetric. 
\end{theorem}
\begin{proof} Following the strategy of the proof of   \cite[Theorem~1.2]{Liu3},    without loss of generality, we may assume that there are only finitely many contractible and  simple Reeb chords, say $c_1, \dots, c_k$, $\pm c_{k+1},\dots, \pm c_{k+l}$.   Here, $c_j^2$, $j=1,2,\dots, k$, are  doubly symmetric and for $k+1 \leq j \leq k+l$, the orbits $c_j^2$ are not doubly symmetric. By $\pm c_j$,  we mean the $\Z_2$-pair of $c_j$, $j=k+1, \dots, k+l$.

Set $r=k+l$.  As in the proof of Theorem~\ref{theorem:appmain}, we obtain $n \leq r$.  We then conclude that
$$
k + 2l = r + l \geq n + l = n + \mathcal{N}(\Sigma).
$$
This completes the proof of the theorem. \end{proof}

Parallel to Corollary~\ref{cor: mainthm}, we have the following result.

\begin{corollary} In addition to the assumptions of Theorem~\ref{thm:appli}, we further assume that there exist precisely $n$ geometrically distinct simple symmetric periodic Reeb orbits on $\Sigma$. Then they are contractible and  doubly symmetric periodic Reeb orbits and their   indices  are different. 
\end{corollary}

 \begin{remark}
As the proofs show, Theorems~\ref{theorem:appmain} and \ref{thm:appli} still hold if we assume that the $HW_*(L;W)$ vanishes, instead of the displaceablity condition. In particular, if the symplectic homology of the domain $W$ vanishes, then $HW_*(L;W)$ vanishes for any admissible Lagrangian $L$ in $W$, see Remark~\ref{rmk: vanishingSHandHW}.
 \end{remark}



\bibliographystyle{abbrv}
\bibliography{mybibfile}

\end{document}